 \documentclass[11pt,activeacute,amsfonts,reqno]{amsart}
\usepackage{todonotes}
\usepackage{a4wide}
\usepackage[english]{babel}
\usepackage{inputenc}
\usepackage{mathabx}
\usepackage{amsmath,amsthm,amsfonts,amssymb,amscd}
\usepackage{amssymb,amsfonts}
\usepackage[all,arc]{xy}
\usepackage{enumerate}
\usepackage{mathrsfs}
\usepackage{marginnote}
\usepackage{mathtools}
\usepackage{amsthm,thmtools,xcolor}

\usepackage[all]{xy}
\usepackage{tikz-cd}
\usetikzlibrary{cd}
\usepackage[square,sort,comma,numbers,nonamebreak]{natbib}

\theoremstyle{plain}
\newtheorem{thm}{Theorem}[section]
\newtheorem{lem}{Lemma}[section]
\newtheorem{cor}{Corollary}[section]
\newtheorem{prop}{Proposition}[section]

\newtheorem{lemx}{Lemma}

\theoremstyle{definition}
\newtheorem{defn}{Definition}[section]

\theoremstyle{remark}
\newtheorem{rem}{\textit{Remark}}[section]

\numberwithin{equation}{section}
\usepackage{color}

\usepackage[colorlinks = true,
linkcolor = blue,
urlcolor  = blue,
citecolor = blue,
anchorcolor = blue]{hyperref}
\usepackage{hyperref}
\numberwithin{equation}{section}

\DeclareMathOperator{\supp}{\mathrm{supp}}
\DeclareMathOperator{\sech}{\mathrm{sech}}
\newcommand\underrel[3][]{\mathrel{\mathop{#3}\limits_{%
			\ifx c#1\relax\mathclap{#2}\else#2\fi}}}
			
\title[KP]{Long time asymptotics of large data in the Kadomtsev-Petviashvili models}

\author[A. M\'endez]{Argenis J. Mendez$^{*}$}
\address{Center for mathematical modelling, Universidad de Chile, Santiago de Chile.}
\email{amendez@dim.uchile.cl}
\thanks{$^{*}$ This work was 
	partially supported by CMM Conicyt Proyecto Basal AFB170001.}

\author[C. Mu\~noz]{Claudio Mu\~noz$^{\lor}$}
\address{CNRS and Departamento de Ingenier\'{\i}a Matem\'atica and Centro
de Modelamiento Matem\'atico (AFB170001 and UMI 2807 CNRS), Universidad de Chile, Casilla
170 Correo 3, Santiago, Chile.}
\email{cmunoz@dim.uchile.cl}
\thanks{$^{\lor}$ C.M.'s work was funded in part by Chilean research grants FONDECYT 1191412, project France-Chile
ECOS-Sud C18E06, MathAmSud EEQUADD II and CMM ANID PIA AFB170001.}

\author[F. Poblete]{Felipe Poblete$^{\ocirc}$}
\address{Instituto de Ciencias F\'isicas y Matem\'aticas, Facultad de Ciencias, Universidad Austral de Chile, Valdivia, Chile.}
\email{felipe.poblete@uach.cl}
\thanks{$^{\ocirc}$ F.P.'s work is partially supported by ANID projects FONDECYT/Iniciaci\'on 11181263 and FONDECYT/Regular 1170466}

\author[J. C. Pozo]{Juan C. Pozo$^{\land}$}
\address{Departamento de Matem\'aticas, Facultad de Ciencias, Universidad de Chile.}
\email{jpozo@uchile.cl}
\thanks{$^{\land}$ J.C.P.'s work was partially supported by Chilean research grant FONDECYT 1181084.}


\date{\today}

\begin{document}

\begin{abstract}
We consider the Kadomtsev-Petviashvili equations posed on $\mathbb{R}^2$. For both equations, we provide sequential in time asymptotic descriptions of solutions, of arbitrarily large data, inside regions not containing lumps or line solitons, and under minimal regularity assumptions. The proof involves the introduction of two new virial identities adapted to the KP dynamics, showing decay in large regions of space, especially in the KP-I case, where no monotonicity property was previously known. Our results do not require the use of the integrability of KP and are adaptable to well-posed perturbations of KP. 
\end{abstract}

\maketitle

\section{Introduction and main results}

\subsection{Setting} Consider the \emph{Kadomtsev-Petviashvili  equations} (KP) posed in $\mathbb R^2$,
\begin{equation}\label{KP:Eq}
\partial_t u+\partial_{x}^{3}u+\kappa\partial_{x}^{-1}\partial_{y}^{2}u +u\partial_x u=0,
\end{equation}
where $u=u({\bf x},t)\in\mathbb{R}$, for $t\in\mathbb{R}$ and ${\bf x}=(x,y)\in\mathbb{R}^2$. The nonlocal operator $\partial_{x}^{-1}f$ is formally defined in the literature as
\begin{equation*}
(\partial_{x}^{-1}f)(x,y):=\int_{-\infty}^{x}f(s,y)\,\mathrm{d}s,
\end{equation*}
and $\kappa\in\{-1,1\}$. The KP equations were first introduced by Kadomtsev and Peviashvili in 1970 \cite{KadomtsevPetviashvili1970} for modeling long and weakly nonlinear waves propagating essentially along the $x$ direction, with a small dependence in the $y$ variable. It is a universal integrable two-dimensional generalizations of the well-known Korteweg–de Vries (KdV) equation, since many other integrable systems can be obtained as reductions \cite{KS}. The nonlocal term $\partial_{x}^{-1}\partial_{y}^{2}$ arises from weak transverse effects, with the sign of $\kappa$ allowing to take into account the surface tension. More precisely, when $\kappa=-1$ KP models capillary gravity waves, in the presence of \emph{strong} surface tension. In such a case KP is known as the KP-I equation. When $\kappa=1$, the equation \eqref{KP:Eq} is known as the KP-II model and it allows to study capillary gravity waves in the presence of \emph{small} surface tension. A rigorous derivation of both models from the symmetric $abcd$ Boussinesq system was obtained in \cite{Lannes1,Lannes2,LS}. We refer the reader to the monographs by Klein and Saut \cite{KS3,KS} for a complete mathematical and physical description of the KP model.

\medskip

The nonlocal term makes KP models hard from the mathematical point of view.  Despite their apparent similarity, KP-I and KP-II  differ significantly with respect to their underlying mathematical structure and the behavior of their solutions. For instance, from the point of view of well-posedness theory KP-II is much better understood than KP-I. Indeed, since the foundational work of Bourgain \cite{Bourgain1993} in 1993, one knows that KP-II is globally well-posed on $L^2(\mathbb{R}^2)$ (see also Ukai \cite{Ukai} and I\'orio-Nunes \cite{IN} for early results, and \cite{Takaoka-Tzvetkov-2001,Isaza-Mejia-2001,Hadac-Herr-Koch-2009} for results improving Bourgain's results). This result is proved via the contraction principle in $X^{s,b}$ spaces and the conservation of the $L^2$-norm
\begin{equation}\label{mass}
M[u](t):= \frac12\int_{\mathbb R^2} u^2(x,y,t)\,\mathrm{d}x\,\mathrm{d}y = \hbox{const.}
\end{equation}
Note that this conservation holds for both KP-I and KP-II. Meanwhile, the KP-I global theory took years to be solved. We emphasize the foundational global well-posedness work by Molinet, Saut and Tzvetkov \cite{Molinet-Saut-Tzvetkov-2002,Molinet-Saut-Tzvetkov-2002_b}, and the improved global well-posedness by C.~E. Kenig \cite{Kenig-2004}, which will be used in this paper several times (see also \cite{CKS}). A key breakthrough was obtained in 2008 by Ionescu, Kenig and Tataru \cite{Ionescu-Kenig-Tataru-2008} (see also \cite{Ionescu-Kenig-2007}), who proved that KP-I is globally well-posed in the natural energy space of the equation 
\begin{equation*}\label{E1}
E^1(\mathbb R^2) := \left\{ u \in L^2 (\mathbb R^2) ~ : ~ 
\|u \|_{L^2}+\|\partial_xu\|_{L^2} +\|\partial_x^{-1}\partial_yu \|_{L^2} <+\infty \right\},
\end{equation*}
in the sense that the flow map extends continuously from Schwartz data to $E^1(\mathbb R^2)$, see \cite{Ionescu-Kenig-Tataru-2008} for further details (it was known from \cite{Molinet-Saut-Tzvetkov-2002_b} that KP-I behaves badly with respect to perturbative methods). This space arises naturally from the conservation of energy ($\kappa=1$ KP-II, $\kappa=-1$ KP-I)
\begin{equation}\label{energy}
E[u](t):= \int_{\mathbb R^2} \left(\frac12 (\partial_x u)^2 -\frac12\kappa  (\partial_x^{-1}\partial_y u)^2 -\frac13 u^3 \right)(x,y,t)\,\mathrm{d}x\,\mathrm{d}y = \hbox{const.}
\end{equation}
It is worth to mention that the space $E^1(\mathbb R^2)$ will be essential for the validity of the functional tools that we shall use below. Additionally, very important in this paper will be the momentum conservation law, valid for solutions in $E^1(\mathbb R^2)$:
\begin{equation}\label{momentum}
P[u](t):= \frac12\int_{\mathbb R^2} \left( u\partial_x^{-1}\partial_y u \right) (x,y,t)\,\mathrm{d}x\,\mathrm{d}y = \hbox{const.}
\end{equation}
Also, for the purposes of this paper, we shall also need the globall well-posedness result for KP-I established by C.~E. Kenig \cite{Kenig-2004}, in the so-called second energy space
\begin{equation}\label{E2}
\begin{aligned}
E^2(\mathbb R^2) := &~{}  \left\{ u \in L^2 (\mathbb R^2)  :  \|u\|_{E^2}:= \|u \|_{L^2} +\|\partial_x^{-1}\partial_y u\|_{L^2} +\|\partial_x^2 u\|_{L^2} +\|\partial_x^{-2}\partial_y^2u \|_{L^2} <+\infty\right\}.
\end{aligned}
\end{equation}
This space is motivated by the conservation of the second energy \cite{ZS}
\[
\begin{aligned}
F[u](t):= &~{} \int_{\mathbb R^2} \left(\frac32 (\partial_x^2 u)^2 + 5(\partial_y u)^2 +\frac56 (\partial_x^{-2}\partial_y^2 u)^2 -\frac56 u^2 \partial_x^{-2}\partial_y^2 u \right. \\
&~{} \qquad\qquad  \left. -\frac56 u (\partial_x^{-1}\partial_y u)^2 +\frac54 u^2\partial_x^2 u   + \frac5{24} u^4 \right)(x,y,t)\,\mathrm{d}x\,\mathrm{d}y = \hbox{const.,}
\end{aligned}
\]
a property that has been rigorously proved in \cite{Molinet-Saut-Tzvetkov-2002}. This second energy is deeply related to the integrability of KP. From \cite{Kenig-2004}, we also know that $\|\partial_y u\|_{L^2} \lesssim\|u\|_{E^2}$. Finally, the Cauchy problem for the fractional version of KP was addressed in \cite{LPS}, see also references therein for more details on the low dispersion problem for KP. 

\medskip

Once a suitable well-posedness theory is available, one may wonder about the long time behavior of globally defined solutions. Except by some particular cases (see below the case of lumps and line solitons, their orbital stability, and the case of small data scattering solutions), no results about large data behavior in KP models starting from Cauchy data are available \cite{KS}.

\subsection{New results} In this paper we provide the first arbitrarily large data, long time asymptotic description of global finite energy solutions to the KP equations \eqref{KP:Eq}, with the minimal data assumptions required to run the \emph{global} well-posedness theory with uniform in time bounds. Basically, we only assume data in $L^2(\mathbb R^2)$ in the KP-II case, and data in $E^1(\mathbb R^2)$ and $E^2(\mathbb R^2)$ in the KP-I case. In particular, for both KP-I and KP-II we describe via a sequence of times the final dynamical behavior of \emph{any global solution} in regions of the plane growing unbounded in time (namely, containing any compact region in $\mathbb R^2$) and not containing the region where \emph{the lumps are present}. Recall that the lump, the soliton-like solution to KP-I with algebraic decay in space, does not decay in time. It was very recently proved to be orbitally stable in the space $E^1(\mathbb R^2)$ in another breakthrough by Liu and Wei \cite{LiuWei}. Consequently, and in view of the results that we will describe below, the asymptotic stability of this object seems a fundamental open problem in the field. 

\smallskip

Our results read as follows. Let $t\gg 1$. Let $\Omega_1(t)$ denote the following rectangular box 
\begin{equation}\label{Omega}
\Omega_1(t)= \left\{ (x,y)\in\mathbb R^2 ~: ~  |x-\ell_1 t^{m_1}| \leq t^{b}, \; |y-\ell_2 t^{m_2}|\leq t^{br}\right\},
\end{equation}
with $\ell_1,\ell_2\in\mathbb R$,
\begin{equation}\label{indices}
\begin{aligned}
&\frac53<r<3, \quad 0<b <\frac2{3+r}, \\
& 0\leq m_1<1-\frac{b}2(r+1), \quad \hbox{and}\quad  0\leq m_2 < 1-\frac{b}2(3-r).
\end{aligned}
\end{equation}
Let also $\Omega_2(t)$ be the union of cylinders 
\begin{equation}\label{Omega0}
\Omega_2(t):= \left\{ (x,y)\in\mathbb R^2 ~: ~  |x|\sim t^{p}\log^{1+\epsilon }t \quad  \hbox{ or } \quad  |y|\sim t^{p}\log^{1+\epsilon }t \right\},
\end{equation}
for any fixed $p\geq 1$ and $\epsilon>0.$ See Fig. \ref{fig:1} for details.

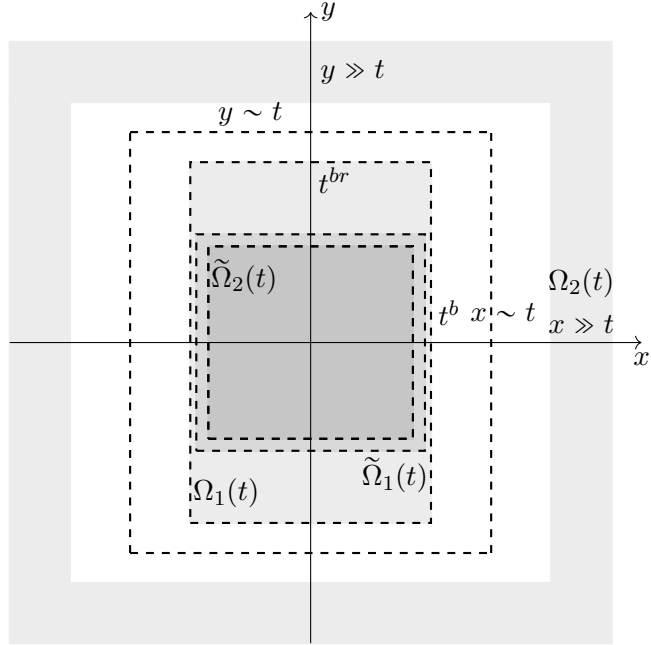
\begin{figure}[h!]
\begin{center}
\begin{tikzpicture}[scale=0.8]
\filldraw[thick, color=lightgray!30] (0,-1) -- (4,-1) -- (4,5) -- (0,5) -- (0,-1);
\draw[thick, dashed] (0,-1) -- (4,-1) -- (4,5) -- (0,5) -- (0,-1);
\draw[thick,dashed] (5,-1.5) -- (5,5.5);
\draw[thick,dashed] (-1,5.5)--(5,5.5);
\draw[thick,dashed] (-1,-1.5)--(-1,5.5);
\draw[thick,dashed] (-1,-1.5)--(5,-1.5);
\filldraw[thick, color=lightgray!60] (0.1,0.2) -- (3.9,0.2) -- (3.9,3.8) -- (0.1,3.8) -- (0.1,0.2);
\draw[thick, dashed] (0.1,0.2) -- (3.9,0.2) -- (3.9,3.8) -- (0.1,3.8) -- (0.1,0.2);
\filldraw[thick, color=lightgray!90] (0.3,0.4) -- (3.7,0.4) -- (3.7,3.6) -- (0.3,3.6) -- (0.3,0.4);
\draw[thick, dashed] (0.3,0.4) -- (3.7,0.4) -- (3.7,3.6) -- (0.3,3.6) -- (0.3,0.4);
\filldraw[thick, color=lightgray!30] (6,-2) -- (6,6) -- (7,6) -- (7,-2) -- (6,-2);
\filldraw[thick, color=lightgray!30] (-3,-2) -- (-3,6) -- (-2,6) -- (-2,-2) -- (-3,-2);
\filldraw[thick, color=lightgray!30] (7,6) -- (7,7) -- (-3,7) -- (-3,6) -- (7,6);
\filldraw[thick, color=lightgray!30] (-3,-2) -- (-3,-3) -- (7,-3) -- (7,-2) -- (-3,-2);
\draw[thick, dashed] (0.3,0.4) -- (3.7,0.4) -- (3.7,3.6) -- (0.3,3.6) -- (0.3,0.4);
\draw[->] (-3,2) -- (7.5,2) node[below] {$x$};
\draw[->] (2,-3) -- (2,7.5) node[right] {$y$};
\node at (6.5,2.3){$x\gg t$};
\node at (2.7,6.5){$y\gg t$};
\node at (5.2,2.5){$x\sim t$};
\node at (1,5.8){$y\sim t$};
\node at (2.4,4.7){$t^{br}$};
\node at (4.3,2.5){$t^b$};
\node at (0.6,-0.5){$\Omega_1(t)$};
\node at (6.5,3){$\Omega_2(t)$};
\node at (3.4,-0.2){$\widetilde\Omega_1(t)$};
\node at (0.9,3.1){$\widetilde\Omega_2(t)$};
\end{tikzpicture}
\end{center}
\caption{\small Schematic figure depicting the sets $\Omega_1(t)$, $\Omega_2(t)$, $\widetilde \Omega_1(t)$ and $\widetilde \Omega_2(t)$, \emph{in the centered case} $l_1=l_2=0$, as defined in \eqref{Omega}, \eqref{Omega0}, \eqref{Omega1} and \eqref{Omega2}, respectively. Recall that $\frac53<r<3$, $0< b<\frac{2}{3+r}$ and $0< br<\frac{2r}{3+r}$. The largest value of $b$ is $\sim\frac{3}{7}$, and $br\sim \frac57$, obtained by $r\sim \frac53$. Most of the white area outside the regions $|x|\sim t$ and $|y|\sim t$ \emph{can be suitably covered} by using shifts $\rho_1(t),\rho_2(t)$, in the non-centered case.}\label{fig:1}
\end{figure}

\begin{thm}\label{KP_decay}
Assume initial data $u_0$ in $L^2(\mathbb R^2)$ in the KP-II case, and $u_0$ in the energy space $E^1(\mathbb R^2)$ for KP-I. Then the corresponding solution $u$ satisfies
\begin{equation}\label{L2}
 \liminf_{t\to \infty} \int_{\Omega_1(t)} u^2(x,y,t)\,\mathrm{d}x\mathrm{d}y=  0,
\end{equation}
and
\begin{equation}\label{decayfar}
 \lim_{t\to \infty} \int_{\Omega_2(t)} u^2(x,y,t)\,\mathrm{d}x\mathrm{d}y=  0.
\end{equation}
\end{thm}

Note that this result holds for both KP-I and KP-II models. This is due to the fact that both models contain a quadratic nonlinearity, and the proof does not depend on the sign of $\kappa$ in \eqref{KP:Eq}, namely \emph{it does not depend on the linear behavior of each KP model}. It is instead a decay result strictly triggered by the nonlinear behavior of KP. In the particular setting $\ell_1(t)=\ell_2(t)=0$ (the centered case), \eqref{L2} reveals that the zone of strong influence of the linear dynamics is outside $\Omega_1(t)$, although under additional regularity assumptions, one may show $L^\infty$ decay in the non centered case, see \cite{CM} for results in that direction in the one dimensional KdV model. The region possibly containing lumps or line solitons $|x|\sim t$ is of different nature, characterized by shifts of order $\sim t$, and must be treated in a separated way, depending on each KP model one considers. However, in the KP-II setting, more can be said: Kenig and Martel \cite{KM} showed strong decay on the right half-plane, in the following form: for any $\beta>0$,
\begin{equation}\label{zerozero}
\lim_{t\to \infty} \int_{x>\beta t} u^2(x,y,t)\,\mathrm{d}x\mathrm{d}y=  0,
\end{equation}
provided the initial data is small enough in $L^1\cap L^2.$ In this sense, \eqref{decayfar} extends \eqref{zerozero} to large data only in $L^2$. The remaining zero limsup property in \eqref{L2} has unfortunately escaped to us; note that such a property would completely clarify the description of the KP dynamics, in the energy space, in such regions of space. Finally, note that our previous results in the one dimensional KdV setting \cite{MMPP} are in some sense sharp, in the sense that, assuming asymptotic linear decay of the nonlinear dynamics, space-time pointwise estimates on the linear profile lead to an agreement with the decay property established by our techniques. 

\medskip

From the proofs, it will become clear that Theorem \ref{KP_decay} is stable under perturbations of the nonlinearity in the form $f(u)= u^2 + o_{u\to 0}(u^2)$, provided the well-posedness theory of the corresponding KP equation is available. This is a delicate subject; for that reason, the treatment of the term $o_{u\to 0}(u^2)$ requires some care.

\medskip

Given the result of Bourgain \cite{Bourgain1993}, \eqref{L2} provides a sequential description of the dynamics in the KP-II case on any compact set of $\mathbb R^2$. This description can be complemented with the previous works by de Bouard and Saut \cite{dBS1} and de Bouard and Martel \cite{dBM}, where nonexistence of KP-II \emph{compact solutions} \emph{uniformly bounded} in $E^1(\mathbb R^2)$ was proved. Although applied to a different problem and setting, some of the elements in the proofs of \cite{dBM} will be adapted to prove Theorem \ref{thm_KPI} below, concerning the particular case of KP-I.  

\medskip

In order to state our second result, we recall the notion of compact solution introduced by de Bouard and Martel \cite{dBM}:

\begin{defn}\label{def_Non}
Let $u=u(x,y,t)$ be an $E^1(\mathbb R^2)$ solution to KP-I. We say that \emph{$u$ is compact as $t\to \infty$} if there are $x(t),y(t): [0,\infty) \to \mathbb R$ such that for all $\varepsilon>0$ there exists $R(\varepsilon)>0$ such that 
\[
\sup_{t\geq 0}\int_{\|(x-x(t),y-y(t))\| \geq R(\varepsilon)} u^2(x,y,t) \,\mathrm{d}x\mathrm{d}y <\varepsilon.
\]
A similar definition can be introduced for $t\to -\infty$.
\end{defn}

Note that we only ask for information about the solution for times $t\geq 0$ (or $t\leq 0$). Clearly the lump \eqref{lump_general} is a compact solution to KP-I satisfying $x(t)\sim t$ and $y(t)\sim t$. In this paper, as a consequence of Theorem \ref{KP_decay}, we extend \cite{dBS} to the KP-I setting and show nonexistence of periodic compact zero-speed solutions to KP-I (a.k.a. breathers): 

\begin{cor}\label{Non}
Assume that $u$ as in Definition \ref{def_Non} is a periodic in time compact solution to KP-I satisfying: for some $C_0>0$ and $t\ge 0$,
\begin{equation}\label{adentro}
|x(t)|\le C_0 |t|^{m_1}, \quad |y(t)|\le C_0 |t|^{m_2},
\end{equation}
with $m_1,m_2$ in \eqref{indices}. Then $u\equiv 0$.
\end{cor}

The periodicity condition may seem too demanding, but actually Corollary \ref{Non} holds under less restrictive assumptions, see Remark \ref{theend}.

\medskip

Standard $L^1\to L^\infty$ estimates of the linear KP flow show decay of order $1/t$, see Saut \cite{Saut}. However, as explained by Hayashi and Naumkin in \cite{HN}, the quadratic nonlinearity in KP is of critical order for scattering methods. From \cite{HNS} and \cite{HN}, one has explicit space/time decay estimates for both \emph{linear} KP models, assuming sufficient decay, smallness and smoothness on the initial datum. In the nonlinear setting, scattering and decay estimates have been extensively studied in KP models during past years, essentially under small data assumptions. The large data case poses serious obstructions to either inverse scattering methods, or linear dispersive techniques. Concerning to PDE methods, Hadac, Herr and Koch \cite{Hadac-Herr-Koch-2009} showed scattering of small solutions for KP-II in the critical space $\dot H^{-1/2,0}(\mathbb R^2)$ (below $L^2$ regularity). Hayashi, Naumkin and Saut proved scattering of small data for \eqref{KP:Eq} with power nonlinearity bigger than 3, in suitable weighted spaces for both KP-I and KP-II \cite{HNS} (see also \cite{Ni}). In that work (see also \cite{Saut} for previous remarks), the parabolic regions
\[
x+ \frac{y^2}{4\kappa t} = \hbox{const.,}
\]
became important as essentially the worst possible decay regions in the small data KP regime. The somehow strange constraint $r>\frac53$ in \eqref{indices} is easily explained by this parabolic heuristic: under the limiting values $x\sim t^b$, $y\sim t^{br}$ and $x \sim \frac{y^2}{t}$, the largest value of $b$ ($\sim \frac2{3+r}$, namely, the largest window in the $x$ variable) is attained when $r\sim \frac53$. See more about this fact in Fig. \ref{fig:2}.

\medskip

Later, Hayashi and Naumkin \cite{HN} improved their previous results to the actual KP model and described the asymptotics of small KP solutions in high regularity weighted Sobolev spaces. Harrop-Griffits, Ifrim and Tataru \cite{HGIT} showed scattering of small data for KP-I with precise asymptotics by using testing with wave packets, generalized vector field methods, weighted spaces and Galilean invariant norms. An improved description of the dynamics in terms of the parabolic and its complement region was obtained in this work. See also Klein and Saut \cite{KS2} for precise numerical simulations describing the KP dynamics. Isaza, Linares and Ponce \cite{ILP} proved propagation of regularity in KP-II. On the other hand, many authors have considered the KP equations via inverse scattering techniques (IST). See e.g. the monograph by Konopelchenko \cite[Ch. 2]{Kono}. In \cite{W}, precise asymptotics are given for KP-II in the case of small data with eight derivatives in $L^1\cap L^2$. In the KP-I case, a similar description is established for the case of small data in the Schwartz class \cite{Sung}. Finally, see \cite{KS,KS3} for a complete description of the state of art via the use of rigorous IST techniques. 

\medskip


\medskip

One may wonder if there is additional decay in the case of KP-I, where the data is in the energy space $E^1(\mathbb R^2)$ and not only in $L^2(\mathbb R^2)$. It turns out that this question seems quite hard, and it is deeply related to the fact that \emph{small KP-I lumps can travel arbitrarily fast}, see Remark \ref{lump_fast} for details. In this paper, we show decay for $\partial_x u$ and $\partial_x^{-1}\partial_y u$ by assuming data in the second energy space $E^2(\mathbb R^2)$ introduced in \eqref{E2}. Consider the regions
\begin{equation}\label{Omega1}
\widetilde \Omega_1(t):= \left\{ (x,y)\in \Omega_1(t) ~: ~  |x-\ell_1t^{m_1}| \leq t^{b(1-\eta_0)}, \; |y-\ell_2 t^{m_2}|\leq t^{b(1-2\eta_0)} \right\},
\end{equation}
and
\begin{equation}\label{Omega2}
\widetilde\Omega_2(t):= \left\{ (x,y)\in \Omega_1(t) ~: ~  |x-\ell_1t^{m_1}| \leq t^{b(1-4\eta_0)}, \; |y-\ell_2t^{m_2}|\leq t^{b(1-3\eta_0)}\right\},
\end{equation}
for any $\eta_0>0$ sufficiently small (but arbitrary). See also Fig. \ref{fig:1}.

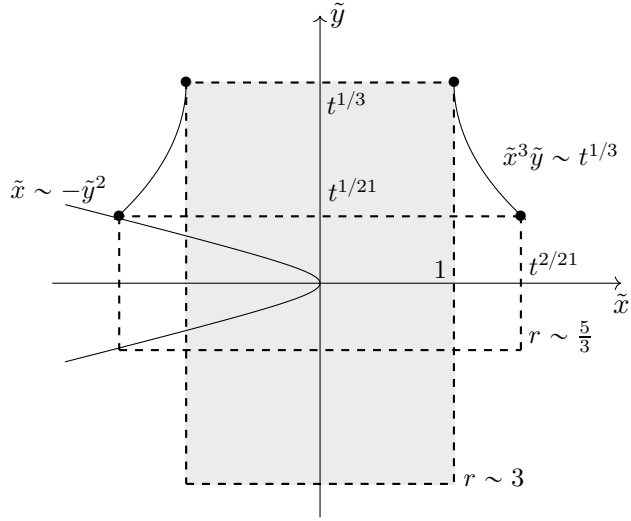
\begin{figure}[h!]
\begin{center}
\begin{tikzpicture}[scale=0.89]
\filldraw[thick, color=lightgray!30] (0,-1) -- (4,-1) -- (4,5) -- (0,5) -- (0,-1);
\draw[thick, dashed] (0,-1) -- (4,-1) -- (4,5) -- (0,5) -- (0,-1);
\draw[thick,dashed] (5,1) -- (5,3);
\draw[thick,dashed] (5,3)--(-1,3);
\draw[thick,dashed] (-1,3)--(-1,1);
\draw[thick,dashed] (-1,1)--(5,1);
\draw[->] (-2,2) -- (6.5,2) node[below] {$\tilde x$};
\draw[->] (2,-1.5) -- (2,6) node[right] {$\tilde y$};
\node at (4,5){$\bullet$};
\node at (5,3){$\bullet$};
\node at (-1,3){$\bullet$};
\node at (0,5){$\bullet$};
\node at (2.4,4.7){\small $t^{1/3}$};
\node at (4.6,-0.9){\small $r\sim 3$};
\node at (5.6,1.2){\small $r\sim \frac53$};
\node at (3.8,2.2){\small $ 1$};
\node at (5.5,2.3){\small $t^{2/21}$};
\node at (2.5,3.4){\small $t^{1/21}$};
\node at (-1.9,3.4){\small $\tilde x \sim - \tilde y^2$};
\draw plot[variable=\t,samples=1000,domain=-78:78] ({3-1*sec(\t)},{2+1/4*tan(\t)});
\draw plot[variable=\t,samples=1000,domain=-19:0] ({4*sec(2*\t)},{5+4.1*tan(1.4*\t)});
\draw plot[variable=\t,samples=1000,domain=-19:0] ({4-4*sec(2*\t)},{5+4.1*tan(1.4*\t)});
\node at (5.6,3.9){\small $\tilde x^3 \tilde y \sim t^{1/3}$};
\end{tikzpicture}
\end{center}
\caption{\small Schematic figure depicting the two limiting rectangles $\Omega_1(t)$ from Theorem \ref{KP_decay} obtained when $r\sim 3$ (shadow region), and $r\sim \frac53$ (white region), in the KP-II case. Here, $\tilde x= \frac{x}{t^{1/3}}$ and $\tilde y:= \frac{y}{t^{2/3}}$ are the standard self-similar variables for KP, see e.g. \cite{Saut}. Additionally, the parabolic region $\tilde x \sim - \tilde y^2$ from \cite{Saut,HNS} is depicted, showing agreement with the limiting value $r\sim \frac53$, which corresponds to the case where (in terms of Theorem \ref{KP_decay}), the largest possible parabolic region is contained. Finally, the first quadrant limit points of the sets $\Omega_1(t)$ (see the bullets on the right), which are of the form $(\tilde x,\tilde y) \sim (t^{\frac{3-r}{3(r+3)}},t^{\frac{2(2r-3)}{3(r+3)}})$, with $\frac53<r<3$, follow the curve $\tilde x^3 \tilde y \sim t^{1/3}$ above depicted. The upper limit $|\tilde y|\sim t^{1/3}$ in the case $r\sim 3$, which corresponds to $|y|\sim t$, is in agreement with the existence of lumps moving in the $y$ direction with speeds $\sim 1$, see \eqref{lump_general}.}\label{fig:2}
\end{figure}

\begin{thm}\label{thm_KPI}
 Assume now initial data $u_0$ in the second energy space $E^2(\mathbb R^2)$ for KP-I. Then the corresponding solution $u$ satisfies
\begin{equation}\label{decay_v}
 \liminf_{t\to \infty} \int_{\widetilde\Omega_1(t)} (\partial_x^{-1}\partial_y u)^2(x,y,t)\,\mathrm{d}x\mathrm{d}y =  0,
\end{equation}
and
\begin{equation}\label{decay_ux}
 \liminf_{t\to \infty} \int_{\widetilde\Omega_2(t)} (\partial_xu)^2(x,y,t)\, \mathrm{d}x\mathrm{d}y=  0.
\end{equation}
\end{thm}

\begin{rem}
Theorem \ref{thm_KPI}, together with Theorem \ref{KP_decay}, states that all the members of the energy space norm decay locally in space along sequences of time, provided the initial data is in the subclass $E^2(\mathbb R^2)$. The family of KP-I lumps belong to this class, see Remark \ref{lump_fast}. It is not clear to us whether or not this condition is also necessary. We conjecture that Theorem \ref{thm_KPI} is valid for data in $E^1(\mathbb R^2)$ only.
\end{rem}





We finish with another corollary, obtained from \cite[Lemma 4]{dBM} and Theorems \ref{KP_decay} and \ref{thm_KPI}: 
\begin{cor}
Any solution $u\in E^2(\mathbb R^2)$ to KP-I satisfies
\[
\liminf_{t\to +\infty} \| u(t)\|_{L^p(K)}  =0, \quad 2\leq p\leq 6,
\]
for any compact set $K\subseteq \mathbb R^2$.
\end{cor}

\subsection{Lumps and line solitons} KP-I has lump solutions, namely solutions of the form
\[
u(x,y,t)=Q_c(x-ct,y), \quad c>0.
\]
The function $Q_c$ is given as $Q_c(x,y):=cQ(\sqrt{c} x,c y)$, where $Q$ is the fixed profile
\[
Q(x,y)=12\partial_x^2 \log(x^2+y^2+3)=\frac{24(3-x^2+y^2)}{(x^2+y^2+3)^2}.
\]
From this formula one clearly sees that lumps are in $E^2(\mathbb R^2)$ and have zero mean. Lumps were first found by Satsuma and Ablowitz \cite{SatAblo} via an intricate limiting process of complex-valued algebraic solutions to KP-I. Also, multi-lump solutions were found in the same reference. de Bouard and Saut \cite{dBS,dBS1,dBS2} described, via PDE techniques, qualitative properties of lumps solutions which are also ground states. In particular, $Q\in E^1(\mathbb R^2)$ and $Q\in E^2(\mathbb R^2)$, and it decays like $1/(x^2+y^2)$, but it is not positive everywhere. However, whether or not lump solutions are ground states is still unknown.

\medskip

Important advances in the understanding of this problem were obtained by Liu and Wei in the aforementioned work \cite{LiuWei}. By using  linear B\"acklund transformation techniques, they proved the orbital stability of the lump in the space $E^1(\mathbb R^2)$. 

\medskip

\begin{rem}\label{lump_fast}
Lumps can travel quite fast. Indeed, recall the invariances of the KP equations \eqref{KP:Eq} \cite{dBM}:
\begin{enumerate}
\item Shifts: 
\[
u(x,y,t) \mapsto u(x+x_0,y+y_0,t +t_0).
\]
\item Scaling: if $c>0$,
\[
u(x,y,t) \mapsto c u(c^{1/2} x, c y,c^{3/2} t).
\]
\item Galilean invariance: for any $\beta\in\mathbb R$, if $u(x,y,t)$ is solution to KP, then
\[
\begin{cases}
u(x,y,t) \mapsto u(x -\beta^2 t -\beta (y-2\beta t) , y-2\beta t,t), & \hbox{KP-I}\\
u(x,y,t) \mapsto u(x + \beta^2 t + \beta (y-2\beta t) , y-2\beta t,t), & \hbox{KP-II},
\end{cases}
\] 
define new solutions to KP.
\end{enumerate}
Using this, one can construct a moving lump solution, of arbitrary size and speed: for any $\alpha,\beta\in\mathbb R$,
\begin{equation}\label{lump_general}
Q_{c,\beta}(x,y,t):= Q_c (x-ct -\beta^2 t -\beta (y-2\beta t) ,y-2\beta t)
\end{equation}
is a moving lump solution of KP-I with speed $\beta\in\mathbb R$. Moreover, a simple computation reveals that (see \eqref{mass}, \eqref{energy} and \eqref{momentum})
\[
\begin{aligned}
&~{} M[Q_{c,\beta}] =M[Q_c] =c^{1/2} M[Q], \\
&~{}  E[Q_{c,\beta}]= E[Q_c] +\beta^2 M[Q_c] = c^{3/2}E[Q] +\beta^2 c^{1/2}M[Q],  \\
&~{}  P[Q_{c,\beta}] = P[Q_c] -\beta M[Q_c] = cP[Q] -\beta c^{1/2}M[Q] =-\beta c^{1/2}M[Q].
\end{aligned}
\]
The previous computations imply that small lumps (in the energy space) may have arbitrarily large speeds ($\beta\gg1$ and $c\ll 1$ such that $\beta^2 c^{1/2}\ll 1$). Precisely, this property of lumps ensures that \emph{no monotonicity on the right of the solution} may hold in the KP-I case, in the sense of Martel and Merle (note that this is not the case in KP-II, see \eqref{zerozero} and de Bouard-Martel \cite{dBM}). A general solution may contain a small fast lump solution that will appear after some time on the right of the main part of the solution. This simply implies that an arbitrary portion of mass on the right of the plane cannot be almost preserved in time. In that sense, the proof of Theorem \ref{thm_KPI} overcomes this difficulty by adding one more degree of regularity, and using the momentum as a trigger of decay, instead of considering only the mass as in standard monotonicity properties.
\end{rem}

In view of the results by de Bouard-Saut \cite{dBS1}, and de Bouard-Martel \cite{dBM}, it is known that KP-II has no lump structures. However, any KdV soliton becomes an (infinite energy) line-soliton solution of KP. Such structure is stable in the KP-II case, as proved by Mizumachi and Tzvetkov \cite{MT}, and Mizumachi \cite{Mizu1,Mizu2}. In the KP-I case, it is transversally unstable, as proved by Rousset and Tzvetkov \cite{RT1,RT2}. Finally, in order to study the stability of the KdV soliton under the flow of the KP-II equation, Molinet, Saut and Tzvetkov \cite{Molinet-Saut-Tzvetkov-2011} proved global well-posedness in $L^2(\mathbb{R}\times \mathbb{T})$  and $L^2(\mathbb R^2)$ (see also Koch and Tzvetkov \cite{Koch-Tzvetkov-2008}).

\subsection{Idea of proofs} 

The proof of Theorems \ref{KP_decay} and \ref{thm_KPI} is based in two new virial identities, being the first one outlined in the early but fundamental work by Gustavo Ponce and the second author \cite{MuPo1,MuPo2}, and improved step by step to the point of being able to consider data only in the ``energy space'' for the Zakharov-Kuznetsov model, see \cite{MMPP}. Some virials identities are independent of the integrability of the equation, consequently, are valid for plenty of dispersive models, see e.g. \cite{MM1,MM2,MM3,Merle,KMM,Cote:Munoz:Pilod:Simpson:2016} and references therein.

\medskip

The idea goes as follows. For technical reasons, we will need the introduction of six scaling factors $\lambda_j(t)$, three attenuation functions $\eta_k(t)$, and two shift parameters $\rho_l(t)$, plus one completely unrelated shift/scaling factor $\theta(t)$. The key point is to have in mind the following inequalities:
\begin{equation}\label{cadena1}
\lambda_1(t) \ll \lambda_2(t) \ll \lambda_4(t) \ll  \lambda_3(t) \ll \lambda_5(t) \ll \lambda_6(t) ,
\end{equation}
and
\begin{equation}\label{cadena2}
\eta_1(t) \gg \eta_2(t) \gg \eta_3(t).
\end{equation}
(See Subsection \ref{Scalings} for further details.) Here, $a\gg b$ for $a,b>0$ means that for all possible constant $C>0$, one has $a\geq Cb$. We will describe here the proof in the KP-I case, since the KP-II one is in some sense simpler. The starting point is the classical Kato smooting functional related to the mass of the solution. For a 2-variable function $\Phi$ which is essentially of the form $\tanh x \sech^2 y$, we introduce the Kato type functional
 \[
 \mathcal K(t) =\frac{1}{\eta_1(t)}\int_{\mathbb{R}^{2}} u^{2}(x,y,t) \Phi \left(\frac{\tilde x}{\lambda_1(t)},\frac{\tilde y}{\lambda_2(t)}\right)\,\mathrm{d}x\mathrm{d}y,
\]
with $\tilde x= x-\rho_1(t)$ and $\tilde y= y-\rho_2(t)$ as translated variables. For simplicity, one can assume $\tilde x=x$ and $\tilde y =y$, having in mind that the results in Theorems \ref{KP_decay} and \ref{thm_KPI} do hold even in the shifted case, with the necessary modifications.

\medskip

This functional (with different weights associated to the soliton dynamics) has been widely used in the KdV setting by Martel and Merle \cite{MM1,MM2,MM3}, and in the Zakharov-Kuznetsov setting in \cite{Cote:Munoz:Pilod:Simpson:2016,MMPP}, because it presents nice Kato smoothing properties (sometimes also called monotonicity properties). However, in the KP setting \emph{these are not available anymore}. The influence of the transversal term $\partial_x^{-1}\partial_y^2 u$ in the equation becomes evident and poses strong difficulties to the obtention of monotonicity properties a la Martel-Merle. Up to now, no such property was known in the KP-I case, and in the KP-II case it holds under suitable assumptions (see \cite{dBM}). 

\medskip

The first virial estimate, related to $\mathcal K(t)$, reads as follows: it is possible to bound $\partial_x u$ $L^2$ locally in space, provided one has local $L^2$ control on $\partial_x^{-1}\partial_y u$ and local $L^3$ control on $u$. Indeed, for constants $\sigma_0,C_0>0$,
\begin{equation}\label{K_intro}
\begin{aligned}
&~{} \frac{\sigma_0}{ t\log t}\int_{\mathbb{R}^{2}} (\partial_x u)^2 \sech^2 \left(\frac{\tilde x}{\lambda_1(t)} \right)\sech^2 \left(\frac{\tilde y}{\lambda_2(t)}\right)   \, \mathrm{d}x\mathrm{d}y  \\
&~{} \qquad  \leq  -\frac{d\mathcal K(t)}{dt} +  \frac{C_0}{t\log t}\int_{\mathbb{R}^{2}} (\partial_x^{-1}\partial_y u)^2   \sech^2 \left(\frac{\tilde x}{\lambda_1(t)} \right) \sech^2\left(\frac{\tilde y}{\lambda_2(t)}\right)   \, \mathrm{d}x\mathrm{d}y \\
&~{} \qquad \quad  +\frac{C_0}{t\log t} \int_{\mathbb{R}^{2}} |u|^3 \sech^2\left(\frac{\tilde x}{\lambda_1(t)} \right) \sech^2\left(\frac{\tilde y}{\lambda_2(t)}\right) \, \mathrm{d}x\mathrm{d}y+ \mathcal K_{int}(t),
\end{aligned}
\end{equation}
with $ \mathcal K_{int}(t)\in L^1(\{t\gg1\})$, see Proposition \ref{paso1} for further details. This is done with a correct choice of $\lambda_5,\lambda_6$ and $\eta_3$ chosen in order to \emph{maximize the area of decay} as much as possible. Note that the terms on the right hand side are still uncontrolled in the large data setting. Indeed, the term with $\partial_x^{-1}\partial_y u$ appears precisely with the wrong sign because of the lack of monotonicity in KP-I, and it is hard to control even in the small data case.

\medskip

In this paper we overcome this problem by adding one more degree of regularity to the initial data. This is done by working in the second energy space $E^2(\mathbb R^2)$ introduced by C.~E. Kenig  \cite{Kenig-2004}. We define the functional
\begin{equation*}
\mathcal{J}(t)=\frac{1}{\eta_2(t)}\int_{\mathbb{R}^{2}} \left( u \partial_x^{-1}\partial_y u \right)(x,y,t)\, \Psi \left(\frac{\tilde x}{\lambda_{3}(t)},\frac{\tilde y}{\lambda_{4}(t)}
\right)\,\mathrm{d}x\,\mathrm{d}y,
\end{equation*}
valid for solutions in $E^1(\mathbb R^2)$. Contrary to the previous case ($\mathcal{K}(t)$ and $\Phi$), here $\Psi$ is similar to $\sech^2 x\tanh y$. This functional is inspired in the one introduced by de Bouard-Martel \cite{dBM} for the KP-II setting, which was $\int y u \partial_x^{-1}\partial_y u \,\mathrm{d}x\,\mathrm{d}y$. For this functional we prove the following. Let $u\in E^2(\mathbb R^2)$ be a globally defined solution to KP-I. Under appropriate choices of $\lambda_3,\lambda_4$ and $\eta_2$ above, and having in mind the chain of inequalities in \eqref{cadena1}-\eqref{cadena2}, there exist constants $\sigma_1,C_1>0$ such that
\begin{equation}\label{J_intro}
\begin{aligned}
&~{} \frac{\sigma_1}{t\log t}\int_{\mathbb{R}^{2}} (\partial_x^{-1}\partial_y u)^{2}\sech^2\left(\frac{\tilde x}{\lambda_{3}(t)}\right)\sech^2 \left(\frac{\tilde y}{\lambda_{4}(t)}
\right)\mathrm{d}x\mathrm{d}y \\
&~{} \qquad \leq   -\frac{d}{dt}\mathcal J(t)  +  \frac{C_1}{t\log t}\int_{\mathbb{R}^{2}} |u|^{3} \sech^2\left(\frac{\tilde x}{\lambda_{3}(t)} \right) \sech^2 \left(\frac{\tilde y}{\lambda_{4}(t)}\right) \mathrm{d}x\mathrm{d}y\\
&~{} \qquad \quad + \frac{C_1}{t} \int_{\mathbb{R}^{2}}u^2\sech^2\left(\frac{\tilde x}{\lambda_{3}(t)} \right) \sech^2 \left(\frac{\tilde y}{\lambda_{4}(t)}\right)\mathrm{d}x\mathrm{d}y  + \mathcal J_{int}(t),
\end{aligned}
\end{equation}
where $\mathcal J_{int}(t)\in L^1(\{ t\gg1\})$; see Proposition \ref{paso2} for more details. Proving \eqref{J_intro} is not direct, mainly because the equation formally satisfied by $\partial_x^{-1}\partial_y u$ introduces several new highly nonlocal terms in the virial estimate, which are hard to control unless one has better regularity estimates. One of those terms is $\partial_x^{-2} \partial_y^2 u$, which is only well-defined in a proper subspace of the energy space $E^1(\mathbb R^2)$. Here is when the space $E^2(\mathbb R^2)$ in \eqref{E2} enters to the scene: bounded in time solutions  in this space make \eqref{J_intro} possible.

\medskip

The previous virial estimate allows us to bound the term $\partial_x^{-1}\partial_y u$ in \eqref{K_intro} just in terms of local $L^2$ and $L^3$ terms on $u$. Note that in principle, these last terms do not integrate in time. Therefore, to conclude we need suitable integrability bounds on those terms, independent of the size of the data. 

\medskip

Here is when a third virial estimate is needed. Generalizing our previous results for the Zakharov-Kuznetsov equation \cite{MMPP}, we introduce the $L^1$ functional   
\[
\mathcal I(t)=\frac{1}{\eta_3(t)}\int_{\mathbb{R}^{2}}u(x,y,t)\psi\left(\frac{\tilde x}{\lambda_5(t)}\right)\phi\left(\frac{\tilde x}{\lambda_5^q(t)}\right)\phi\left(\frac{\tilde y}{\lambda_6(t)}\right)\,\mathrm{d}x\mathrm{d}y,
\]
with $q>1$, and $\lambda_5,\lambda_6$ and $\eta_3$ well-chosen, always following \eqref{cadena1} and \eqref{cadena2}. The weight $\psi$ behaves as $\tanh$ and $\phi$ as $\sech$. At this moment the quadratic character of KP enters and makes possible to obtain integrability properties for the local $L^2$ norm of $u$. This is a key property to describe the radiative terms in dispersive models (and it is not satisfied by soliton-like solutions, which are nondecaying objects). These $L^1$ type functionals have been considered in different settings, see e.g. \cite{BSS,MM3,Merle}, specially in blow-up or instability of solitons. A primitive version of $\mathcal I(t)$ was introduced by Ponce and the second author in \cite{MuPo1} and \cite{MuPo2} in the one dimensional gKdV setting, but the additional hypothesis $u\in L^\infty_{loc,t} L^1_x$ was always necessary. This extra integrability condition was removed in \cite{MMPP}, where we considered the less involved Zakharov-Kuznetsov (ZK) model. In particular, in that work estimate \eqref{K_intro} (Proposition \ref{paso1}) was direct because of nice Kato smoothing properties present in ZK, but not present in KP. 

\medskip

In principle, $\int_{\mathbb R^2} u$ is not well-defined in $L^2$; this makes very important to consider the weights above introduced. The weight with power $q$ is introduced by technical reasons; the most important contribution is always related to the weight with scaling $\lambda_5(t)$.  For this functional, we prove, regardless one considers KP-I or KPI-II, the following: assume data only in $L^2(\mathbb R^2)$ for KP-II, or data in $E^1(\mathbb R^2)$ in the KP-I case. Then there exists $\sigma_2>0$ and $\varepsilon_0>0$ small enough such that, for any $t\gg 1$, one has the bound
\begin{equation}\label{I_intro}
\begin{split}
\frac{\sigma_2}{t}\int_{\mathbb{R}^{2}} u^{2} \sech^2\left(\frac{\tilde x}{\lambda_5(t)}\right)\sech^2\left(\frac{\tilde y}{\lambda_6(t)}\right)\,\mathrm{d}x\mathrm{d}y &\leq \frac{d\mathcal I }{dt}(t) + \mathcal I _{int}(t),
\end{split}
\end{equation}
provided $q=1+\varepsilon_0$ in $\mathcal I(t)$ \eqref{mI}, and where $\mathcal I _{int}(t)$ are terms that belong to $L^1\left(\{t\gg 1\}\right)$ (see Proposition \ref{le:dT}). Corollary \ref{Non} follows directly from this estimate. A similar estimate was proved in the ZK case in \cite{MMPP}. However, there are some key differences in \eqref{I_intro} in the KP and ZK cases. At first sight,  a standard estimate as in \cite{MMPP} reveals that \eqref{I_intro} should not hold in the KP case because the nonlocal operator $\partial_x^{-1}$ does not send $L^2$ into $L^2$, leading to a critical loss of time decay. Fortunately, we will  benefit of a new decay relation between the nonlocal operator $\partial_x^{-1}$ and well-chosen virial weights $\psi$ and $\phi$. After some work, it will be proved that \eqref{I_intro} is also valid in the KP setting. The price to pay is the more restrictive condition $r>\frac53$ in $\Omega_1(t)$. Recall that $r$ is a parameter related to the $y$ variable, but also reduces the $x$ variable, and $r$ larger means that the size of $\Omega_1(t)$ is smaller in $x$ than in the ZK case. This difference is an intriguing feature obtained from the proof, and we believe that it is not caused by technical reasons, but because of weaker spatial decay properties.

\medskip

Estimates \eqref{J_intro} and \eqref{I_intro} are, as far as we know, new in the KP setting, even in the small data regime. We emphasize that no smallness condition is needed for proving both estimates, only the validity of data in the corresponding energy spaces. Moreover, \eqref{I_intro} allows us to conclude in a direct form that (see Lemma \ref{cubico} and Subsection \ref{finaliza}) every solution to KP-I in $E^2(\mathbb R^2)$ satisfy the mild decay in time properties
\begin{equation}\label{las3}
\begin{aligned}
& ~{} \frac{1}{t\log t}\int_{\mathbb{R}^{2}}|u|^{3} (x,y,t)\sech^2 \left(\frac{\tilde x}{\lambda_{3}(t)} \right) \sech^2 \left(\frac{\tilde y}{\lambda_{4}(t)}\right)\mathrm{d}x\mathrm{d}y \in L^1(\{t\gg 1\}),\\
& ~{} \frac{1}{t\log t}\int_{\mathbb{R}^{2}} (\partial_x^{-1}\partial_y u)^{2}(x,y,t)\sech^2\left(\frac{\tilde x}{\lambda_{3}(t)}\right)\sech^2 \left(\frac{\tilde y}{\lambda_{4}(t)} \right)\mathrm{d}x\mathrm{d}y\in L^1(\{t\gg 1\}),\\
& ~{} \frac{1}{ t\log t}\int_{\mathbb{R}^{2}} (\partial_x u)^2 (x,y,t) \sech^2 \left(\frac{\tilde x}{\lambda_1(t)} \right)\sech^2 \left(\frac{\tilde y}{\lambda_2(t)}\right)   \, \mathrm{d}x\mathrm{d}y \in L^1(\{t\gg 1\}).
\end{aligned}
\end{equation}
In view of \eqref{K_intro}, the last estimate on the gradient can be recast as the first monotonicity property obtained for solutions to KP-I. Unlike the gKdV and ZK setting, it requires a combination of the functionals $\mathcal K(t)$, $\mathcal J(t)$ and $\mathcal I(t)$. This combination has no apparent sign, in concordance with the existence of fast small lumps. Following the ideas in \cite{MuPo2}, the three estimates in \eqref{las3} lead to \eqref{L2} in Theorem \ref{KP_decay} (valid for data only in $L^2$ or $ E^1(\mathbb R^2)$), and combined with Propositions \ref{paso1} and \ref{paso2} (\eqref{K_intro} and \eqref{J_intro}), lead to Theorem \ref{thm_KPI}. 

\medskip

Finally, \eqref{decayfar} is proved using a technique introduced in \cite{MPS}, which follows a different procedure, starting with another virial functional in the spirit of $\mathcal K(t)$, but with a completely different choice of parameters. Indeed, we introduce the functionals
\[
\mathcal L_x(t)= \frac{1}{2}\int_{\mathbb{R}^{2}}u^{2}\chi\left(\frac{x+\theta (t)}{\theta (t)}\right)\,\mathrm{d}x\mathrm{d}y, \quad \mathcal L_y(t)= \frac{1}{2}\int_{\mathbb{R}^{2}}u^{2}\chi\left(\frac{y+\theta (t)}{\theta (t)}\right)\,\mathrm{d}x\mathrm{d}y;
\]
with $|\theta(t)| \gg t$, and for a bump function satisfying $\chi' \leq 0$, and with support in $[-1,0]$. Then we prove (see Proposition \ref{paso4}) that
\[
\frac1t\int_{\mathbb{R}^{2}}u^{2} (x,y,t)\left| \chi'\left(\frac{x+\theta (t)}{\theta (t)}\right) \right|  \,\mathrm{d}x\mathrm{d}y \leq -\frac{d}{dt} \mathcal L_x(t) + \mathcal L_{x,int}(t),
\]
and
\[
\frac1t\int_{\mathbb{R}^{2}}u^{2}(x,y,t) \left| \chi'\left(\frac{y+\theta (t)}{\theta (t)}\right) \right|  \,\mathrm{d}x\mathrm{d}y \leq -\frac{d}{dt} \mathcal L_y(t) + \mathcal L_{y,int}(t),
\]
where both $ \mathcal L_{x,int}(t)$ and $ \mathcal L_{y,int}(t)$ are in $L^1(\{t\gg 1\})$. Note that these estimate overcome the monotonicity problem present in KP-I, since we work in regions where lumps are not present. This is the key of the proof, implying local decay in the regions $|x|\sim \theta(t)$ and $|y|\sim \theta(t)$. These virial identities hold without the use of attenuation functions in $\mathcal L_x(t)$ and $\mathcal L_y(t)$, which implies that the recover the full limit and not only a liminf.

\medskip

We finish this long introduction by mentioning that we believe that the techniques introduced in the proofs above mentioned will be very useful for the proof of asymptotic stability of the KP-I lump \eqref{lump_general} as well. 

\subsection*{Organization of this paper}  This work is organized as follows. In Section \ref{Sect:2} we introduce the necessary preliminary results needed for the proof of main theorems. Section \ref{Sect:6} is devoted to the proof of \eqref{decayfar}. Section \ref{sect:ux} deals with a virial estimate for the derivative $\partial_x u$ of the solution in the KP-I case. Section \ref{Sect:v} is devoted to a new virial estimate for the variable $\partial_x^{-1}\partial_y u$ in the KP-I case. Section \ref{Sect:3} is concerned with the proof of Theorem \ref{KP_decay} of $L^2$-decay in both KP models.  Finally, Section \ref{Final} ends with the proof of Theorem \ref{thm_KPI} and Corollary \ref{Non}.

\subsection*{Acknowledgments} We thank Yvan Martel and Jean-Claude Saut for many useful discussions regarding the KP models along the years, Jean-Claude Saut for making available to us the monograph by Klein and Saut \cite{KS}, and Felipe Linares and Jean-Claude Saut for useful comments on a first version of this paper.

\bigskip

\section{Preliminaries}\label{Sect:2}

This section is devoted to establish or recall a series of relevant results needed for the proofs of Theorems  \ref{decayfar}, \ref{KP_decay} and \ref{thm_KPI}.

\subsection{Interpolation Inequalities}
By $H^{s}_{-1}(\mathbb{R}^{2}), s\in \mathbb{R},$ we denote the classical Sobolev spaces equipped with the norm
\begin{equation}\label{Sobolev}
\|f\|_{H^{s}_{-1}}:=\left\|\left(1+|\xi |\right)^{-1}\langle |\xi|+|\eta|\rangle ^{s}\mathcal{F}(f)(\xi ,\eta)\right\|_{L^{2}_{\xi,\eta}},
\end{equation}
where $\langle x\rangle:=\left(1+|x|^{2}\right)^{1/2}$ and $\mathcal{F}$ denotes the Fourier transform. The space $H^{\infty}_{-1}(\mathbb{R}^{2}),$ is defined as the  intersection of all the $H^{s}_{-1}(\mathbb{R}^{2}).$ For the proof of the following interpolation inequality, see L. Molinet, J.-C. Saut, and  N. Tzvetkov \cite{Molinet-Saut-Tzvetkov-2002}.
\begin{lem}\label{Interpol}
	For $2\leq p\leq 6$ there exist  $c>0$ such that  for every $f\in H^{\infty}_{-1}(\mathbb{R}^{2})$
	\begin{equation}\label{Interpol1}
	\|f\|_{L^{p}}\leq c\|f\|_{L^{2}}^{\frac{6-p}{2p}}\|\partial_{x}f\|_{L^{2}}^{\frac{p-2}{p}}\left\|\partial_{x}^{-1}\partial_{y}f\right\|_{L^{2}}^{\frac{p-2}{2p}}.
	\end{equation}
\end{lem}
This estimate pass to the limit and by density holds true in the energy space $E^1(\mathbb R^2)$.

\subsection{Local and global well-posedness} Let us first mention some of the consequences of the global well-posedness of Ionescu-Kenig-Tataru \cite{Ionescu-Kenig-Tataru-2008}. Assume that $u_0\in E^1(\mathbb R^2)$; then the corresponding solution to KP-I satisfies
\begin{equation}\label{bounded1}
\sup_{t\geq 0} \left( \|u(t)\|_{L^2_{xy}} +  \|\partial_x u(t)\|_{L^2_{xy}} +\| \partial_x^{-1} \partial_y u (t)\|_{L^2_{xy}}\right) \leq C,
\end{equation}
for a fixed constant $C=C(\|u_0\|_{E^1})$. Estimate \eqref{bounded1} is a standard consequence of the conservation of mass \eqref{mass} and energy \eqref{energy}, together with \eqref{Interpol1} for $p=3$.

\medskip

More involved is the following boundedness property obtained from C.~E.~Kenig \cite{Kenig-2004}. Recall the space $E^2(\mathbb R^2)$ introduced in \eqref{E2}. We also need the auxiliary space, defined for $s\in\mathbb R$:
\begin{equation*}\label{Ys}
Y_s:= \left\{ u\in L^2(\mathbb R^2) ~: ~ \|u\|_{Y_s}:= \|u\|_{L^2_{xy}} + \|\langle D_x\rangle^s u\|_{L^2_{xy}} + \|\partial_x^{-1}\partial_y u\|_{L^2_{xy}} <\infty \right\}.
\end{equation*}
Recall that, in terms of Fourier variables, $\mathcal F(\langle D_x\rangle^s u)(\xi,\mu)= (1+|\xi|^2)^{s/2} \mathcal F( u)(\xi,\mu).$
\begin{lem}\label{boundedE2}
Let $u_0\in E^2(\mathbb R^2)$ and let $u$ be the global in time solution of KP-I with initial data $u(t=0)=u_0$. Then one has
\begin{equation}\label{bounded2}
\sup_{t\ge 0} \left( \| \partial_y u(t)\|_{L^2_{xy}} +  \|\partial_x^2 u(t)\|_{L^2_{xy}} +\| \partial_x^{-2}\partial_y^2 u(t)\|_{L^2_{xy}} \right)  \leq C,
\end{equation}
with constant only dependent on $\|u_0\|_{E^2}$.
\end{lem}
Indeed, in \cite{Kenig-2004}, for any $s\in (\frac32,2]$ a unique global solution $u\in C(\mathbb R_+,Y_s) \cap L^\infty(\mathbb R_+ ,E^2(\mathbb R^2))$ was constructed from initial data $u_0 \in E^2(\mathbb R^2)$ ($E^2(\mathbb R^2)$ is continuously embedded in $Y_s$ for $s\in (\frac32,2]$). Moreover, one has $\sup_{t\geq 0}\|u(t)\|_{E^2} \le C(\|u_0\|_{E^2})$. Finally, from \cite{Kenig-2004}, we also know that $\|\partial_y u\|_{L^2} \lesssim\|u\|_{E^2}$. These arguments provide \eqref{bounded2}.

\subsection{Weighted functions adapted to KP}\label{Preli1}
Let    $\phi$ be a smooth even and  positive function  such that
\begin{itemize}
	\item[(i)] $\phi'\leq 0$ \quad on\quad  $\mathbb{R}^{+},$
	
	\item[(ii)]  $\phi|_{[0,1]}=1,$\quad  $\phi(x)=e^{-x}$\quad  on\quad  $[2,\infty),\quad e^{-x}\leq \phi(x)\leq 3e^{-x}$\quad  on\quad  $\mathbb{R}^{+}.$
	
	\item[(iii)] The derivatives of $\phi$ satisfy:
\begin{equation*}
|\phi'(x)|\leq c \phi(x)\qquad\mbox{and}\qquad |\phi''(x)|\leq c\phi(x),
\end{equation*}
for some positive constant $c$.
\end{itemize}

\medskip

Let
\begin{equation}\label{psi_phi}
\psi(x):=\int_{0}^{x} \phi(s)\,\mathrm{d}s, \quad  \psi'(x)=\phi(x).
\end{equation}
Then $\psi$ is an odd function  such that  $\psi(x)=x$ on $[-1,1]$ and $|\psi(x)|\leq 3.$ Notice that
\begin{equation}\label{sigmabound}
\begin{aligned}
\psi(x)= &~{} x\quad \mbox{on}\quad [-1,1],\\
 |\psi(x)|\leq 3,\qquad & e^{-|x|}\leq \phi(x)\leq 3 e^{-|x|}\quad \mbox{on}\quad \mathbb{R}.
\end{aligned}
\end{equation}
Finally, recall that $\psi$ is odd and $\phi$ is even. The following result is key for the proof of Theorem \ref{KP_decay}.

\begin{lem}\label{Localizacion}
	Let $a,b >0$ and $x_0\in\mathbb R$. Then 
	\[
	\partial_{x}^{-1}\left(\psi\left(\frac{\cdot -x_0}{a}\right)\phi\left(\frac{\cdot -x_0}{b}\right)\right)(x)
	\]
	is Schwartz in $x$ and  the following inequality holds
	\begin{equation}\label{bound}
\left|	\partial_{x}^{-1}\left(\psi\left(\frac{\cdot -x_0}{a}\right)\phi\left(\frac{\cdot -x_0}{b}\right)\right)(x)
\right|\leq 9b\phi\left(\frac{x-x_0}{b}\right),\quad \mbox{for all}\quad x\in \mathbb{R}.
	\end{equation}
\end{lem}
\begin{proof}
First, note that $\psi\left(\frac{\cdot -x_0}{a}\right)\phi\left(\frac{\cdot-x_0}{b}\right)$ is odd with respect to $x_0$ and Schwartz in $x$. By virtue of \eqref{sigmabound}, the term	 $\partial_{x}^{-1}\left(\psi\left(\frac{\cdot-x_0}{a}\right)\phi\left(\frac{\cdot-x_0}{b}\right)\right)(x)$ is then well-defined because
\[
\begin{aligned}
\partial_{x}^{-1}\left(\psi\left(\frac{\cdot-x_0}{a}\right)\phi\left(\frac{\cdot-x_0}{b}\right)\right)(x)= &~{} \int_{-\infty}^x  \left(\psi\left(\frac{s-x_0}{a}\right)\phi\left(\frac{s-x_0}{b}\right)\right)\mathrm{d}s\\
=&~{} \int_{-\infty}^{x-x_0}  \left(\psi\left(\frac{s}{a}\right)\phi\left(\frac{s}{b}\right)\right)\mathrm{d} s.
\end{aligned}
\]	
With no loss of generality, we can assume $x_0=0$ below. To obtain the bound in \eqref{bound} we shall examine the following cases:	
\begin{itemize}
	\item[(I)] Case $x\leq0:$
	by definition  we have 
	\begin{equation*}
	\begin{split}
	\partial_{x}^{-1}\left(\psi\left(\frac{\cdot}{a}\right)\phi\left(\frac{\cdot}{b}\right)\right)(x)&=\int_{-\infty}^{x}\psi\left(\frac{s}{a}\right)\phi\left(\frac{s}{b}\right)\,\mathrm{d}s=b\int_{-\infty}^{x/b} \psi\left(\frac{bs}{a}\right)\phi(s)\,\mathrm{d}s.
	\end{split}
	\end{equation*}
	Thus,
	\begin{equation*}
	\begin{split}
	\left|\partial_{x}^{-1}\left(\psi\left(\frac{\cdot}{a}\right)\phi\left(\frac{\cdot}{b}\right)\right)(x)\right|&\leq b\int_{-\infty}^{x/b} \left|\psi\left(\frac{bs}{a}\right)\right|\phi(s)\,\mathrm{d}s\\
	&\leq 9 b \int_{-\infty}^{x/b} e^{s}\,\mathrm{d}s=9 b e^{-\frac{|x|}{ b}}\leq 9 b\phi\left(\frac{x}{b}\right).
	\end{split}
	\end{equation*}
\item[(II)]
Case $x\geq 0:$ since the function 	$\psi\left(\frac{x}{a}\right)\phi\left(\frac{x}{b}\right)$ is odd, we  have from the definition  of $\partial_{x}^{-1}$ that 
\begin{equation*}
\begin{split}
\left|\partial_{x}^{-1}\left(\psi\left(\frac{\cdot}{a}\right)\phi\left(\frac{\cdot}{b}\right)\right)(x)\right|&=\left|-\int_{x}^{\infty}\psi\left(\frac{s}{a}\right)\phi\left(\frac{s}{b}\right)\,\mathrm{d}s\right|\\
&\leq \int_{x}^{\infty}\left|\psi\left(\frac{s}{a}\right)\right|\phi\left(\frac{s}{b}\right)\,\mathrm{d}s\\
&\leq 9 b\int_{x/b}^{\infty}e^{-s }\,\mathrm{d}s\leq 9 b\phi\left(\frac{x}{b}\right).
\end{split}
\end{equation*}
\end{itemize}	  
Thus, combining both cases we obtain  the bound
\begin{equation*}
\left|\partial_{x}^{-1}\left(\psi\left(\frac{\cdot}{a}\right)\phi\left(\frac{\cdot}{b}\right)\right)(x)\right|\leq 9b\phi\left(\frac{x}{b}\right),\quad x\in\mathbb{R}.
\end{equation*}
This finishes the proof.
\end{proof}

Finally, we will consider a function $\chi\in C^{\infty}(\mathbb{R})$ satisfying:
\begin{equation}\label{chichi}
\begin{aligned}
& \hbox{$\chi(x)=1$ for $x\leq -1,$ $\chi(x)=0$ for $x\geq 0,$}\\
& \hbox{$\supp(\chi)\subset(-\infty,0]$ and $\chi'(x)\leq 0$  for all $x\in\mathbb{R}.$}
\end{aligned}
\end{equation}
Additionally,  $\chi'$ satisfies  the inequality 
$-\chi'(x)\geq c1_{[-\frac{3}{4},-\frac{1}{4}]}(x)$ for all $x\in\mathbb{R},$ and $c$ is a positive constant. This will be useful in Section \ref{Sect:6}.
 
\subsection{Scaling and attenuation parameters}\label{Scalings} We introduce now a set of time-dependent parameters that allow us to capture decay properties in the long time regime. We classify them in three classes: scaling parameters $\lambda_j(t)$, and $j=1,\ldots, 6$, attenuation parameters $\eta_1(t), \eta_2(t)$ and $\eta_3(t)$, and a sort of shift/scaling parameter $\theta(t)$. But first, we recall the exponents introduced in the definition of the set $\Omega_{1}(t)$ in \eqref{Omega}.

\medskip

Let  $b,r>0$ and  $q\in(1,2)$ numbers satisfying the conditions
\begin{equation}\label{bqconditions_new}
  b\le\frac{2}{2+q+r}<\frac{2}{3+r},\quad \frac{5}{3}<r<3.
\end{equation}
Note that for $\frac{5}{3}<r<3$ we easily have
\begin{equation*}\label{br}
br<\frac{2r}{3+r} <1.
\end{equation*}
Always consider $t\gg 1$. The key idea to have in mind is that
\begin{equation}\label{comparacion_final}
\begin{aligned}
& \lambda_1(t) \ll \lambda_2(t)  \ll \lambda_4(t) \ll \lambda_3(t) \ll \lambda_5(t) \ll \lambda_6(t),\\  
& \eta_3(t) \ll \eta_2(t) \ll \eta_1(t).
\end{aligned}
\end{equation}
First, we shall introduce $\lambda_5(t)$, $\lambda_6(t)$ and $\eta_3(t)$, as follows:
\begin{equation}\label{defns}
\lambda_5(t)=\frac{t^{b}}{\log t}\quad\mbox{and}\quad \eta_3(t)=t^{p}\log t,
\end{equation}
where $p$ is a  positive constant satisfying 
\begin{equation}\label{eq4}
p+b=1.
\end{equation}
We also consider
\begin{equation}\label{eq5}
\lambda_6(t)=\lambda_5^{r}(t) \qquad\mbox{where}\quad \frac{5}{3}< r< 3.
\end{equation}
Then,
\begin{equation*}
\frac{\lambda_5'(t)}{\lambda_5(t)}\sim \frac{\eta_3'(t)}{\eta_3(t)}\sim \frac{1}{t}\qquad\mbox{for}\quad t\gg1.
\end{equation*}
Also,
\begin{equation*}
\lambda_{1}'(t)=\frac{1}{t^{1-b}\log t}\left(\frac{b\log t-1}{\log t}\right)\quad \mbox{and}\quad \lambda_5(t)\eta_3(t)= t.
\end{equation*}
These parameters will be important to understand the $L^2$-decay of KP solutions, mainly in Section \ref{Sect:3}.

\medskip

Based on the previous parameters $\lambda_5(t)$, $\lambda_6(t)$ and $\eta_3(t)$, we construct $\eta_2(t)$, $\lambda_{3}(t)$ and $\lambda_{4}(t)$. Let $\eta_2(t)$, $\lambda_{3}(t),\lambda_{4}(t)$ locally bounded functions defined as follows: for $q=1+\varepsilon_0>1$, $\varepsilon_0>0$ small enough, 
\begin{equation}\label{def_eta2}
\begin{aligned}
& \eta_2(t)=t^{1-b} \log^{4} t, \quad \lambda_{3}(t)= \frac{\lambda_5(t)}{\log t} = \frac{t^b}{\log^2 t}, \\
& \lambda_{4}(t) = \frac{t^b}{\log^3 t} \ll \lambda_6(t) = \frac{t^{br}}{\log^{r} t} ,   \qquad \left(r>\frac53\right);
\end{aligned}
\end{equation}
such that
\begin{equation}\label{producto}
\lambda_3(t) \gg \lambda_4(t),\quad  \eta_2(t) \lambda_4(t) = t\log t, \quad \eta_2(t) \lambda_3(t) =t \log^2 t.
\end{equation}
Note that
\begin{equation}\label{condiciones1}
\frac1{\eta_2(t)\lambda_4(t)}\not\in L^1(\{t\gg 1\}), \quad \frac1{\eta_2(t)\lambda_3(t)}\in L^1(\{t\gg 1\}).
\end{equation}
The previously defined parameters will be important in Section \ref{Sect:v}.

\medskip

Finally, let $\lambda_1(t),\lambda_2(t),$ and $\eta_{3}(t)$ be time-dependent functions satisfying
 \begin{equation}\label{def_eta3}
\begin{aligned}
& \eta_1(t)=t^{1-b} \log^{6} t, \quad \lambda_1(t)= \frac{\lambda_3(t)}{\log^3 t} = \frac{t^b}{\log^5 t}, \\
& \lambda_2(t) = \frac{t^b}{\log^4 t} \ll \lambda_4(t) = \frac{t^{b}}{\log^{3} t};
\end{aligned}
\end{equation}
such that
\begin{equation}\label{producto2}
\lambda_2(t) \gg \lambda_1(t),\quad  \eta_1(t) \lambda_1(t) = t\log t, \quad \eta_1(t) \lambda_2(t) =t \log^2 t ;
\end{equation}
Finally, note that
\begin{equation*}\label{condiciones2}
\frac1{\eta_1(t)\lambda_1(t)}\not\in L^1(\{t\gg 1\}), \quad \frac1{\eta_1(t)\lambda_2(t)}\in L^1(\{t\gg 1\}).
\end{equation*}

\medskip

Next, we introduce the shift parameters needed for $\Omega_1(t)$ in \eqref{Omega}. For constants $\ell_1,\ell_2\in\mathbb R$ and $m_1,m_2\geq 0$, let
\begin{equation}\label{rho12}
\rho_1(t):=\ell_1 t^{m_1}, \quad \rho_2(t):=\ell_2 t^{m_2}.
\end{equation}
For further purposes, we need the following additional restrictions on $m_1$ and $m_2$:
\begin{equation}\label{cond_rho1rho2}
0\leq m_1<1-\frac{b}2(r+1), \quad 0\leq m_2 < 1-\frac12b(q+2-r).
\end{equation}
Note that $r>1$, therefore $\frac12(r+1)<r$. Consequently, $\frac{b}2(r+1) <br<1$. On the other hand, since $b\leq \frac{2}{2+q+r}$,
\[
\frac12b(q+2-r) \leq \frac{2+q-r}{2+q+r}< \frac{\frac13+q}{\frac{11}3+q} <1.
\]
Finally, fixed $m_2$ such that $0\leq m_2 < 1-\frac12b(3-r)$, as in \eqref{indices}, for $q=1+\varepsilon_0$ and $\varepsilon_0>0$ small enough, one has $0\leq m_2 < 1-\frac12b(q+2-r)$. These last facts corroborates that the conditions \eqref{cond_rho1rho2} are well-defined.
\medskip

Finally, in order to prove \eqref{decayfar}, we will need the following scaling functions. Let 
\begin{equation}\label{theta_1}
\theta(t)= t^{p}\log^{1+\epsilon}t,
\end{equation}
for some $\epsilon>0,$ and $p\geq 1.$ Thus,
	\begin{equation}\label{e2}
		\frac{\theta'(t)}{\theta(t)}\sim \frac{1}{t}\quad \mbox{for}\quad t\gg 1.
	\end{equation}
Unlike the previous results, here we have 
\begin{equation}\label{e3}
	\frac{1}{\theta(t)}=\frac{1}{t^{p}\log^{1+\epsilon}t}\in L^{1}(\{t\gg 1\}).
\end{equation}

\bigskip

\section{Decay in far regions. Proof of decay property \eqref{decayfar}}\label{Sect:6}

\subsection{Preliminaries} The proof of this result follows the ideas in \cite{MMPP}, with some minor changes. For that reason, we include as preparative for the more involved proofs in next sections. Let $\theta(t)$ be as in \eqref{theta_1}, and $\chi$ as in \eqref{chichi}. Recall that we consider both cases in \eqref{KP:Eq}, $\kappa=\pm 1$. Consider the functionals
\[
\mathcal L_x(t):= \frac{1}{2}\int_{\mathbb{R}^{2}}u^{2}\chi\left(\frac{x+\theta (t)}{\theta (t)}\right)\,\mathrm{d}x\mathrm{d}y, \quad \mathcal L_y(t):= \frac{1}{2}\int_{\mathbb{R}^{2}}u^{2}\chi\left(\frac{y+\theta (t)}{\theta (t)}\right)\,\mathrm{d}x\mathrm{d}y;
\]
well-defined for solutions of both KP equations. 

\begin{prop}\label{paso4}
For $t\gg1$, $\theta(t)$ as in \eqref{theta_1}, and $\chi$ as in \eqref{chichi}, one has
\begin{equation}\label{virial_Lx}
\frac{\theta '(t)}{2\theta (t)}\int_{\mathbb{R}^{2}}u^{2} (x,y,t)\left| \chi'\left(\frac{x+\theta (t)}{\theta (t)}\right) \right| \left(1- \left(\frac{x+\theta (t)}{\theta (t)}\right) \right) \,\mathrm{d}x\mathrm{d}y \leq -\frac{d}{dt} \mathcal L_x(t) + \mathcal L_{x,int}(t),
\end{equation}
and
\begin{equation}\label{virial_Ly}
\frac{\theta '(t)}{2\theta (t)}\int_{\mathbb{R}^{2}}u^{2}(x,y,t) \left| \chi'\left(\frac{y+\theta (t)}{\theta (t)}\right) \right| \left(1- \left(\frac{y+\theta (t)}{\theta (t)}\right) \right) \,\mathrm{d}x\mathrm{d}y \leq -\frac{d}{dt} \mathcal L_y(t) + \mathcal L_{y,int}(t),
\end{equation}
where both $ \mathcal L_{x,int}(t)$ and $ \mathcal L_{y,int}(t)$ are in $L^1(\{t\gg 1\})$.
\end{prop}

Although \eqref{virial_Lx} and \eqref{virial_Ly} are similar estimates, they require different proofs, because of different terms appearing in the computation of $\frac{d}{dt} \mathcal L_x(t)$ and $\frac{d}{dt} \mathcal L_y(t)$. The key point in \eqref{virial_Lx} and \eqref{virial_Ly} is the fact that $\chi'$ is supported in the interval $[-1,0]$, and $\chi'\leq 0$, which implies that  both terms in the LHS of  \eqref{virial_Lx} and \eqref{virial_Ly} are nonnegative.

\medskip

From the proof it will become clear that similar estimates as \eqref{virial_Lx} and \eqref{virial_Ly} hold for $\theta(t)= -t^{p} \log^{1+\epsilon}t$ in \eqref{theta_1}, after the change $\chi(s)\mapsto \chi(-s)$. We left the details to the interested reader. These new estimates will provide the whole set $\Omega_2(t)$ described in Theorem \ref{KP_decay} eq. \eqref{decayfar}.

\begin{proof}[Proof of Proposition \ref{paso4}]
We first prove \eqref{virial_Lx}. We multiply the equation in \eqref{KP:Eq} by $u\chi\left(\frac{x+\theta (t)}{\theta (t)}\right)$, and after that we integrate in the spatial variables $(x,y)\in \mathbb R^2$  to obtain the identity
	\begin{equation*}
		\begin{aligned}
			\frac{\mathrm{d}}{\mathrm{d}t} \mathcal L_{x}(t) 
			&~{} = \frac{1}{2}\int_{\mathbb{R}^{2}} u^{2}\partial_{t}\left(\chi\left(\frac{x+\theta (t)}{\theta (t)}\right)\right)\,\mathrm{d}x\mathrm{d}y - \int_{\mathbb{R}^{2}}u\partial_{x}^{3}u\chi\left(\frac{x+\theta (t)}{\theta (t)}\right)\,\mathrm{d}x\mathrm{d}y  \\
			&~{}\quad -\kappa \int_{\mathbb{R}^{2}}u\partial_{x}^{-1}\partial_{y}^{2}u\chi\left(\frac{x+\theta (t)}{\theta (t)}\right)\,\mathrm{d}x\mathrm{d}y- \int_{\mathbb{R}^{2}}u^{2}\partial_{x}u\chi\left(\frac{x+\theta (t)}{\theta (t)}\right)\,\mathrm{d}x\mathrm{d}y\\
			&~{} = : A_1(t) +A_2(t) +A_3(t) +A_4(t).
		\end{aligned}
	\end{equation*}
Now, we proceed to  estimate every term separately. In this sense,  for $A_{1}$ we have the following: 
\begin{equation*}
	\begin{split}
		A_{1}(t)&=\frac{1}{2}\int_{\mathbb{R}^{2}} u^{2}\partial_{t}\left(\chi\left(\frac{x+\theta (t)}{\theta (t)}\right)\right)\,\mathrm{d}x\mathrm{d}y\\
		&= \frac{\theta '(t)}{2\theta (t)}\int_{\mathbb{R}^{2}}u^{2}\chi'\left(\frac{x+\theta (t)}{\theta (t)}\right)\,\mathrm{d}x\mathrm{d}y - \frac{\theta '(t)}{2\theta (t)}\int_{\mathbb{R}^{2}}u^{2}\chi'\left(\frac{x+\theta (t)}{\theta (t)}\right)\left(\frac{x+\theta (t)}{\theta (t)}\right)\,\mathrm{d}x\mathrm{d}y\\
	&=A_{1,1}(t)+A_{1,2}(t).
	\end{split}
\end{equation*}
Notice that $\chi'\leq 0,$ $\supp(\chi') \subset (-1,0)$ and \eqref{e2} imply that  $A_{1,1}(t)\leq 0$  and $A_{1,2}(t)\leq 0,$ therefore,  $A_{1}(t)\leq 0$  for all $t\geq0.$ This is exactly the LHS of \eqref{virial_Lx}, up to a minus sign.

\medskip

In the second place,  after applying integration by parts we get 
\begin{equation*}
	\begin{split}
A_{2}(t)&=		-\frac{3}{2\theta (t)}\int_{\mathbb{R}^{2}}\left(\partial_{x}u\right)^{2}\chi'\left(\frac{x+\theta (t)}{\theta (t)}\right)\,\mathrm{d}x\mathrm{d}y+\frac{1}{2\theta ^{3}(t)}\int_{\mathbb{R}^{2}}u^{2}\chi'''\left(\frac{x+\theta (t)}{\theta (t)}\right)\,\mathrm{d}x\mathrm{d}y\\
&=A_{2,1}(t)+A_{2,2}(t).
	\end{split}
\end{equation*}
Combining   \eqref{bounded1} and  \eqref{e3} we obtain
\begin{equation*}
	\begin{split}
	|A_{2,1}(t)|&\lesssim\frac{\|\partial_{x}u(t)\|_{	L^{\infty}_{t}L^{2}_{xy}}^{2}}{\theta (t)}	\lesssim \frac{C}{\theta (t)}\in L^{1}(\{t\gg 1\}).
	\end{split}
\end{equation*}
As for $A_{2,2}$ the $L^{2}-$mass conservation and the integrability of $\theta $  yield
\begin{equation*}
	|A_{2,2}(t)|\lesssim \frac{\|u_{0}\|_{L^{2}_{xy}}^{2}}{\theta ^{3}(t)}\in L^{1}(\{t\gg 1\}).
\end{equation*}  
In the third place,  we combine    integration by parts  and \eqref{bounded1}  from where we obtain 
\begin{equation*}
	\begin{split}
	|A_{3}(t)|\lesssim \frac{1}{\theta (t)}\int_{\mathbb{R}^{2}}\left(\partial_{x}^{-1}\partial_{y}u\right)^{2}\left|\chi'\left(\frac{x+\theta (t)}{\theta (t)}\right)\right|\,\mathrm{d}x\mathrm{d}y\lesssim \frac{\|\partial_{x}^{-1}\partial_{y}u(t)\|_{L^{2}_{xy}}^{2}}{\theta (t)}\lesssim \frac{C}{\theta (t)}\in L^{1}(\{t\gg 1\}).
\end{split}
\end{equation*}
Next,  we handle $A_{4},$ for that   we use interpolation inequality \eqref{Interpol1} that combined with  the integrability of $\theta$ imply  that 
\begin{equation*}
	\begin{split}
		|A_{4}(t)|\leq \frac{1}{3\theta (t)}\int_{\mathbb{R}^{2}}|u|^{3}\left|\chi'\left(\frac{x+\theta (t)}{\theta (t)}\right)\right|\,\mathrm{d}x\mathrm{d}y\lesssim \frac{\|u(t)\|_{L^{\infty}_{t}L^{3}_{xy}}^{3}}{\theta (t)} \lesssim \frac{C}{\theta (t)}\in L^{1}(\{t\gg 1\}).
	\end{split}
\end{equation*}
Finally, we gather the estimates above  to obtain 
\[
\frac{\mathrm{d}}{\mathrm{d}t} \mathcal L_{x}(t)  = \frac{\theta '(t)}{2\theta (t)}\int_{\mathbb{R}^{2}}u^{2} (x,y,t)  \chi'\left(\frac{x+\theta (t)}{\theta (t)}\right) \left(1- \left(\frac{x+\theta (t)}{\theta (t)}\right) \right) \,\mathrm{d}x\mathrm{d}y +\mathcal L_{x,int}(t),
\] 
with $\mathcal L_{x,int}(t)\in L^1(\{t\gg 1\})$, which is nothing but \eqref{virial_Lx}.

\medskip

Now, we prove \eqref{virial_Ly}. Proceeding as before,
 \begin{equation*}
 	\begin{split}
 		\frac{\mathrm{d}}{\mathrm{d}t} \mathcal L_{y}(t)= &~{} \frac{1}{2}\int_{\mathbb{R}^{2}} u^{2}\partial_{t}\left(\chi\left(\frac{y+\theta (t)}{\theta (t)}\right)\right)\,\mathrm{d}x\mathrm{d}y- \int_{\mathbb{R}^{2}}u\partial_{x}^{3}u\chi\left(\frac{y+\theta (t)}{\theta (t)}\right)\,\mathrm{d}x\mathrm{d}y \\
 		& +\kappa\int_{\mathbb{R}^{2}}u\partial_{x}^{-1}\partial_{y}^{2}u\chi\left(\frac{y+\theta (t)}{\theta (t)}\right)\,\mathrm{d}x\mathrm{d}y-\int_{\mathbb{R}^{2}}u^{2}\partial_{x}u\chi\left(\frac{y+\theta (t)}{\theta (t)}\right)\,\mathrm{d}x\mathrm{d}y\\
		=:&~{}B_1(t)+B_2(t)+B_3(t)+B_4(t)  .
 	\end{split}
 \end{equation*}
In this case the localization in the $y-$direction allows us to simplify several terms. More precisely, one has $B_{2}(t)=B_{4}(t)=0.$

\medskip

Next,    we obtain  by using integration by parts that 
\begin{equation*}
	\begin{split}
		B_{1}(t)&= \frac{\theta '(t)}{2\theta (t)}\int_{\mathbb{R}^{2}}u^{2}\chi'\left(\frac{y+\theta (t)}{\theta (t)}\right)\,\mathrm{d}x\mathrm{d}y-\frac{\theta '(t)}{2\theta (t)}\int_{\mathbb{R}^{2}}u^{2}\chi'\left(\frac{y+\theta (t)}{\theta (t)}\right)\left(\frac{y+\theta (t)}{\theta (t)}\right)\,\mathrm{d}x\mathrm{d}y\\
		&= B_{1,1}(t)+B_{1,2}(t).
	\end{split}
\end{equation*}
As we mentioned in the previous case, the terms $B_{1,1}$ and $B_{1,2}$ are positive and these will provide  the information in the LHS of \eqref{virial_Ly}.

\medskip

Finally, as for $B_{3},$ we have that after combining integration by parts and \eqref{bounded1},
\begin{equation*}
	\begin{split}
		|B_{3}(t)|&\leq \frac{1}{\theta (t)} \left| \int_{\mathbb{R}^{2}}u\partial_{x}^{-1}\partial_{y}u\chi'\left(\frac{y+\theta (t)}{\theta (t)}\right)\,\mathrm{d}x\mathrm{d}y \right| \lesssim \frac{C}{\theta (t)}\in L^{1}(\{t\gg 1\}).
	\end{split}
\end{equation*}
Collecting the previous estimates, this ends the proof.
\end{proof}
Now we are ready to conclude the proof of \eqref{decayfar}. First of all, from \eqref{virial_Lx} and \eqref{virial_Ly} we obtain the existence of an increasing sequence $t_n\to +\infty$ such that 
\begin{equation*}
\begin{split}
\lim_{n\to +\infty} \int_{\mathbb{R}^{2}}u^{2} (x,y,t_n)\left| \chi'\left(\frac{x+\theta (t_n)}{\theta (t_n)}\right) \right| \left(1- \left(\frac{x+\theta (t_n)}{\theta (t_n)}\right) \right) \,\mathrm{d}x\mathrm{d}y =0,\,
 \mbox{and}\\
\lim_{n\to +\infty} \int_{\mathbb{R}^{2}}u^{2} (x,y,t_n)\left| \chi'\left(\frac{y+\theta (t_n)}{\theta (t_n)}\right) \right| \left(1- \left(\frac{y+\theta (t_n)}{\theta (t_n)}\right) \right) \,\mathrm{d}x\mathrm{d}y=0.
\end{split}
\end{equation*}
Proceeding as in \cite{MPS,MMPP},  we choose a $C^\infty$ bump function $\xi\geq 0$ with support in $[-\frac34, -\frac14]$, and consider the functionals
\[
\mathcal M_x(t):= \frac{1}{2}\int_{\mathbb{R}^{2}}u^{2}\xi\left(\frac{x+\theta (t)}{\theta (t)}\right)\,\mathrm{d}x\mathrm{d}y, \quad \mathcal M_y(t):= \frac{1}{2}\int_{\mathbb{R}^{2}}u^{2}\xi\left(\frac{y+\theta (t)}{\theta (t)}\right)\,\mathrm{d}x\mathrm{d}y.
\]
 As we did  before, and using that $\xi' \le C |\chi'|$, we get
\[
\begin{aligned}
&~{} \left| \frac{d}{dt}\mathcal M_x(t) \right| +\left| \frac{d}{dt}\mathcal M_y(t) \right|\\
&~{} \quad  \lesssim \frac{1}{t}\int_{\mathbb{R}^{2}}u^{2} (x,y,t) \left( \left| \chi'\left(\frac{x+\theta (t)}{\theta (t)}\right) \right|  + \left| \chi'\left(\frac{y+\theta (t)}{\theta (t)}\right) \right| \right)  \,\mathrm{d}x\mathrm{d}y + \mathcal M_{int}(t),
\end{aligned}
\]
with $ \mathcal M_{int}(t)\in L^1(\{t\gg 1\})$. Integrating between $[t,t_n]$, using Proposition \ref{paso4} and passing to the limit, we conclude.

\begin{rem}
A proof of strong decay for $\partial_x u$ and $\partial_x^{-1} \partial_y u$ as in \eqref{decayfar} is an interesting open problem. In the case of regions around the origin, the answers are provided in \eqref{decay_v} and \eqref{decay_ux}. This is possible thanks to the chain of virial identities that we will prove in forthcoming sections. In the case of the region $\Omega_2(t)$, these arguments need important improvements to hold true.
\end{rem}

\medskip

 \section{Kato Virial estimates in the KP-I case}\label{sect:ux}

\subsection{Preliminaries} Let $\Phi(x,y)$ be the smooth function defined as (see \eqref{psi_phi})
\begin{equation}\label{def_phi3}
\Phi(x,y) := \psi(x) \phi (y).
\end{equation} 
Let $\lambda_1(t),\lambda_2(t),$ and $\eta_{1}(t)$ be the time-dependent functions introduced in Subsection \ref{Scalings}. 
 We consider now the classical mass virial functional \cite{dBM}
 \begin{equation}\label{mK}
 \mathcal K(t) :=\frac{1}{\eta_1(t)}\int_{\mathbb{R}^{2}} u^{2}\Phi\left(\frac{\tilde x}{\lambda_1(t)},\frac{\tilde y}{\lambda_2(t)}\right)\,\mathrm{d}x\mathrm{d}y,
 \end{equation} 
with
\begin{equation}\label{tildexy}
\tilde x:= x-\rho_1(t), \quad \tilde y:= y-\rho_2(t),
\end{equation}
and $\rho_1(t),\rho_2(t)$ defined in \eqref{rho12}.

\medskip

In order to give a first insight on how to show \eqref{decay_ux}, we will prove the following estimate.

\begin{prop}\label{paso1}
Let $u_0 \in  E^1(\mathbb R^2)$. Let $u$ be the corresponding global solution to \eqref{KP:Eq} with initial data $u(t=0)=u_0$ in the KP-I case.  Let $ \mathcal K(t)$ be the functional defined in \eqref{mK}. Then, under the setting of Subsection \ref{Scalings}, one has that $ \mathcal K(t)$ is well-defined and bounded in time, and for a fixed $\sigma_0,C_0>0$,
\begin{equation}\label{estimate_ux}
\begin{aligned}
&~{} \frac{\sigma_0}{t\log t}\int_{\mathbb{R}^{2}} (\partial_x u)^2 \phi \left(\frac{\tilde x}{\lambda_1(t)} \right) \phi\left(\frac{\tilde y}{\lambda_2(t)}\right)   \, \mathrm{d}x\mathrm{d}y  \\
&~{} \qquad  \leq  -\frac{d\mathcal K(t)}{dt} +  \frac{C_0}{t\log t}\int_{\mathbb{R}^{2}} (\partial_x^{-1}\partial_y u)^2   \phi \left(\frac{\tilde x}{\lambda_1(t)} \right) \phi\left(\frac{\tilde y}{\lambda_2(t)}\right)   \, \mathrm{d}x\mathrm{d}y \\
&~{} \qquad \quad  +\frac{C_0}{t\log t} \int_{\mathbb{R}^{2}} |u|^3 \phi \left(\frac{\tilde x}{\lambda_1(t)} \right) \phi\left(\frac{\tilde y}{\lambda_2(t)}\right) \, \mathrm{d}x\mathrm{d}y+ \mathcal K_{int}(t),
\end{aligned}
\end{equation} 
with $ \mathcal K_{int}(t)\in L^1(\{t\gg1\}).$
\end{prop} 

\begin{rem}
We believe that in the KP-II setting, estimate \eqref{estimate_ux} should be better, in the sense that one should have
\begin{equation}\label{estimate_ux_KPII}
\begin{aligned}
&~{} \frac{\sigma_0}{t\log t}\int_{\mathbb{R}^{2}} \left( (\partial_x u)^2 + (\partial_x^{-1}\partial_y u)^2 \right) \phi \left(\frac{\tilde x}{\lambda_1(t)} \right) \phi\left(\frac{\tilde y}{\lambda_2(t)}\right)   \, \mathrm{d}x\mathrm{d}y  \\
&~{} \qquad  \leq  -\frac{d\mathcal K(t)}{dt}  +\frac{C_0}{t\log t} \int_{\mathbb{R}^{2}} |u|^3 \phi \left(\frac{\tilde x}{\lambda_1(t)} \right) \phi\left(\frac{\tilde y}{\lambda_2(t)}\right) \, \mathrm{d}x\mathrm{d}y+ \mathcal K_{int}(t).
\end{aligned}
\end{equation} 
Indications leading to this conjecture are given in de Bouard-Martel \cite{dBM}, where similar estimates are proved in the case of \emph{bounded solutions in $E^1(\mathbb R^2)$}, and where the KP-II Cauchy problem in $E^1(\mathbb R^2)$ was showed globally well-posed. However, in our case some key uniform in time bounds in $E^1(\mathbb R^2)$ are missing. The proof of \eqref{estimate_ux_KPII} is an interesting open problem. 
\end{rem}

Estimate \eqref{estimate_ux} shows that we need two additional estimates to control the remaining terms on the right of \eqref{estimate_ux}, since the classical Kato smoothing estimate fails in the KP case. This will be done in next sections. For the moment, we prove Proposition \ref{paso1}.
 
\subsection{Virial computations} The following computations are somehow classical, see \cite{dBS} and \cite{dBM} for instance. We have, 
 \begin{equation}\label{dtK}
 \begin{split}
& \frac{\mathrm{d}}{\mathrm{d}t}\mathcal K(t)\\
&=\frac{2}{\eta_1(t)}\int_{\mathbb{R}^{2}} u\partial_{t}u \Phi \left(\frac{\tilde x}{\lambda_1(t)},\frac{\tilde y}{\lambda_2(t)}\right)\,\mathrm{d}x\mathrm{d}y-\frac{\eta_1'(t)}{\eta_1^{2}(t)}\int_{\mathbb{R}^{2}} u^{2}\Phi \left(\frac{\tilde x}{\lambda_1(t)},\frac{\tilde y}{\lambda_2(t)}\right)\,\mathrm{d}x\mathrm{d}y\\
&+\frac{1}{\eta_1(t)}\int_{\mathbb{R}^{2}} u^{2}\partial_{t}\left(\Phi \left(\frac{\tilde x}{\lambda_1(t)},\frac{\tilde y}{\lambda_2(t)}\right)\right)\,\mathrm{d}x\mathrm{d}y\\
&=\mathcal K_{1}(t)+\mathcal K_{2}(t)+\mathcal K_{3}(t).
 \end{split}
 \end{equation}
 We easily bound $\mathcal K_{2}(t)$ as follows:
 \begin{equation}\label{K2}
 |\mathcal K_{2}(t)| \lesssim \frac1{t\eta_1(t)} \|u_0\|_{L^2_{xy}}^2 \in L^1(\{t\gg 1\}),
 \end{equation}
 since $\eta_1(t)\gg \log^2 t$.  On the one hand, $\mathcal K_{3}(t)$ is bounded as follows. First of all,
 \[
  \begin{aligned}
\mathcal K_{3}(t) =&~{} \frac{1}{\eta_1(t)}\int_{\mathbb{R}^{2}} u^{2}\partial_{t}\left(\Phi \left(\frac{\tilde x}{\lambda_1(t)},\frac{\tilde y}{\lambda_2(t)}\right)\right)\,\mathrm{d}x\mathrm{d}y\\
=&~{} - \frac{\lambda_1'(t)}{\lambda_1(t)\eta_1(t)}\int_{\mathbb{R}^{2}} u^{2}\left( \frac{\tilde{x}}{\lambda_1(t)}\right) \partial_x \Phi \left(\frac{\tilde x}{\lambda_1(t)},\frac{\tilde y}{\lambda_2(t)}\right)\,\mathrm{d}x\mathrm{d}y\\
&~{}  -\frac{\lambda_2'(t)}{\lambda_2(t)\eta_1(t)}\int_{\mathbb{R}^{2}} u^{2} \left(\frac{\tilde{y}}{\lambda_2(t)}\right)\partial_{y} \Phi \left(\frac{\tilde x}{\lambda_1(t)},\frac{\tilde y}{\lambda_2(t)}\right)\,\mathrm{d}x\mathrm{d}y\\
&~{}  -\frac{\rho_1'(t)}{\lambda_1(t)\eta_1(t)}\int_{\mathbb{R}^{2}} u^{2} \partial_{x} \Phi \left(\frac{\tilde x}{\lambda_1(t)},\frac{\tilde y}{\lambda_2(t)}\right)\,\mathrm{d}x\mathrm{d}y\\
&~{}  -\frac{\rho_2'(t)}{\lambda_2(t)\eta_1(t)}\int_{\mathbb{R}^{2}} u^{2} \partial_{y} \Phi \left(\frac{\tilde x}{\lambda_1(t)},\frac{\tilde y}{\lambda_2(t)}\right)\,\mathrm{d}x\mathrm{d}y\\
= : &~{} \mathcal K_{3,1}(t) +\mathcal K_{3,2}(t) +\mathcal K_{3,3}(t) +\mathcal K_{3,4}(t).
 \end{aligned}
 \]
 The terms $\mathcal K_{3,1}(t)$ and $\mathcal K_{3,2}(t) $ are easily bounded: we have
 \[
 \left| \mathcal K_{3,1}(t) \right| \lesssim \frac{\|u_{0}\|_{L^{2}_{xy}}^{2}}{t \eta_1(t)} \in L^1(\{t\gg 1\}),
 \]
 and the same holds for $\mathcal K_{3,2}(t) $. On the other hand, the term $\mathcal K_{3,3}(t) $ is bounded as follows: from \eqref{def_eta3},
 \[
\left| \mathcal K_{3,3}(t) \right| = \left| \frac{\rho_1'(t)}{\lambda_1(t)\eta_1(t)}\int_{\mathbb{R}^{2}} u^{2} \partial_{x} \Phi \left(\frac{\tilde x}{\lambda_1(t)},\frac{\tilde y}{\lambda_2(t)}\right)\,\mathrm{d}x\mathrm{d}y \right| \lesssim \frac{|\ell_1|m_1\|u_{0}\|_{L^{2}_{xy}}^{2}}{t^{2-m_1} \log t},
 \]
which integrates in time since $m_1<1$ in \eqref{cond_rho1rho2}. Bounding the term $\mathcal K_{3,4}(t) $ follows the same idea. Consequently, 
 \begin{equation}\label{K3}
 \mathcal K_{3}(t) \in L^1(\{t\gg 1\}).
 \end{equation}
 On the other hand, $\mathcal K_{1}(t)$ is treated as follows. Using \eqref{KP:Eq},
 \begin{equation}\label{deco_K1}
 \begin{aligned}
 \mathcal K_{1}(t) =&~{} \frac{2}{\eta_1(t)}\int_{\mathbb{R}^{2}} u\partial_{t}u \Phi \left(\frac{\tilde x}{\lambda_1(t)},\frac{\tilde y}{\lambda_2(t)}\right)\,\mathrm{d}x\mathrm{d}y\\
= &~{} \frac{2}{\eta_1(t)}\int_{\mathbb{R}^{2}} u\left( -\partial_{x}^{3}u -\kappa \partial_{x}^{-1}\partial_{y}^{2}u -u \partial_x u\right) \Phi \left(\frac{\tilde x}{\lambda_1(t)},\frac{\tilde y}{\lambda_2(t)}\right)\,\mathrm{d}x\mathrm{d}y\\
= &~{} \frac{2}{\eta_1(t)}\int_{\mathbb{R}^{2}} \partial_x^2 u \partial_x \left( u \Phi \left(\frac{\tilde x}{\lambda_1(t)},\frac{\tilde y}{\lambda_2(t)}\right)  \right)   \, \mathrm{d}x\mathrm{d}y\\
&~{} + \frac{2 \kappa }{\eta_1(t)}\int_{\mathbb{R}^{2}} \partial_x^{-1}\partial_y u  \partial_y \left( u \Phi \left(\frac{\tilde x}{\lambda_1(t)},\frac{\tilde y}{\lambda_2(t)}\right)  \right)   \, \mathrm{d}x\mathrm{d}y\\
&~{} + \frac{2}{3\eta_1(t) \lambda_1(t) }\int_{\mathbb{R}^{2}} u^3 \partial_x \Phi \left(\frac{\tilde x}{\lambda_1(t)},\frac{\tilde y}{\lambda_2(t)}\right)  \, \mathrm{d}x\mathrm{d}y\\
= : &~{} \mathcal K_{1,1}(t)+\mathcal K_{1,2}(t)+\mathcal K_{1,3}(t).
 \end{aligned}
\end{equation}
As for the term $\mathcal K_{1,1}(t)$, we have
 \[
 \begin{aligned}
\mathcal K_{1,1}(t)= &~{} -\frac{1}{\eta_1(t) \lambda_1(t)}\int_{\mathbb{R}^{2}} (\partial_x u)^2 \left(\partial_x  \Phi\right) \left(\frac{\tilde x}{\lambda_1(t)},\frac{\tilde y}{\lambda_2(t)}\right)    \, \mathrm{d}x\mathrm{d}y \\
&~{} + \frac{2}{\eta_1(t)\lambda_1(t)}\int_{\mathbb{R}^{2}} \partial_x^2 u u \partial_x   \Phi \left(\frac{\tilde x}{\lambda_1(t)},\frac{\tilde y}{\lambda_2(t)}\right)   \, \mathrm{d}x\mathrm{d}y\\
= &~{} -\frac{3}{\eta_1(t) \lambda_1(t)}\int_{\mathbb{R}^{2}} (\partial_x u)^2 \partial_x  \Phi \left(\frac{\tilde x}{\lambda_1(t)},\frac{\tilde y}{\lambda_2(t)}\right)    \, \mathrm{d}x\mathrm{d}y \\
&~{} + \frac{1}{\eta_1(t)\lambda_1^3(t)}\int_{\mathbb{R}^{2}} u^2 \partial_x^3   \Phi \left(\frac{\tilde x}{\lambda_1(t)},\frac{\tilde y}{\lambda_2(t)}\right)   \, \mathrm{d}x\mathrm{d}y.
\end{aligned}
 \]
 The last term above is bounded by
 \[
\lesssim  \frac{1}{\eta_1(t)\lambda_1^3(t)}\|u_0\|_{L^2_{xy}}^2 \in L^1(\{t\gg 1\}),
\]
consequently $\mathcal K_{1,1}(t)$ obeys the decomposition
\begin{equation}\label{K11_final}
\mathcal K_{1,1}(t) = -\frac{3}{\eta_1(t) \lambda_1(t)}\int_{\mathbb{R}^{2}} (\partial_x u)^2 \partial_x  \Phi \left(\frac{\tilde x}{\lambda_1(t)},\frac{\tilde y}{\lambda_2(t)}\right)    \, \mathrm{d}x\mathrm{d}y + \mathcal K_{1,1,int}(t), 
\end{equation}
with $ \mathcal K_{1,1,int}(t)\in L^1(\{t\gg1\}).$ 

\medskip

The term $\mathcal K_{1,2}(t)$ in \eqref{deco_K1} is treated as follows:
\begin{equation*}\label{deco_K12}
 \begin{aligned}
\mathcal K_{1,2}(t) = &~{}   \frac{2\kappa}{\eta_1(t)}\int_{\mathbb{R}^{2}} \partial_x^{-1}\partial_y u  \partial_y \left( u \Phi \left(\frac{\tilde x}{\lambda_1(t)},\frac{\tilde y}{\lambda_2(t)}\right)  \right)   \, \mathrm{d}x\mathrm{d}y\\
= &~{}   \frac{-\kappa}{\eta_1(t) \lambda_1(t)}\int_{\mathbb{R}^{2}}  \left( \partial_x^{-1}\partial_y u\right)^2  \partial_x\Phi \left(\frac{\tilde x}{\lambda_1(t)},\frac{\tilde y}{\lambda_2(t)}\right)   \, \mathrm{d}x\mathrm{d}y\\
&~{} + \frac{2\kappa}{\eta_1(t)\lambda_2(t)}\int_{\mathbb{R}^{2}} u \partial_x^{-1}\partial_y u  \partial_y \Phi \left(\frac{\tilde x}{\lambda_1(t)},\frac{\tilde y}{\lambda_2(t)}\right)    \, \mathrm{d}x\mathrm{d}y.
 \end{aligned}
\end{equation*}
Recall that we are using that $\kappa=-1$, and $u$ is bounded in time in the energy space $ E^1(\mathbb R^2)$. Using \eqref{producto2}, the last term above is treated as
\begin{equation}\label{K12_final_a}
\lesssim \frac{1}{\eta_1(t)\lambda_2(t)} \|u_0\|_{L^2_{xy}} \| \partial_x^{-1} \partial_y u_0\|_{L^2_{xy}} \in L^1(\{t\gg 1\}).
\end{equation}
We have from 
\eqref{comparacion_final}
\[
\begin{aligned}
&~{} \left|  \frac{1}{\eta_1(t) \lambda_1(t)}\int_{\mathbb{R}^{2}} (\partial_x^{-1} \partial_y u)^2   \partial_x\Phi \left(\frac{\tilde x}{\lambda_1(t)},\frac{\tilde y}{\lambda_2(t)}\right)   \, \mathrm{d}x\mathrm{d}y \right| \\
&~{} \qquad   \sim   \frac{1}{t\log t}\int_{\mathbb{R}^{2}} (\partial_x^{-1} \partial_y u)^2   \phi \left(\frac{\tilde x}{\lambda_1(t)} \right) \phi\left(\frac{\tilde y}{\lambda_2(t)}\right)   \, \mathrm{d}x\mathrm{d}y .
\end{aligned}
\]
This last term cannot be estimated properly, unless we have independent estimates. Later, we will prove that (see \eqref{estimate_v}),
\[
\left|  \frac{1}{\eta_1(t) \lambda_1(t)}\int_{\mathbb{R}^{2}} (\partial_x^{-1} \partial_y u)^2   \partial_x\Phi \left(\frac{\tilde x}{\lambda_1(t)},\frac{\tilde y}{\lambda_2(t)}\right)   \, \mathrm{d}x\mathrm{d}y \right|   \in L^1(\{t\gg 1\}),
\]
but for the moment we will save this term for later purposes. Collecting \eqref{K12_final_a}, we conclude
\begin{equation}\label{K12_final}
|\mathcal K_{1,2}(t)| \lesssim \frac{1}{\eta_1(t) \lambda_1(t)}\int_{\mathbb{R}^{2}} (\partial_x^{-1} \partial_y u)^2   \partial_x\Phi \left(\frac{\tilde x}{\lambda_1(t)},\frac{\tilde y}{\lambda_2(t)}\right)   \, \mathrm{d}x\mathrm{d}y+  \mathcal{K}_{1,2,int}(t),
\end{equation}
with $\mathcal K_{1,2,int}(t) \in L^1(\{t\gg1\})$. 

\medskip

Finally, $\mathcal K_{1,3}$ is simply bounded as follows:
\[
\begin{split}
|\mathcal{K}_{1,3}(t)|\leq&\left| \frac{2}{3\eta_1(t) \lambda_1(t) }\int_{\mathbb{R}^{2}} u^3 \partial_x \Phi \left(\frac{\tilde x}{\lambda_1(t)},\frac{\tilde y}{\lambda_2(t)}\right)  \, \mathrm{d}x\mathrm{d}y \right| \\
&~{}  \lesssim  \frac{1}{t\log t} \int_{\mathbb{R}^{2}} |u|^3 \phi \left(\frac{\tilde x}{\lambda_1(t)} \right) \phi\left(\frac{\tilde y}{\lambda_2(t)}\right) \, \mathrm{d}x\mathrm{d}y.
\end{split}
\]
Gathering this last estimate, \eqref{K12_final} and \eqref{K11_final}, we conclude that for some fixed constant $c>0$,
\begin{equation}\label{K1_final}
\begin{aligned}
\mathcal K_1(t) \leq &~{}  -\frac{3}{\eta_1(t) \lambda_1(t)}\int_{\mathbb{R}^{2}} (\partial_x u)^2 \partial_x  \Phi \left(\frac{\tilde x}{\lambda_1(t)},\frac{\tilde y}{\lambda_2(t)}\right)    \, \mathrm{d}x\mathrm{d}y \\
&~{} +  \frac{1}{t\log t}\int_{\mathbb{R}^{2}} (\partial_x^{-1} \partial_y u)^2   \phi \left(\frac{\tilde x}{\lambda_1(t)} \right) \phi\left(\frac{\tilde y}{\lambda_2(t)}\right)   \, \mathrm{d}x\mathrm{d}y \\
&~{}  +\frac{c}{t\log t} \int_{\mathbb{R}^{2}} |u|^3 \phi \left(\frac{\tilde x}{\lambda_1(t)} \right) \phi\left(\frac{\tilde y}{\lambda_2(t)}\right) \, \mathrm{d}x\mathrm{d}y+ \mathcal K_{1,int}(t), 
\end{aligned}
\end{equation}
with $ \mathcal K_{1,int}(t)\in L^1(\{t\gg1\}).$ Coming back to \eqref{dtK}, and collecting \eqref{K2},  \eqref{K3}, and \eqref{K1_final} , we conclude that for some fixed constant $\sigma_0,C_0>0$,
\[
\begin{aligned}
&~{} \frac{\sigma_0}{t\log t}\int_{\mathbb{R}^{2}} (\partial_x u)^2 \phi \left(\frac{\tilde x}{\lambda_1(t)} \right) \phi\left(\frac{\tilde y}{\lambda_2(t)}\right)   \, \mathrm{d}x\mathrm{d}y  \\
&~{} \qquad  \leq  -\frac{d\mathcal K(t)}{dt} +  \frac{C_0}{t\log t}\int_{\mathbb{R}^{2}} (\partial_x^{-1} \partial_y u)^2   \phi \left(\frac{\tilde x}{\lambda_1(t)} \right) \phi\left(\frac{\tilde y}{\lambda_2(t)}\right)   \, \mathrm{d}x\mathrm{d}y \\
&~{} \qquad \quad  +\frac{C_0}{t\log t} \int_{\mathbb{R}^{2}} |u|^3 \phi \left(\frac{\tilde x}{\lambda_1(t)} \right) \phi\left(\frac{\tilde y}{\lambda_2(t)}\right) \, \mathrm{d}x\mathrm{d}y+ \mathcal K_{int}(t),
\end{aligned}
\]
with $ \mathcal K_{int}(t)\in L^1(\{t\gg1\}).$ In the last identity we have also used \eqref{producto2} and \eqref{def_phi3}. This last identity is nothing but \eqref{estimate_ux}. The proof is complete.

\bigskip

\section{Virial estimate for $\partial_x^{-1}\partial_y u$ in the KP-I case}\label{Sect:v}

Proposition \ref{paso1} gives us a control on the derivative $\partial_x u$, but depending on $\partial_x^{-1}\partial_y u$ and $u$ locally in $L^2$ and $L^3$, respectively. Obtaining additional control on these variables for arbitrary large data is key to get a decay property for solutions to KP-I.

\medskip

In order to prove this result, and following a similar argument as in the previous section, we will introduce a new virial identity, in the spirit of \cite{dBM}. In this reference, the authors consider compact solutions in the KP-II case. Here we will use a similar idea to obtain decay in the KP-I case. Some particular complications arise when dealing with solutions in the energy space, leading us to assume more regularity in the data than expected.

\subsection{Preliminaries} 
Recall the functions $\lambda_3(t)$, $\lambda_4(t)$ and $\eta_2(t)$ introduced in \eqref{def_eta2},  
and for $\phi$ and $\psi$ in \eqref{psi_phi}, let
\begin{equation}\label{def_phi2}
\Psi(x,y) := \phi(x) \psi (y).
\end{equation}
Note the difference with the choice in \eqref{def_phi3}. 

\medskip

For $u$ a solution of the KP-I equation in $ E^1(\mathbb R^2)$, we set (as in \cite{dBS}) $v:=\partial_x^{-1}\partial_y u$, for simplicity of notation. Consider the functional 
\begin{equation*}
\mathcal{J}(t):=\frac{1}{\eta_2(t)}\int_{\mathbb{R}^{2}} uv\Psi\left(\frac{\tilde x}{\lambda_{3}(t)},\frac{\tilde y}{\lambda_{4}(t)}
\right)\,\mathrm{d}x\,\mathrm{d}y.
\end{equation*} 
Clearly $\mathcal J$ is well-defined and bounded uniformly in time for data in $ E^1(\mathbb R^2)$, since
\[
\sup_{t\in\mathbb R}\|u(t)\|_{L^2_{xy}} +  \|\partial_x u(t)\|_{L^2_{xy}} +\|v(t)\|_{L^2_{xy}} \leq C,
\]
for data in $E^1(\mathbb R^2)$. However, since we are assuming data in $ E^2(\mathbb R^2)$, we also have from Lemma \ref{boundedE2}
\[
\sup_{t\in\mathbb R}\| \partial_y u(t)\|_{L^2_{xy}} +  \|\partial_x^2 u(t)\|_{L^2_{xy}} +\| \partial_x^{-1}\partial_y v(t)\|_{L^2_{xy}} \leq C.
\]
These bounds are key to the proof of decay. Indeed, we will prove that

\begin{prop}\label{paso2}
Let $u\in E^2(\mathbb R^2)$ be a globally defined solution to KP-I. Then there exist constants $\sigma_1,C_1>0$ such that
\begin{equation}\label{estimate_v}
\begin{aligned}
&~{} \frac{\sigma_1}{t\log t}\int_{\mathbb{R}^{2}} v^{2}\left(\partial_{y}\Psi\right)\left(\frac{\tilde x}{\lambda_{3}(t)},\frac{\tilde y}{\lambda_{4}(t)}
\right)\mathrm{d}x\mathrm{d}y \\
&~{} \qquad \leq   -\frac{d}{dt}\mathcal J(t)  +  \frac{C_1}{t\log t}\int_{\mathbb{R}^{2}} |u|^{3} \phi\left(\frac{\tilde x}{\lambda_{3}(t)} \right) \phi \left(\frac{\tilde y}{\lambda_{4}(t)}\right) \mathrm{d}x\mathrm{d}y\\
&~{} \qquad \quad + \frac{C_1}{t} \int_{\mathbb{R}^{2}}u^2 \phi\left(\frac{\tilde x}{\lambda_{3}(t)} \right) \phi \left(\frac{\tilde y}{\lambda_{4}(t)}\right)\mathrm{d}x\mathrm{d}y  + \mathcal J_{int}(t),
\end{aligned}
\end{equation}
where $\mathcal J_{int}(t)\in L^1(\{ t\gg1\})$. 
\end{prop}

Proving \eqref{estimate_v} for data only in $E^1(\mathbb R^2)$ remains an important open question. The  proof of Proposition \ref{paso2} will be carried out in next subsection.

\subsection{Virial computations}
We estimate the variation in time of the functional $\mathcal{J}.$
In this sense, we have that 
\begin{equation}\label{dJ_calculo}
\begin{split}
\frac{\mathrm{d}}{\mathrm{d}t} \mathcal J(t)= &~{} \frac{\mathrm{d}}{\mathrm{d}t}\left(\frac{1}{\eta_2(t)}\int_{\mathbb{R}^{2}}uv\Psi\left(\frac{\tilde x}{\lambda_{3}(t)},\frac{\tilde y}{\lambda_{4}(t)}
\right)\,\mathrm{d}x\,\mathrm{d}y\right)\\
= &~{} -\frac{\eta_2'(t)}{\eta_2^{2}(t)}\int_{\mathbb{R}^{2}}uv\Psi\left(\frac{\tilde x}{\lambda_{3}(t)},\frac{\tilde y}{\lambda_{4}(t)}
\right)\,\mathrm{d}x\,\mathrm{d}y \\
&~{} +\frac{1}{\eta_2(t)}\int_{\mathbb{R}^{2}}\partial_{t}\left(uv\right)\Psi\left(\frac{\tilde x}{\lambda_{3}(t)},\frac{\tilde y}{\lambda_{4}(t)}
\right)\,\mathrm{d}x\,\mathrm{d}y\\
&+\frac{1}{\eta_2(t)}\int_{\mathbb{R}^{2}}uv \partial_{t}\left(\Psi\left(\frac{\tilde x}{\lambda_{3}(t)},\frac{\tilde y}{\lambda_{4}(t)}\right)\right)\,\mathrm{d}x\,\mathrm{d}y\\
= : & ~{} \mathcal J_{1}(t)+\mathcal J_{2}(t)+\mathcal J_{3}(t).
\end{split}
\end{equation}
In the first place, we have  for $t\gg 1,$ the following: 
\begin{equation}\label{J1_final}
\begin{split}
|\mathcal J_{1}(t)|&=\left| -\frac{\eta_2'(t)}{\eta_2^{2}(t)}\int_{\mathbb{R}^{2}}uv\Psi\left(\frac{\tilde x}{\lambda_{3}(t)},\frac{\tilde y}{\lambda_{4}(t)}
\right)\,\mathrm{d}x\,\mathrm{d}y\right|\\
&\lesssim\frac{1}{t\eta_2(t)}\int_{\mathbb{R}^{2}}|uv|\Psi\left(\frac{\tilde x}{\lambda_{3}(t)},\frac{\tilde y}{\lambda_{4}(t)}
\right)\,\mathrm{d}x\,\mathrm{d}y\\
&\lesssim \frac{1}{t\eta_2(t)}\|u_{0}\|_{L^{2}_{xy}}\|v(t)\|_{L^{2}_{xy}} \lesssim  \frac{1}{t\eta_2(t)} \in L^1(\{t\gg 1\}),
\end{split}
\end{equation}  
because of \eqref{bounded1} and \eqref{def_eta2}. 

\medskip

Next, for $\mathcal J_{2}(t)$ we have that 
\begin{equation}\label{J2}
\begin{split}
\mathcal J_{2}(t)=&~{} \frac{1}{\eta_2(t)}\int_{\mathbb{R}^{2}}\partial_{t}\left(uv\right)\Psi\left(\frac{\tilde x}{\lambda_{3}(t)},\frac{\tilde y}{\lambda_{4}(t)}
\right)\,\mathrm{d}x\,\mathrm{d}y\\
=&~{}  \frac{1}{\eta_2(t)}\int_{\mathbb{R}^{2}}\partial_{t}u v\Psi\left(\frac{\tilde x}{\lambda_{3}(t)},\frac{\tilde y}{\lambda_{4}(t)}
\right)\,\mathrm{d}x\,\mathrm{d}y
\\
&~{} + \frac{1}{\eta_2(t)}\int_{\mathbb{R}^{2}}\partial_{t}v u\Psi\left(\frac{\tilde x}{\lambda_{3}(t)},\frac{\tilde y}{\lambda_{4}(t)}
\right)\,\mathrm{d}x\,\mathrm{d}y\\
=: & ~{} \mathcal J_{2,1}(t)+\mathcal J_{2,2}(t).
\end{split}
\end{equation}
Recall that if $u$ satisfies \eqref{KP:Eq}-I then $v=\partial_{x}^{-1}\partial_{y}u$ formally satisfies the equation
\begin{equation}\label{veq}
\partial_{t}v+\partial_{x}^{3}v-\partial_{x}^{-1}\partial_{y}^{2}v+u\partial_{y}u=0.
\end{equation}
It is clear that (after a density argument), using that $u$ solves \eqref{KP:Eq}-I,
\begin{equation}\label{J21}
\begin{split}
\mathcal J_{2,1}(t)=& \frac{1}{\eta_2(t)}\int_{\mathbb{R}^{2}}\partial_{t}u v\Psi\left(\frac{\tilde x}{\lambda_{3}(t)},\frac{\tilde y}{\lambda_{4}(t)}
\right)\mathrm{d}x\mathrm{d}y\\
=&\frac{1}{\eta_2(t)}\int_{\mathbb{R}^{2}}\left(\partial_{x}^{-1}\partial_{y}^{2}u-\partial_{x}^{3}u-u\partial_{x}u\right) v\Psi\left(\frac{\tilde x}{\lambda_{3}(t)},\frac{\tilde y}{\lambda_{4}(t)}
\right)\mathrm{d}x\mathrm{d}y\\
=& \frac{1}{\eta_2(t)}\int_{\mathbb{R}^{2}}\partial_{x}^{-1}\partial_{y}^{2}uv\Psi\left(\frac{\tilde x}{\lambda_{3}(t)},\frac{\tilde y}{\lambda_{4}(t)}
\right)\,\mathrm{d}x\,\mathrm{d}y \\
&~{} -\frac{1}{\eta_2(t)}\int_{\mathbb{R}^{2}}\partial_{x}^{3}u v\Psi\left(\frac{\tilde x}{\lambda_{3}(t)},\frac{\tilde y}{\lambda_{4}(t)}
\right)\mathrm{d}x\mathrm{d}y\\
&-\frac{1}{\eta_2(t)}\int_{\mathbb{R}^{2}}u\partial_{x}u v\Psi\left(\frac{\tilde x}{\lambda_{3}(t)},\frac{\tilde y}{\lambda_{4}(t)}
\right)\mathrm{d}x\mathrm{d}y\\
=&\mathcal J_{2,1,1}(t)+\mathcal J_{2,1,2}(t)+\mathcal J_{2,1,3}(t).
\end{split}
\end{equation}
For the term $\mathcal J_{2,1,1}(t)$, we have that 
\begin{equation}\label{J211}
\begin{split}
\mathcal J_{2,1,1}(t)=& ~{}\frac{1}{\eta_2(t)}\int_{\mathbb{R}^{2}} \partial_{y}v v\Psi\left(\frac{\tilde x}{\lambda_{3}(t)},\frac{\tilde y}{\lambda_{4}(t)}
\right)\,\mathrm{d}x\,\mathrm{d}y \\
=&~{} -\frac{1}{2\eta_2(t)\lambda_{4}(t)}\int_{\mathbb{R}^{2}}v^{2}\left(\partial_{y}\Psi\right)\left(\frac{\tilde x}{\lambda_{3}(t)},\frac{\tilde y}{\lambda_{4}(t)}
\right)\mathrm{d}x\mathrm{d}y.
\end{split}
\end{equation}
This term has a good sign, and it will be preserved until the end. 

\medskip

Instead, for $\mathcal J_{2,1,2}$, and integrating by parts once, we have 
\[
\begin{split}
\mathcal J_{2,1,2}(t) 
= &~{} \frac{1}{\eta_2(t)}\int_{\mathbb{R}^{2}}\partial_{x}^{2}u  \partial_x \left( v\Psi\left(\frac{\tilde x}{\lambda_{3}(t)},\frac{\tilde y}{\lambda_{4}(t)}
\right) \right)\mathrm{d}x\mathrm{d}y \\
= &~{} \frac{1}{\eta_2(t)}\int_{\mathbb{R}^{2}}\partial_{x}^{2}u  \partial_y u \Psi\left(\frac{\tilde x}{\lambda_{3}(t)},\frac{\tilde y}{\lambda_{4}(t)}
\right) \mathrm{d}x\mathrm{d}y \\
 &~{} + \frac{1}{\eta_2(t)\lambda_3(t)}\int_{\mathbb{R}^{2}}\partial_{x}^{2}u v  (\partial_x \Psi)\left(\frac{\tilde x}{\lambda_{3}(t)},\frac{\tilde y}{\lambda_{4}(t)} \right)\mathrm{d}x\mathrm{d}y.
\end{split}
\]
After another integration,
\[
\begin{split}
\mathcal J_{2,1,2}(t)= &~{} - \frac{1}{\eta_2(t)}\int_{\mathbb{R}^{2}}\partial_{x}u  \partial_{xy} u \Psi\left(\frac{\tilde x}{\lambda_{3}(t)},\frac{\tilde y}{\lambda_{4}(t)}
\right) \mathrm{d}x\mathrm{d}y \\
&~{} - \frac{2}{\eta_2(t)\lambda_3(t)}\int_{\mathbb{R}^{2}}\partial_{x}u  \partial_{y} u (\partial_x\Psi)\left(\frac{\tilde x}{\lambda_{3}(t)},\frac{\tilde y}{\lambda_{4}(t)}
\right) \mathrm{d}x\mathrm{d}y \\
  &~{} - \frac{1}{\eta_2(t)\lambda_3^2(t)}\int_{\mathbb{R}^{2}}\partial_{x} u v  (\partial_x^2 \Psi)\left(\frac{\tilde x}{\lambda_{3}(t)},\frac{\tilde y}{\lambda_{4}(t)} \right)\mathrm{d}x\mathrm{d}y.
\end{split}
\]
Rearranging terms,
\begin{equation}\label{J212}
\begin{split}
\mathcal J_{2,1,2}(t)
= &~{} \frac{1}{2\eta_2(t)\lambda_{4}(t)}\int_{\mathbb{R}^{2}}\left(\partial_{x}u\right)^{2}\left(\partial_{y}\Psi\right)\left(\frac{\tilde x}{\lambda_{3}(t)},\frac{\tilde y}{\lambda_{4}(t)}
\right)\mathrm{d}x\mathrm{d}y\\
&-\frac{2}{\eta_2(t)\lambda_{3}(t)}\int_{\mathbb{R}^{2}}\partial_{x}u\partial_{y}u\left(\partial_{x}\Psi\right)\left(\frac{\tilde x}{\lambda_{3}(t)},\frac{\tilde y}{\lambda_{4}(t)}
\right)\mathrm{d}x\mathrm{d}y\\
&-\frac{1}{2\eta_2(t)\lambda_{3}^{2}(t)\lambda_{4}(t)}\int_{\mathbb{R}^{2}}u^{2}\left(\partial_{x}^{2}\partial_{y}\Psi\right)\left(\frac{\tilde x}{\lambda_{3}(t)},\frac{\tilde y}{\lambda_{4}(t)}
\right)\mathrm{d}x\mathrm{d}y\\
&+\frac{1}{\eta_2(t)\lambda_{3}^{3}(t)}\int_{\mathbb{R}^{2}}uv \left(\partial_{x}^3\Psi\right)\left(\frac{\tilde x}{\lambda_{3}(t)},\frac{\tilde y}{\lambda_{4}(t)}
\right)\mathrm{d}x\mathrm{d}y\\
=:& ~{} \mathcal J_{2,1,2,1}(t) +\mathcal J_{2,1,2,2}(t) +\mathcal J_{2,1,2,3}(t) +\mathcal J_{2,1,2,4}(t).
\end{split}
\end{equation} 
In what follows, we estimate $\mathcal J_{2,1,2,2}(t)$, $\mathcal J_{2,1,2,3}(t)$ and $\mathcal J_{2,1,2,4}(t)$, which are bad terms here. On the other hand, $\mathcal J_{2,1,2,1}(t)$ is a good term, to be saved for later.

\medskip

First of all, using \eqref{condiciones1} and \eqref{bounded2}, and Cauchy-Schwarz,
\begin{equation}\label{J2122}
|\mathcal J_{2,1,2,2}(t)| \lesssim \frac{1}{\eta_2(t)\lambda_{3}(t)} \|\partial_{x}u(t)\|_{L^2_{xy}} \|\partial_{y}u(t)\|_{L^2_{xy}} \lesssim \frac{1}{\eta_2(t)\lambda_{3}(t)}  \in L^1(\{t\gg 1\}).
\end{equation}
Similarly, thanks to the previous estimate and \eqref{bounded1},
\begin{equation}\label{J2123}
|\mathcal J_{2,1,2,3}(t)| \lesssim \frac{1}{\eta_2(t)\lambda_3^{2}(t)\lambda_{4}(t)} \in L^1(\{t\gg 1\}),
\end{equation}
and 
\begin{equation*}\label{J2124}
|\mathcal J_{2,1,2,4}(t)| \lesssim \frac{1}{\eta_2(t)\lambda_{3}^{3}(t)} \in L^1(\{t\gg 1\}).
\end{equation*}
We conclude from \eqref{J212} that
\begin{equation}\label{J212_new}
\begin{split}
\mathcal J_{2,1,2}(t)= &~{} \frac{1}{2\eta_2(t)\lambda_{4}(t)}\int_{\mathbb{R}^{2}}\left(\partial_{x}u\right)^{2}\left(\partial_{y}\Psi\right)\left(\frac{\tilde x}{\lambda_{3}(t)},\frac{\tilde y}{\lambda_{4}(t)}
\right)\mathrm{d}x\mathrm{d}y + \mathcal J_{2,1,2,int}(t),
\end{split}
\end{equation} 
with. $\mathcal J_{2,1,2,int}(t)\in L^1(\{t\gg 1\})$.
 
\medskip 

Coming back to \eqref{J21}, one has
\begin{equation*}
\begin{split}
\mathcal J_{2,1,3}(t)=& ~{} -\frac{1}{\eta_2(t)}\int_{\mathbb{R}^{2}}u\partial_{x}u v\Psi\left(\frac{\tilde x}{\lambda_{3}(t)},\frac{\tilde y}{\lambda_{4}(t)}
\right)\mathrm{d}x\mathrm{d}y\\
= & ~{} \frac{1}{2\eta_2(t)}\int_{\mathbb{R}^{2}}u^2 \partial_x \left( v\Psi\left(\frac{\tilde x}{\lambda_{3}(t)},\frac{\tilde y}{\lambda_{4}(t)}
\right) \right)\mathrm{d}x\mathrm{d}y\\
= & ~{} -\frac{1}{6\eta_2(t)\lambda_{4}(t)}\int_{\mathbb{R}^{2}}u^{3}\left(\partial_{y}\Psi\right)\left(\frac{\tilde x}{\lambda_{3}(t)},\frac{\tilde y}{\lambda_{4}(t)}
\right)\mathrm{d}x\mathrm{d}y\\
&+\frac{1}{2\eta_2(t)\lambda_{3}(t)}\int_{\mathbb{R}^{2}}u^{2}v\left(\partial_{x}\Psi\right)\left(\frac{\tilde x}{\lambda_{3}(t)},\frac{\tilde y}{\lambda_{4}(t)}
\right)\mathrm{d}x\mathrm{d}y.
\end{split}
\end{equation*}
The last term above is actually integrable in time. Indeed, by Cauchy's inequality and \eqref{condiciones1},
\[
\begin{split}
&\left| \frac{1}{2\eta_2(t)\lambda_{3}(t)}\int_{\mathbb{R}^{2}}u^{2}v\left(\partial_{x}\Psi\right)\left(\frac{\tilde x}{\lambda_{3}(t)},\frac{\tilde y}{\lambda_{4}(t)}
\right)\mathrm{d}x\mathrm{d}y \right|\\
 &\leq  \frac{1}{4\eta_2(t)\lambda_{3}(t)} \left\{\int_{\mathbb{R}^{2}}u^{4}\left(\partial_{x}\Psi\right)\left(\frac{\tilde x}{\lambda_{3}(t)},\frac{\tilde y}{\lambda_{4}(t)}
 \right)\mathrm{d}x\mathrm{d}y +\int_{\mathbb{R}^{2}} v^{2}\left(\partial_{x}\Psi\right)\left(\frac{\tilde x}{\lambda_{3}(t)},\frac{\tilde y}{\lambda_{4}(t)}
 \right)\mathrm{d}x\mathrm{d}y  \right\} \\
 \lesssim &~{}  \frac{C}{\eta_2(t)\lambda_{3}(t)} \in L^1(\{t\gg1\}),
\end{split}
\]
thanks to Lemma \ref{Interpol} and \eqref{bounded1}.

Hence, after gathering all the terms we get   
\begin{equation}\label{J21_final}
\begin{split}
\mathcal J_{2,1}(t)
=&~{} -\frac{1}{2\eta_2(t)\lambda_{4}(t)}\int_{\mathbb{R}^{2}} v^{2}\left(\partial_{y}\Psi\right)\left(\frac{\tilde x}{\lambda_{3}(t)},\frac{\tilde y}{\lambda_{4}(t)}
\right)\mathrm{d}x\mathrm{d}y\\
& +\frac{1}{2\eta_2(t)\lambda_{4}(t)}\int_{\mathbb{R}^{2}}\left(\partial_{x}u\right)^{2}\left(\partial_{y}\Psi\right)\left(\frac{\tilde x}{\lambda_{3}(t)},\frac{\tilde y}{\lambda_{4}(t)}\right)\mathrm{d}x\mathrm{d}y\\
&- \frac{1}{6\eta_2(t)\lambda_{4}(t)}\int_{\mathbb{R}^{2}}u^{3}\left(\partial_{y}\Psi\right)\left(\frac{\tilde x}{\lambda_{3}(t)},\frac{\tilde y}{\lambda_{4}(t)}\right)\mathrm{d}x\mathrm{d}y+ \mathcal J_{2,1,int}(t),
\end{split}
\end{equation}
where $\mathcal J_{2,1,int}(t)\in L^1(\{ t\gg1\})$. 

\medskip

Now, we consider the term $\mathcal J_{2,2}(t)$ in \eqref{J2}. By virtue of \eqref{veq} we have that 
\begin{equation}\label{J22}
\begin{split}
\mathcal J_{2,2}(t)= &~{} \frac{1}{\eta_2(t)}\int_{\mathbb{R}^{2}}u\partial_{t}v\Psi\left(\frac{\tilde x}{\lambda_{3}(t)},\frac{\tilde y}{\lambda_{4}(t)}\right)\mathrm{d}x\mathrm{d}y\\
=&~{} \frac{1}{\eta_2(t)}\int_{\mathbb{R}^{2}}u\left(-\partial_{x}^{3}v+\partial_{x}^{-1}\partial_{y}^{2}v-u\partial_{y}u\right)\Psi\left(\frac{\tilde x}{\lambda_{3}(t)},\frac{\tilde y}{\lambda_{4}(t)}\right)\mathrm{d}x\mathrm{d}y\\
=&~{} -\frac{1}{\eta_2(t)}\int_{\mathbb{R}^{2}}u\partial_{x}^{3}v\Psi\left(\frac{\tilde x}{\lambda_{3}(t)},\frac{\tilde y}{\lambda_{4}(t)}\right)\mathrm{d}x\mathrm{d}y\\ 
&+\frac{1}{\eta_2(t)}\int_{\mathbb{R}^{2}}u\partial_{x}^{-1}\partial_{y}^{2}v\Psi\left(\frac{\tilde x}{\lambda_{3}(t)},\frac{\tilde y}{\lambda_{4}(t)}\right)\mathrm{d}x\mathrm{d}y\\
&-\frac{1}{\eta_2(t)}\int_{\mathbb{R}^{2}}u^{2}\partial_{y}u\Psi\left(\frac{\tilde x}{\lambda_{3}(t)},\frac{\tilde y}{\lambda_{4}(t)}\right)\mathrm{d}x\mathrm{d}y\\
=: &~{} \mathcal J_{2,2,1}(t)+\mathcal J_{2,2,2}(t)+\mathcal J_{2,2,3}(t).
\end{split}
\end{equation}
First of all, by integration by parts and using $\partial_x v =\partial_y u$ we obtain
\[
\begin{aligned}
\mathcal J_{2,2,1}(t)= &~{}  \frac{1}{\eta_2(t)}\int_{\mathbb{R}^{2}} \partial_x u\partial_{xy}u\Psi\left(\frac{\tilde x}{\lambda_{3}(t)},\frac{\tilde y}{\lambda_{4}(t)}\right)\mathrm{d}x\mathrm{d}y\\
&~{} +\frac{1}{\eta_2(t)\lambda_3(t)}\int_{\mathbb{R}^{2}}u\partial_{xy}u (\partial_x\Psi)\left(\frac{\tilde x}{\lambda_{3}(t)},\frac{\tilde y}{\lambda_{4}(t)}\right)\mathrm{d}x\mathrm{d}y,
\end{aligned}
\]
so that,
\begin{equation}\label{J221}
\begin{split}
\mathcal J_{2,2,1}(t)&=-\frac{1}{2\eta_2(t)\lambda_{4}(t)}\int_{\mathbb{R}^{2}}\left(\partial_{x}u\right)^{2}\left(\partial_{y}\Psi\right)\left(\frac{\tilde x}{\lambda_{3}(t)},\frac{\tilde y}{\lambda_{4}(t)}\right)\mathrm{d}x\mathrm{d}y\\
&\quad -\frac{1}{\eta_2(t)\lambda_{3}(t)}\int_{\mathbb{R}^{2}}\partial_{y}u\partial_{x}u\left(\partial_{x}\Psi\right)\left(\frac{\tilde x}{\lambda_{3}(t)},\frac{\tilde y}{\lambda_{4}(t)}\right)\mathrm{d}x\mathrm{d}y\\
&\quad +\frac{1}{2\eta_2(t)\lambda_{3}(t)\lambda_{4}^{2}(t)}\int_{\mathbb{R}^{2}}u^{2}\left(\partial_{y}^{2}\partial_{x}\Psi\right)\left(\frac{\tilde x}{\lambda_{3}(t)},\frac{\tilde y}{\lambda_{4}(t)}\right)\mathrm{d}x\mathrm{d}y.
\end{split}
\end{equation}
Performing similar estimates as in \eqref{J2122} and \eqref{J2123}, one can conclude that $\mathcal J_{2,2,1}(t)$ in \eqref{J221} follows the decomposition
\begin{equation}\label{J221_final}
\mathcal J_{2,2,1}(t)=-\frac{1}{2\eta_2(t)\lambda_{4}(t)}\int_{\mathbb{R}^{2}}\left(\partial_{x}u\right)^{2}\left(\partial_{y}\Psi\right)\left(\frac{\tilde x}{\lambda_{3}(t)},\frac{\tilde y}{\lambda_{4}(t)}\right)\mathrm{d}x\mathrm{d}y  + \mathcal J_{2,2,1,int}(t),
\end{equation}
with $\mathcal J_{2,2,1,int}(t)\in L^1(\{t\gg1\}$. Note that the first term above and the one in \eqref{J212_new} cancels out. Additionally,
\begin{equation}\label{J222}
\begin{split}
\mathcal J_{2,2,2}(t)
=&~{} \frac{1}{\eta_2(t)}\int_{\mathbb{R}^{2}}u\partial_{x}^{-1}\partial_{y}^{2}v\Psi\left(\frac{\tilde x}{\lambda_{3}(t)},\frac{\tilde y}{\lambda_{4}(t)}\right)\mathrm{d}x\mathrm{d}y\\
=&~{} -\frac{1}{\eta_2(t)}\int_{\mathbb{R}^{2}}\partial_{y}u\partial_{x}^{-1}\partial_{y}v\Psi\left(\frac{\tilde x}{\lambda_{3}(t)},\frac{\tilde y}{\lambda_{4}(t)}\right)\mathrm{d}x\mathrm{d}y\\
& -\frac{1}{\eta_2(t)\lambda_{4}(t)}\int_{\mathbb{R}^{2}}u\partial_{x}^{-1}\partial_{y}v\left(\partial_{y}\Psi\right)\left(\frac{\tilde x}{\lambda_{3}(t)},\frac{\tilde y}{\lambda_{4}(t)}\right)\mathrm{d}x\mathrm{d}y\\
=:&~{} \mathcal J_{2,2,2,1}(t)+\mathcal J_{2,2,2,2}(t).
\end{split}
\end{equation}
Using that $\partial_{y}u = \partial_x v$, and integrating by parts, the term $\mathcal J_{2,2,2,1}(t)$ is bounded as follows:
\[
\begin{split}
\mathcal J_{2,2,2,1}(t) =&~{} \frac{1}{\eta_2(t)\lambda_3(t)}\int_{\mathbb{R}^{2}}v\partial_{x}^{-1}\partial_{y}v(\partial_x\Psi)\left(\frac{\tilde x}{\lambda_{3}(t)},\frac{\tilde y}{\lambda_{4}(t)}\right)\mathrm{d}x\mathrm{d}y \\
& + \frac{1}{\eta_2(t)}\int_{\mathbb{R}^{2}} v \partial_{y}v\Psi\left(\frac{\tilde x}{\lambda_{3}(t)},\frac{\tilde y}{\lambda_{4}(t)}\right)\mathrm{d}x\mathrm{d}y.
\end{split}
\]
The first term on the RHS above satisfies
\[
\lesssim \frac{1}{\eta_2(t)\lambda_3(t)} \|v(t)\|_{L^2_{xy}} \|\partial_{x}^{-1}\partial_{y}v (t)\|_{L^2_{xy}} \lesssim  \frac{1}{\eta_2(t)\lambda_3(t)} \in L^1(\{t\gg1\}),
\]
thanks to \eqref{bounded2} and \eqref{condiciones1}. We conclude
\begin{equation}\label{J2221}
\begin{split}
\mathcal J_{2,2,2,1}(t) =- \frac{1}{2\eta_2(t) \lambda_4(t)}\int_{\mathbb{R}^{2}} v^2 (\partial_y\Psi)\left(\frac{\tilde x}{\lambda_{3}(t)},\frac{\tilde y}{\lambda_{4}(t)}\right)\mathrm{d}x\mathrm{d}y + \mathcal J_{2,2,2,1,int}(t),
\end{split}
\end{equation}
with $\mathcal J_{2,2,2,1,int}(t)\in L^1(\{t\gg 1\})$. This term adds up to the one in \eqref{J211}. 

\medskip

On the other hand, $\mathcal J_{2,2,2,2}(t)$ in \eqref{J222} is treated as follows. Using Young's inequality,
\begin{equation*}
\begin{split}
|\mathcal J_{2,2,2,2}(t)| \leq &~{}  \frac{\log t}{2\eta_2(t)\lambda_{4}(t)}\int_{\mathbb{R}^{2}}u^2 \left(\partial_{y}\Psi\right)\left(\frac{\tilde x}{\lambda_{3}(t)},\frac{\tilde y}{\lambda_{4}(t)}\right)\mathrm{d}x\mathrm{d}y\\
& + \frac{1}{2\log t \, \eta_2(t)\lambda_{4}(t)}\int_{\mathbb{R}^{2}}(\partial_{y}\partial_{x}^{-1}v)^2 \left(\partial_{y}\Psi\right)\left(\frac{\tilde x}{\lambda_{3}(t)},\frac{\tilde y}{\lambda_{4}(t)}\right)\mathrm{d}x\mathrm{d}y.
\end{split}
\end{equation*}
The second term above integrates in time. Indeed, from \eqref{bounded2} and \eqref{producto},
\[
\begin{aligned}
& \frac{1}{2\log t \, \eta_2(t)\lambda_{4}(t)}\int_{\mathbb{R}^{2}}(\partial_{y}\partial_{x}^{-1}v)^2 \left(\partial_{y}\Psi\right)\left(\frac{\tilde x}{\lambda_{3}(t)},\frac{\tilde y}{\lambda_{4}(t)}\right)\mathrm{d}x\mathrm{d}y \\
& \qquad \lesssim \frac{1}{t\log^2 t}\sup_{t\gg1} \|\partial_{x}^{-1}\partial_{y}v(t)\|_{L^2_{xy}}^2 \in L^1(\{t\gg1\}).
\end{aligned}
\]
Therefore,
\begin{equation}\label{J2222}
\mathcal J_{2,2,2,2}(t) \lesssim  \frac{\log t}{\eta_2(t)\lambda_{4}(t)}\int_{\mathbb{R}^{2}}u^2 \left(\partial_{y}\Psi\right)\left(\frac{\tilde x}{\lambda_{3}(t)},\frac{\tilde y}{\lambda_{4}(t)}\right)\mathrm{d}x\mathrm{d}y +  \mathcal J_{2,2,2,2,int }(t) , 
\end{equation}
with $\mathcal J_{2,2,2,2,int}(t) \in L^1(\{t\gg 1\})$. We conclude from \eqref{J2221} and \eqref{J2222}
\begin{equation}\label{J222_final}
\begin{aligned}
\mathcal J_{2,2,2}(t)  \leq &~{}  - \frac{1}{2\eta_2(t) \lambda_4(t)}\int_{\mathbb{R}^{2}} v^2 (\partial_y\Psi)\left(\frac{\tilde x}{\lambda_{3}(t)},\frac{\tilde y}{\lambda_{4}(t)}\right)\mathrm{d}x\mathrm{d}y \\
&~{} +\frac{c\log t}{\eta_2(t)\lambda_{4}(t)}\int_{\mathbb{R}^{2}}u^2 \left(\partial_{y}\Psi\right)\left(\frac{\tilde x}{\lambda_{3}(t)},\frac{\tilde y}{\lambda_{4}(t)}\right)\mathrm{d}x\mathrm{d}y + \mathcal J_{2,2,2,int}(t),
\end{aligned}
\end{equation}
with $\mathcal J_{2,2,2,int}(t)\in L^1(\{t\gg 1\})$ and some constant $c>0$.

\medskip

Now, we deal with $\mathcal J_{2,2,3}(t)$ in \eqref{J22}. Integrating by parts,
\begin{equation}\label{J223}
\begin{split}
\mathcal J_{2,2,3}(t)&=-\frac{1}{\eta_2(t)}\int_{\mathbb{R}^{2}}u^{2}\partial_{y}u\Psi\left(\frac{\tilde x}{\lambda_{3}(t)},\frac{\tilde y}{\lambda_{4}(t)}\right)\mathrm{d}x\mathrm{d}y\\
&=\frac{1}{3\eta_2(t) \lambda_{4}(t)}\int_{\mathbb{R}^{2}}u^{3}\left(\partial_{y}\Psi\right)\left(\frac{\tilde x}{\lambda_{3}(t)},\frac{\tilde y}{\lambda_{4}(t)}\right)\mathrm{d}x\mathrm{d}y.
\end{split}
\end{equation}
We conclude from \eqref{J221_final}, \eqref{J222_final} and \eqref{J223} that
\begin{equation*}\label{J22_final}
\begin{aligned}
\mathcal J_{2,2}(t)\leq &~{} -\frac{1}{2\eta_2(t)\lambda_{4}(t)}\int_{\mathbb{R}^{2}}\left(\partial_{x}u\right)^{2}\left(\partial_{y}\Psi\right)\left(\frac{\tilde x}{\lambda_{3}(t)},\frac{\tilde y}{\lambda_{4}(t)}\right)\mathrm{d}x\mathrm{d}y  \\
&~{} - \frac{1}{2\eta_2(t) \lambda_4(t)}\int_{\mathbb{R}^{2}} v^2 (\partial_y\Psi)\left(\frac{\tilde x}{\lambda_{3}(t)},\frac{\tilde y}{\lambda_{4}(t)}\right)\mathrm{d}x\mathrm{d}y \\
&~{} + \frac{1}{3\eta_2(t) \lambda_{4}(t)}\int_{\mathbb{R}^{2}}u^{3}\left(\partial_{y}\Psi\right)\left(\frac{\tilde x}{\lambda_{3}(t)},\frac{\tilde y}{\lambda_{4}(t)}\right)\mathrm{d}x\mathrm{d}y  + \mathcal J_{2,2,int}(t),
\end{aligned}
\end{equation*}
with $\mathcal J_{2,2,2,int}(t)\in L^1(\{t\gg 1\})$. Adding this result to the one in \eqref{J21_final} for $\mathcal J_{2,1}(t)$, we conclude
\begin{equation}\label{J2_final}
\begin{split}
\mathcal J_{2}(t)
\leq &~{} -\frac{1}{\eta_2(t)\lambda_{4}(t)}\int_{\mathbb{R}^{2}} v^{2}\left(\partial_{y}\Psi\right)\left(\frac{\tilde x}{\lambda_{3}(t)},\frac{\tilde y}{\lambda_{4}(t)}
\right)\mathrm{d}x\mathrm{d}y\\
&+ \frac{1}{6\eta_2(t)\lambda_{4}(t)}\int_{\mathbb{R}^{2}}u^{3}\left(\partial_{y}\Psi\right)\left(\frac{\tilde x}{\lambda_{3}(t)},\frac{\tilde y}{\lambda_{4}(t)}\right)\mathrm{d}x\mathrm{d}y+ \mathcal J_{2,int}(t),
\end{split}
\end{equation}
where $\mathcal J_{2,int}(t)\in L^1(\{ t\gg1\})$.

\medskip

Finally, we have  that $\mathcal J_{3}(t)$ in \eqref{dJ_calculo} satisfies
\[
\begin{split}
\mathcal J_{3}(t)=&~{} \frac{1}{\eta_2(t)}\int_{\mathbb{R}^{2}}uv \partial_{t}\left(\Psi\left(\frac{\tilde x}{\lambda_{3}(t)},\frac{\tilde y}{\lambda_{4}(t)}\right)\right)\,\mathrm{d}x\,\mathrm{d}y\\
=&~{}  -\frac{1}{\eta_2(t)}\int_{\mathbb{R}^{2}} uv \nabla \Psi\cdot \left(\frac{\tilde x}{\lambda_{3}(t)}\left(\frac{\lambda_{3}'(t)}{\lambda_{3}(t)}\right),\frac{\tilde y}{\lambda_{4}(t)}\left(\frac{\lambda_{4}'(t)}{\lambda_{4}(t)}\right)\right)\,\mathrm{d}x\,\mathrm{d}y\\
&~{} -\frac{\rho_{1}'(t)}{\eta_2(t)\lambda_{3}(t)}\int_{\mathbb{R}^{2}} uv \left(\partial_{x}\Psi\right)\left(\frac{\tilde x}{\lambda_{3}(t)},\frac{\tilde y}{\lambda_{4}(t)}\right)\,\mathrm{d}x\,\mathrm{d}y\\
&~{} -\frac{\rho_{2}'(t)}{\eta_2(t)\lambda_{4}(t)}\int_{\mathbb{R}^{2}} uv \left(\partial_{y}\Psi\right)\left(\frac{\tilde x}{\lambda_{3}(t)},\frac{\tilde y}{\lambda_{4}(t)}\right)\,\mathrm{d}x\,\mathrm{d}y.
\end{split}
\]
Expanding terms,
\[
\begin{split}
\mathcal J_{3}(t)
=&~{} -\frac{\lambda_{3}'(t)}{\eta_2(t)\lambda_{3}(t)}\int_{\mathbb{R}^{2}} uv \left(\partial_{x}\Psi\right)\left(\frac{\tilde x}{\lambda_{3}(t)},\frac{\tilde y}{\lambda_{4}(t)}\right)\left(\frac{\tilde x}{\lambda_{3}(t)}\right)\,\mathrm{d}x\,\mathrm{d}y\\
&~{} -\frac{\lambda_{4}'(t)}{\eta_2(t)\lambda_{4}(t)}\int_{\mathbb{R}^{2}} uv \left(\partial_{y}\Psi\right)\left(\frac{\tilde x}{\lambda_{3}(t)},\frac{\tilde y}{\lambda_{4}(t)}\right)\left(\frac{\tilde y}{\lambda_{4}(t)}\right)\,\mathrm{d}x\,\mathrm{d}y\\
&~{} -\frac{\rho_{1}'(t)}{\eta_2(t)\lambda_{3}(t)}\int_{\mathbb{R}^{2}} uv \left(\partial_{x}\Psi\right)\left(\frac{\tilde x}{\lambda_{3}(t)},\frac{\tilde y}{\lambda_{4}(t)}\right)\,\mathrm{d}x\,\mathrm{d}y\\
&~{} -\frac{\rho_{2}'(t)}{\eta_2(t)\lambda_{4}(t)}\int_{\mathbb{R}^{2}} uv \left(\partial_{y}\Psi\right)\left(\frac{\tilde x}{\lambda_{3}(t)},\frac{\tilde y}{\lambda_{4}(t)}\right)\,\mathrm{d}x\,\mathrm{d}y\\
=:&~{} \mathcal J_{3,1}(t) + \mathcal J_{3,2}(t) + \mathcal J_{3,3}(t)+ \mathcal J_{3,4}(t).
\end{split}
\]
The terms $ \mathcal J_{3,1}(t)$ and $ \mathcal J_{3,2}(t)$ are easy to control. Indeed, we have then from \eqref{bounded1}, \eqref{def_eta2} that combined with  the Cauchy-Schwarz inequality yield
\[
|\mathcal J_{3,1}(t)| \lesssim \frac{1}{t \eta_2(t)} \in L^1(\{ t\gg1\}),
\]
and similar with $\mathcal J_{3,2}(t)$ since $b<1.$ On the other hand, $\mathcal J_{3,3}(t)$ and $\mathcal J_{3,4}(t)$ require some care. Using \eqref{def_eta2} and \eqref{cond_rho1rho2},
\[
\left| \mathcal J_{3,3}(t)\right| =\left|  \frac{\rho_{1}'(t)}{\eta_2(t)\lambda_{3}(t)}\int_{\mathbb{R}^{2}} uv \left(\partial_{x}\Psi\right)\left(\frac{\tilde x}{\lambda_{3}(t)},\frac{\tilde y}{\lambda_{4}(t)}\right)\,\mathrm{d}x\,\mathrm{d}y \right| \lesssim \frac{|\ell_1 m_1| t^{m_1-1}}{t \log^2 t} \in L^1(\{t\gg 1\}), 
\]
since $m_1<1$. The estimate for $\mathcal J_{3,4}(t)$ is completely analogous and we skip it. We conclude
\begin{equation}\label{J3_final}
\mathcal J_{3}(t) \in L^1(\{t\gg1\}).
\end{equation}
Gathering estimates \eqref{J1_final}, \eqref{J2_final}, and \eqref{J3_final}, we conclude in \eqref{dJ_calculo} that 
\[
\begin{aligned}
\frac{d}{dt}\mathcal J(t) \leq &~{} -\frac{1}{\eta_2(t)\lambda_{4}(t)}\int_{\mathbb{R}^{2}} v^{2}\left(\partial_{y}\Psi\right)\left(\frac{\tilde x}{\lambda_{3}(t)},\frac{\tilde y}{\lambda_{4}(t)}
\right)\mathrm{d}x\mathrm{d}y\\
&~{} + \frac{C}{\eta_2(t)\lambda_{4}(t)}\int_{\mathbb{R}^{2}}|u|^{3}\left(\partial_{y}\Psi\right)\left(\frac{\tilde x}{\lambda_{3}(t)},\frac{\tilde y}{\lambda_{4}(t)}\right)\mathrm{d}x\mathrm{d}y \\
&~{} + \frac{C\log t}{\eta_2(t)\lambda_{4}(t)}\int_{\mathbb{R}^{2}}u^2 \left(\partial_{y}\Psi\right)\left(\frac{\tilde x}{\lambda_{3}(t)},\frac{\tilde y}{\lambda_{4}(t)}\right)\mathrm{d}x\mathrm{d}y + \mathcal J_{int}(t),
\end{aligned}
\]
where $\mathcal J_{int}(t)\in L^1(\{ t\gg1\})$. Using the identity (see \eqref{def_phi2})
\begin{equation}\label{derivada}
\left(\partial_{y}\Psi\right)\left(\frac{\tilde x}{\lambda_{3}(t)},\frac{\tilde y}{\lambda_{4}(t)}
\right) = \phi\left(\frac{\tilde x}{\lambda_{3}(t)} \right) \phi \left(\frac{\tilde y}{\lambda_{4}(t)}\right),
\end{equation}
and the fact that $\eta_2(t)\lambda_{4}(t)=t\log t$, we conclude \eqref{estimate_v}.

\bigskip

\section{$L^2$ decay in KP-I and KP-II}\label{Sect:3}

This section is devoted to the proof of Theorem \ref{KP_decay} and by extension the proof of Theorem \ref{thm_KPI}, which will be proved in next section. We consider both cases KP-I and KP-II, since the proof does not depend on the sign in the equation \eqref{KP:Eq}, only on the well-posedness theory available for each model.

\medskip

Theorem \ref{KP_decay} will be a consequence of the following integrability result. Recall the set $\Omega_1(t)$ already introduced in \eqref{Omega}:
\begin{equation*}\label{Omega_new}
\Omega_1(t)= \left\{ (x,y)\in\mathbb R^2 ~: ~  |x-\ell_1 t^{m_1}| \leq t^{b}, \; |y-\ell_2 t^{m_2}|\leq t^{br}\right\},
\end{equation*}
with $\frac53<r<3$, $0<b <\frac2{3+r}$, $0\leq m_1<1-\frac{b}2(r+1)$, and $0\leq m_2 < 1-\frac12b(q+2-r)$.

\begin{thm}\label{lem:d2L2}
Assume that $u_0\in L^2(\mathbb R^2)$ in the KP-II case, and $u_0\in E^1(\mathbb R^2)$ in the KP-I case. Let $u=u(x,y,t)$ be the corresponding unique solution of the IVP \eqref{KP:Eq} with $\kappa=\pm1$. Then, there exists a constant $c>0,$ such that
\begin{equation*}\label{IntegraL2}
\begin{split}
&\int_{\{t\gg1\}} \frac{1}{t}	
\left(\int_{\Omega_1(t)} u^{2}(x,y,t) \,\mathrm{d}x\mathrm{d}y\right)\mathrm{d}t \leq c.
\end{split}
\end{equation*}
\end{thm}
This bound, and a very similar argument to the one performed in \cite{MuPo2}, allows to conclude Theorem \ref{KP_decay} \eqref{L2} easily. We skip the details. 

\subsection{Setting} Recall the weighted functions $\psi$ and $\phi$ defined in \eqref{psi_phi}.

\medskip

Recall the functions $\psi$ and $\phi$ introduced in Subsection \ref{Preli1}. For $u$  a solution of  the  Kadomtsev-Petviashvili  equation \eqref{KP:Eq}, we set the functional
\begin{equation}\label{mI}
\mathcal I(t):=\frac{1}{\eta_3(t)}\int_{\mathbb{R}^{2}}u(x,y,t)\psi\left(\frac{\tilde x}{\lambda_5(t)}\right)\phi\left(\frac{\tilde x}{\lambda_5^q(t)}\right)\phi\left(\frac{\tilde y}{\lambda_6(t)}\right)\,\mathrm{d}x\mathrm{d}y,
\end{equation}
where $\tilde x,\tilde y$ were defined in \eqref{tildexy}, and for $t\gg 1$, $\lambda_5(t)$, $\lambda_6(t)$ and $\eta_3(t)$ were introduced in Subsection \ref{Scalings}.

\begin{lem}\label{bounded2d}
For $u\in L^{2}(\mathbb{R}^{2})$, the functional $\mathcal I $ is  well defined and bounded in time.
\end{lem}
\begin{proof}
By Cauchy-Schwarz inequality we obtain
\begin{equation}\label{eq1}
\begin{split}
|\mathcal I (t)|&\leq \frac{1}{\eta_3(t)}\|u(t)\|_{L^{2}_{x,y}}\left\|\psi\right\|_{L^{\infty}_{x}}\left\|\phi\left(\frac{\cdot}{\lambda_5^q(t)}\right)\phi\left(\frac{\cdot}{\lambda_6(t)}\right)\right\|_{L^{2}_{x,y}}\\
&=\frac{\left(\lambda_5^q(t)\lambda_6(t)\right)^{1/2}}{\eta_3(t)}\|u_{0}\|_{L^{2}_{xy}}\left\|\psi\right\|_{L^{\infty}_{x}}\left\|\phi\right\|_{L^{2}_{y}}\left\|\phi\right\|_{L^{2}_{x}}\\
&\lesssim \left(\frac{1}{\log ^{(2+q+r)/2}t}\right)\frac{\|u_{0}\|_{L^{2}_{xy}}}{t^{(2-b(2+q+r))/2}}.
\end{split}
\end{equation}
Since \eqref{bqconditions_new} is satisfied we have
\[
\sup_{t\gg 1} |\mathcal I (t)| <\infty,
\]
which finishes the proof.
\end{proof}

\subsection{Dynamics for $\mathcal I (t)$}  In what follows, we compute and estimate the dynamics of $\mathcal I (t)$ in the long time regime.

\begin{prop}\label{le:dT}
There exists $\sigma_2>0$ and $\varepsilon_0>0$ small enough such that, for any $t\geq 2$, one has the bound
\begin{equation}\label{dT_ppal}
\begin{split}
\frac{\sigma_2}{t}\int_{\mathbb{R}^{2}} u^{2} \phi\left(\frac{\tilde x}{\lambda_5(t)}\right)\phi\left(\frac{\tilde y}{\lambda_6(t)}\right)\,\mathrm{d}x\mathrm{d}y &\leq \frac{d\mathcal I }{dt}(t) + \mathcal I _{int}(t),
\end{split}
\end{equation}
provided $q=1+\varepsilon_0$ in $\mathcal I(t)$ \eqref{mI}, and where $\mathcal I _{int}(t)$ are terms that belong to $L^1\left(\{t\gg 1\}\right)$.
\end{prop}

Assuming this estimate and Lemma \ref{bounded2d}, one concludes Theorem \ref{lem:d2L2} as in \cite{MuPo2}. The rest of the Section will be devoted to the proof of Proposition \ref{le:dT}. We have
\begin{equation}\label{separation}
\begin{split}
\frac{\mathrm{d}}{\mathrm{d}t}\mathcal I (t)&=\frac{1}{\eta_3(t)}\int_{\mathbb{R}^{2}}\partial_{t}\left(u\psi\left(\frac{\tilde x}{\lambda_5(t)}\right)\phi\left(\frac{\tilde x}{\lambda_5^q(t)}\right)\phi\left(\frac{\tilde y}{\lambda_6(t)}\right)\right)\,\mathrm{d}x\mathrm{d}y\\
&\quad -\frac{\eta_3'(t)}{\eta_3^{2}(t)}\int_{\mathbb{R}^{2}}u\psi\left(\frac{\tilde x}{\lambda_5(t)}\right)\phi\left(\frac{\tilde x}{\lambda_5^q(t)}\right)\phi\left(\frac{\tilde y}{\lambda_6(t)}\right)\,\mathrm{d}x\mathrm{d}y\\
&=: \mathcal I _{1}(t)+ \mathcal I _{2}(t).
\end{split}
\end{equation}
First, we bound  $\mathcal I _{2},$ that in virtue of \eqref{eq1} the same analysis applied there  yield
\begin{equation}\label{2}
\begin{split}
|\mathcal I _{2}(t)|&\leq\left|\frac{\eta_3'(t)}{\eta_3^{2}(t)}\int_{\mathbb{R}^{2}}u\psi\left(\frac{\tilde x}{\lambda_5(t)}\right)\phi\left(\frac{\tilde x}{\lambda^q_{5}(t)}\right)\phi\left(\frac{\tilde y}{\lambda_6(t)}\right)\,\mathrm{d}x\mathrm{d}y\right|\\
&\lesssim\|u_{0}\|_{L^{2}_{xy}}\frac{\left(\lambda^q_{5}(t)\lambda_6(t)\right)^{1/2}|\eta_3'(t)|}{\eta_3^{2}(t)}\\
&\lesssim \frac{1}{t}\frac{1}{\eta_3(t)}\frac{1}{(\lambda_5(t))^{-(q+r)/2}} =\frac{1}{t^{2-\frac b2(2+q+r)}\log^{1+\frac12(q+r)} t}.
\end{split}
\end{equation}
From \eqref{defns} and \eqref{bqconditions_new} we have $b\le\frac{2}{q+r+2}$, then  $\mathcal I _{2}\in L^1(t\gg1)$.

\medskip

Unlike $\mathcal I _{2}$, to bound  $\mathcal I _{1}$ it is required   to take into consideration the dispersive part associated to the KP equation.   More precisely,
\begin{equation}\label{separation_Xi}
\begin{split}
\mathcal I _{1}(t)&=\frac{1}{\eta_3(t)}\int_{\mathbb{R}^{2}}\partial_{t}u\psi\left(\frac{\tilde x}{\lambda_5(t)}\right)\phi\left(\frac{\tilde x}{\lambda_5^q(t)}\right)\phi\left(\frac{\tilde y}{\lambda_6(t)}\right)\,\mathrm{d}x\mathrm{d}y\\
&\quad -\frac{\lambda_{5}'(t)}{\lambda_5(t)\eta_3(t)}\int_{\mathbb{R}^{2}}u \psi'\left(\frac{\tilde x}{\lambda_5(t)}\right)\left(\frac{\tilde x}{\lambda_5(t)}\right)\phi\left(\frac{\tilde x}{\lambda_5^q(t)}\right)\phi\left(\frac{\tilde y}{\lambda_6(t)}\right)\,\mathrm{d}x\mathrm{d}y\\
&\quad -q\frac{\lambda_{5}'(t)}{\lambda_5(t)\eta_3(t)}\int_{\mathbb{R}^{2}}u \psi\left(\frac{\tilde x}{\lambda_5(t)}\right)\left(\frac{\tilde x}{\lambda_5^q(t)}\right)\phi'\left(\frac{\tilde x}{\lambda_5^q(t)}\right)\phi\left(\frac{\tilde y}{\lambda_6(t)}\right)\,\mathrm{d}x\mathrm{d}y\\
&\quad -\frac{\lambda_{6}'(t)}{\lambda_6(t)\eta_3(t)}\int_{\mathbb{R}^{2}}u \psi\left(\frac{\tilde x}{\lambda_5(t)}\right)\phi\left(\frac{\tilde x}{\lambda_5^q(t)}\right)\left(\frac{\tilde y}{\lambda_6(t)}\right)\phi'\left(\frac{\tilde y}{\lambda_6(t)}\right)\,\mathrm{d}x\mathrm{d}y\\
&\quad - \frac{\rho_1'(t)}{\lambda_5(t)\eta_3(t)}\int_{\mathbb{R}^{2}}u\psi'\left(\frac{\tilde x}{\lambda_5(t)}\right)\phi\left(\frac{\tilde x}{\lambda_5^q(t)}\right)\phi\left(\frac{\tilde y}{\lambda_6(t)}\right)\,\mathrm{d}x\mathrm{d}y\\
&\quad - \frac{\rho_1'(t)}{\lambda_1^q(t)\eta_3(t)}\int_{\mathbb{R}^{2}}u\psi\left(\frac{\tilde x}{\lambda_5(t)}\right)\phi'\left(\frac{\tilde x}{\lambda_5^q(t)}\right)\phi\left(\frac{\tilde y}{\lambda_6(t)}\right)\,\mathrm{d}x\mathrm{d}y\\
&\quad - \frac{\rho_2'(t)}{\lambda_6(t)\eta_3(t)}\int_{\mathbb{R}^{2}}u\psi\left(\frac{\tilde x}{\lambda_5(t)}\right)\phi\left(\frac{\tilde x}{\lambda_5^q(t)}\right)\phi'\left(\frac{\tilde y}{\lambda_6(t)}\right)\,\mathrm{d}x\mathrm{d}y\\
&=: \mathcal I _{1,1}(t)+\mathcal I _{1,2}(t)+\mathcal I _{1,3}(t)+\mathcal I _{1,4}(t) +\mathcal I _{1,5}(t) +\mathcal I _{1,6}(t) +\mathcal I _{1,7}(t).
\end{split}
\end{equation}
Concerning to $\mathcal I _{1,1}$ we have by \eqref{KP:Eq}  and integration by  parts
\begin{equation}\label{separation_Xi1}
\begin{split}
\mathcal I _{1,1}(t)&=-\frac{1}{\eta_3(t)}\int_{\mathbb{R}^{2}}\left(\partial_{x}^{3} u+\kappa\partial_{x}^{-1}\partial_{y}^{2}u+u\partial_{x}u\right)\psi\left(\frac{\tilde x}{\lambda_5(t)}\right)\phi\left(\frac{\tilde x}{\lambda^q_{5}(t)}\right)\phi\left(\frac{\tilde y}{\lambda_6(t)}\right)\,\mathrm{d}x\mathrm{d}y\\
&=\frac{1}{\eta_3(t)}\int_{\mathbb{R}^{2}} \partial_{x}^{3} u\psi\left(\frac{\tilde x}{\lambda_5(t)}\right)\phi\left(\frac{\tilde x}{\lambda^q_{5}(t)}\right)\phi\left(\frac{\tilde y}{\lambda_6(t)}\right)\,\mathrm{d}x\mathrm{d}y\\
&\quad +\frac{\kappa}{\eta_3(t)}\int_{\mathbb{R}^{2}} \partial_{x}^{-1}\partial_{y}^{2} u\psi\left(\frac{\tilde x}{\lambda_5(t)}\right)\phi\left(\frac{\tilde x}{\lambda^q_{5}(t)}\right)\phi\left(\frac{\tilde y}{\lambda_6(t)}\right)\,\mathrm{d}x\mathrm{d}y\\
&\quad +\frac{1}{2\eta_3(t)\lambda_5(t)}\int_{\mathbb{R}^{2}}  u^{2}\psi'\left(\frac{\tilde x}{\lambda_5(t)}\right)\phi\left(\frac{\tilde x}{\lambda_5^q(t)}\right)\phi\left(\frac{\tilde y}{\lambda_6(t)}\right)\,\mathrm{d}x\mathrm{d}y\\
&\quad +\frac{1}{2\eta_3(t)\lambda^q_{5}(t)}\int_{\mathbb{R}^{2}}  u^{2}\psi\left(\frac{\tilde x}{\lambda_5(t)}\right)\phi'\left(\frac{\tilde x}{\lambda^q_{5}(t)}\right)\phi\left(\frac{\tilde y}{\lambda_6(t)}\right)\,\mathrm{d}x\mathrm{d}y\\
&=: \mathcal I _{1,1,1}(t)+\mathcal I _{1,1,2}(t)+\mathcal I _{1,1,3}(t)+ \mathcal I _{1,1,4}(t).
\end{split}
\end{equation}
For  $\mathcal I _{1,1,1}$ we have  after combining  integration by parts
\begin{equation*}
\begin{split}
\mathcal I _{1,1,1}(t)&=\frac{1}{\eta_3(t)\lambda_{5}^{3}(t)}\int_{\mathbb{R}^{2}}u\psi'''
\left(\frac{\tilde x}{\lambda_5(t)}\right)\phi\left(\frac{\tilde x}{\lambda^q_{5}(t)}\right)\phi\left(\frac{\tilde y}{\lambda_6(t)}\right)\,\mathrm{d}x\mathrm{d}y\\
&\quad +\frac{3}{\eta_3(t)\lambda_{5}^{2+q}(t)}\int_{\mathbb{R}^{2}}u\psi''\left(\frac{\tilde x}{\lambda_5(t)}\right)\phi'\left(\frac{\tilde x}{\lambda^q_{5}(t)}\right)\phi\left(\frac{\tilde y}{\lambda_6(t)}\right)\mathrm{d}x\mathrm{d}y\\
&\quad +\frac{3}{\eta_3(t)\lambda_{5}^{1+2q}(t)}\int_{\mathbb{R}^{2}}u\psi'\left(\frac{\tilde x}{\lambda_5(t)}\right)\phi''\left(\frac{\tilde x}{\lambda^q_{5}(t)}\right)\phi\left(\frac{\tilde y}{\lambda_6(t)}\right)\mathrm{d}x\mathrm{d}y\\
&\quad +\frac{1}{\eta_3(t)\lambda_{5}^{3q}(t)}\int_{\mathbb{R}^{2}}u\psi\left(\frac{\tilde x}{\lambda_5(t)}\right)\phi'''\left(\frac{\tilde x}{\lambda^q_{5}(t)}\right)\phi\left(\frac{\tilde y}{\lambda_6(t)}\right)\mathrm{d}x\mathrm{d}y.
\end{split}
\end{equation*}
First, we bound each term using Cauchy-Schwarz inequality, as follows:
\[
  \begin{aligned}
    & |\mathcal I _{1,1,1}(t)| \\
    & \lesssim \frac{(\lambda_5(t)\lambda_6(t))^{1/2}}{\eta_3(t)\lambda_5^3(t)}\|u_0\|_{L^2_{xy}}\|\psi^{\prime\prime\prime}\|_{L^2_x}\|\phi\|_{L^\infty_{x}}\|\phi\|_{L^2_{y}}+\frac{(\lambda_5(t)\lambda_6(t))^{1/2}}{\eta_3(t)\lambda_5^{2+q}(t)}\|u_0\|_{L^2_{xy}}\|\psi^{\prime\prime}\|_{L^2_x}\|\phi^\prime\|_{L^\infty_{x}}\|\phi\|_{L^2_{y}}\\
    &+\frac{(\lambda_5(t)\lambda_6(t))^{1/2}}{\eta_3(t)\lambda_5^{1+2q}(t)}\|u_0\|_{L^2_{xy}}\|\psi^{\prime}\|_{L^2_x}\|\phi^{\prime\prime}\|_{L^\infty_{x}}\|\phi\|_{L^2_{y}}+\frac{(\lambda_5^{q}(t)\lambda_6(t))^{1/2}}{\eta_3(t)\lambda_5^{3q}(t)}\|u_0\|_{L^2_{xy}}\|\psi\|_{L^{\infty}_x}\|\phi\|_{L^2_{x}}\|\phi^{\prime\prime}\|_{L^2_{y}}.
  \end{aligned}
\]
Consequently, replacing \eqref{eq5}, \eqref{eq4} and \eqref{defns}, and since $q>1$,
\begin{equation*}\label{111}
  \begin{aligned}
    |\mathcal I _{1,1,1}(t)|    \lesssim&~ {} \frac{1}{\eta_3(t)(\lambda_5(t))^{5/2-r/2}}+\frac{1}{\eta_3(t)(\lambda_5(t))^{3/2+q-r/2}}\\
    &+\frac{1}{\eta_3(t)(\lambda_5(t))^{1/2+2q-r/2}}+\frac{1}{\eta_3(t)(\lambda_5(t))^{5q/2-r/2}}\\
    \lesssim &~{}  \frac{1}{\eta_3(t)(\lambda_5(t))^{5/2-r/2}} = \frac{1}{t^{1 +\frac b2(3-r) }\log^{1-\frac b2(5-r)} t}.
  \end{aligned}
\end{equation*}
From \eqref{bqconditions_new} we easily see that $\mathcal I _{1,1,1}\in L^1(\{t\gg1\})$. 

\medskip

To bound the term $\mathcal I _{1,1,2}(t)$ we will require Lemma \ref{Localizacion} as key step. Combining integration by parts, \eqref{bound} and Cauchy-Schwarz inequality, 
\[
\begin{split}
\mathcal I _{1,1,2}(t) &=\frac{\kappa}{\eta_3(t)}\int_{\mathbb{R}^{2}} \partial_{x}^{-1}\partial_{y}^{2} u\psi\left(\frac{\tilde x}{\lambda_5(t)}\right)\phi\left(\frac{\tilde x}{\lambda_5^q(t)}\right)\phi\left(\frac{\tilde y}{\lambda_6(t)}\right)\,\mathrm{d}x\mathrm{d}y\\
&=-\frac{\kappa}{\eta_3(t)\lambda_{6}^{2}(t)}\int_{\mathbb{R}^{2}}u\partial_{x}^{-1}\left(\psi\left(\frac{\tilde x}{\lambda_5(t)}\right)\phi\left(\frac{\tilde x}{\lambda_5^q(t)}\right)\right)\phi''\left(\frac{\tilde y}{\lambda_6(t)}\right)\,\mathrm{d}x\mathrm{d}y,
\end{split}
\]
so that,
\begin{equation*}\label{X112}
\begin{split}
|\mathcal I _{1,1,2}(t)| &\leq \frac{|\kappa|}{\eta_3(t)\lambda_{6}^{2}(t)}\int_{\mathbb{R}^{2}}|u|\left|\partial_{x}^{-1}\left(\psi\left(\frac{\tilde x}{\lambda_5(t)}\right)\phi\left(\frac{\tilde x}{\lambda_5^q(t)}\right)\right)\phi''\left(\frac{\tilde y}{\lambda_6(t)}\right)\right|\,\mathrm{d}x\mathrm{d}y\\
&\lesssim \frac{\lambda_{5}^{q}(t)}{\eta_3(t)\lambda_{6}^{2}(t)}\int_{\mathbb{R}^{2}}|u|\phi\left(\frac{\tilde x}{\lambda_5^q(t)}\right)\left|\phi''\left(\frac{\tilde y}{\lambda_6(t)}\right)\right|\,\mathrm{d}x\mathrm{d}y\\
&\lesssim \frac{\lambda_{5}^{3q/2}(t)}{\eta_3(t)\lambda_{6}^{3/2}(t)}\|u_{0}\|_{L^{2}_{xy}}\|\phi\|_{L^{2}_{x}}\|\phi\|_{L^{2}_{y}}\\
& \sim \frac{1}{\eta_3(t)\lambda_{5}^{\frac32(r-q)}(t)} \\
& = \frac{1}{t^{1+\frac32b(r-q-\frac23)} \log^{1-\frac32(r-q)} t}.
		\end{split}
\end{equation*}
Since $r>5/3,$ taking $q=1+\varepsilon_0$, with $\varepsilon_0>0$ sufficiently small, we obtain that  $\mathcal I _{1,1,2}\in L^1(\{t\gg1\}).$


\medskip

We emphasize that the term $\mathcal I _{1,1,3}$ in \eqref{separation_Xi1}
\begin{equation}\label{Xi113}
\mathcal I _{1,1,3}(t)= \frac{1}{2\eta_3(t)\lambda_5(t)}\int_{\mathbb{R}^{2}}  u^{2}\psi'\left(\frac{\tilde x}{\lambda_5(t)}\right)\phi\left(\frac{\tilde x}{\lambda_5^q(t)}\right)\phi\left(\frac{\tilde y}{\lambda_6(t)}\right)\,\mathrm{d}x\mathrm{d}y,
\end{equation}
is the term to be estimated after integrating in time, leading to the left hand side in \eqref{dT_ppal}. Therefore, it will be taken until the end of the proof.

\medskip

The therm $\mathcal I _{1,1,4}$ in \eqref{separation_Xi1} satisfies de following estimate
\begin{equation*}\label{11}
\begin{split}
|\mathcal I _{1,1,4}(t)|&\leq\left|\frac{1}{2\eta_3(t)\lambda_5^q(t)}\int_{\mathbb{R}^{2}}  u^{2}\psi\left(\frac{\tilde x}{\lambda_5(t)}\right)\phi'\left(\frac{\tilde x}{\lambda_5^q(t)}\right)\phi\left(\frac{\tilde y}{\lambda_6(t)}\right)\,\mathrm{d}x\mathrm{d}y\right|\\
&\le\frac{1}{2\eta_3(t)\lambda_5^q(t)}\|u_0\|^2_{L^2_{xy}}\|\psi\|_{L_x^\infty}\|\phi\|_{L_x^\infty}\|\phi\|_{L_y^\infty}\\
& \sim \frac{1}{t^{1 +(q-1)b} \log^{1-q} t}.
\end{split}
\end{equation*}
Since \eqref{bqconditions_new} are satisfied, and $q=1+\varepsilon_0>1$, $\mathcal I _{1,1,4}\in L^1(t\gg1)$. 

\medskip

Now we deal with the term $\mathcal I _{1,5}$ in \eqref{separation_Xi}. We have
\[
\begin{aligned} 
& \left| \frac{\rho_1'(t)}{\lambda_5(t)\eta_3(t)}\int_{\mathbb{R}^{2}}u\psi'\left(\frac{\tilde x}{\lambda_5(t)}\right)\phi\left(\frac{\tilde x}{\lambda_5^q(t)}\right)\phi\left(\frac{\tilde y}{\lambda_6(t)}\right)\,\mathrm{d}x\mathrm{d}y \right| \\
&~{} \lesssim \frac{|\rho_1'(t)| \lambda_6^{1/2}(t) \|u_0\|_{L^2_{xy}}}{\lambda_5^{1/2}(t)\eta_3(t)} \lesssim \frac{|\ell_1m_1|}{t^{2 -\frac12b(r+1)-m_1} \log^{1+\frac12(r-1)} t},
\end{aligned}
\]
which integrates in time because of \eqref{cond_rho1rho2}. The term $\mathcal I _{1,6}$ is treated very similarly: one has for $t\gg 1$,
 \[
 \begin{aligned} 
& \left| \frac{\rho_1'(t)}{\lambda_5^q(t)\eta_3(t)}\int_{\mathbb{R}^{2}}u\psi\left(\frac{\tilde x}{\lambda_5(t)}\right)\phi'\left(\frac{\tilde x}{\lambda_5^q(t)}\right)\phi\left(\frac{\tilde y}{\lambda_6(t)}\right)\,\mathrm{d}x\mathrm{d}y\right|\\
&~{} \lesssim \frac{|\rho_1'(t)| \lambda_6^{1/2}(t) \|u_0\|_{L^2_{xy}}}{\lambda_5^{q/2}(t)\eta_3(t)}\\
&\lesssim  \frac{|\rho_1'(t)| \lambda_6^{1/2}(t) \|u_0\|_{L^2_{xy}}}{\lambda_5^{1/2}(t)\eta_3(t)},
\end{aligned} 
\] 
since $q>1$. Finally, the term $\mathcal I _{1,7}$ can be treated as follows:
\[
\begin{aligned}
& \left|  \frac{\rho_2'(t)}{\lambda_6(t)\eta_3(t)}\int_{\mathbb{R}^{2}}u\psi\left(\frac{\tilde x}{\lambda_5(t)}\right)\phi\left(\frac{\tilde x}{\lambda_5^q(t)}\right)\phi'\left(\frac{\tilde y}{\lambda_6(t)}\right)\,\mathrm{d}x\mathrm{d}y \right|\\
&~{} \lesssim \frac{|\rho_2'(t)| \lambda_5^{q/2}(t) \|u_0\|_{L^2_{x,y}}}{\lambda_6^{1/2}(t)\eta_3(t)} \lesssim \frac{|\ell_2m_2| \|u_0\|_{L^2_{xy}}}{ t^{2-\frac12b(q+2-r)-m_2}\log^{1- \frac12(r-q)} t},
\end{aligned}
\]
which integrates in time because of \eqref{cond_rho1rho2}.

\medskip

This last estimate ends the study of the term $\mathcal I _{1,1}$ in \eqref{separation_Xi}, concluding that 
\begin{equation}\label{I11_final}
\mathcal I _{1,1} (t)=  \frac{1}{2\eta_3(t)\lambda_5(t)}\int_{\mathbb{R}^{2}}  u^{2}\psi'\left(\frac{\tilde x}{\lambda_5(t)}\right)\phi\left(\frac{\tilde x}{\lambda_5^q(t)}\right)\phi\left(\frac{\tilde y}{\lambda_6(t)}\right)\,\mathrm{d}x\mathrm{d}y +\mathcal I_{1,1,int}(t),
\end{equation}
with $\mathcal I_{1,1,int}(t)\in L^1(\{t\gg1\})$.

\medskip

Now, we focus our attention in the remaining terms in \eqref{separation_Xi}. First, by means of Young's inequality,  we have for   $\epsilon>0,$
\begin{equation*}
\begin{split}
\left| \mathcal I _{1,2}(t)  \right| &= \left| \frac{\lambda_{5}'(t)}{\lambda_5(t)\eta_3(t)}\int_{\mathbb{R}^{2}}u \psi'\left(\frac{\tilde x}{\lambda_5(t)}\right)\left(\frac{\tilde x}{\lambda_5(t)}\right)\phi\left(\frac{\tilde x}{\lambda_5^q(t)}\right)\phi\left(\frac{\tilde y}{\lambda_6(t)}\right)\,\mathrm{d}x\mathrm{d}y \right| \\
&\leq \frac{1}{4\epsilon}\left|\frac{\lambda_{5}'(t)}{\lambda_5(t)\eta_3(t)}\right|\int_{\mathbb{R}^{2}}u^{2}\psi'\left(\frac{\tilde x}{\lambda_5(t)}\right)
\phi\left(\frac{\tilde x}{\lambda_5^q(t)}\right)\phi\left(\frac{\tilde y}{\lambda_6(t)}\right)\,\mathrm{d}x\mathrm{d}y\\
&\quad +\epsilon\left|\frac{\lambda_{5}'(t)}{\lambda_5(t)\eta_3(t)}\right|\int_{\mathbb{R}^{2}}
\psi'\left(\frac{\tilde x}{\lambda_5(t)}\right)\left(\frac{\tilde x}{\lambda_5(t)}\right)^{2}\phi\left(\frac{\tilde x}{\lambda_5^q(t)}\right)\phi\left(\frac{\tilde y}{\lambda_6(t)}\right)\,\mathrm{d}x\mathrm{d}y\\
&=\frac{1}{4\epsilon}\left|\frac{\lambda_{5}'(t)}{\lambda_5(t)\eta_3(t)}\right|\int_{\mathbb{R}^{2}}u^{2}\psi'\left(\frac{\tilde x}{\lambda_5(t)}\right)
\phi\left(\frac{\tilde x}{\lambda_5^q(t)}\right)\phi\left(\frac{\tilde y}{\lambda_6(t)}\right)\,\mathrm{d}x\mathrm{d}y\\
&\quad+\epsilon\left|\frac{\lambda_{5}'(t)\lambda_6(t)}{\eta_3(t)}\right|\|(\cdot)^2\psi^\prime_{\sigma}\|_{L^1_x}\left\|\phi(\cdot)\right\|_{L_x^\infty}\|\phi\|_{L^{1}_{y}},\\
\end{split}
\end{equation*}
so that, taking $\epsilon=\lambda'_{5}(t)>0$ for $t\gg1;$ it is clear that
\begin{equation*}
\begin{split}
\left|\mathcal I _{1,2}(t)\right|&\le \frac{1}{4\lambda_5(t)\eta_3(t)}\int_{\mathbb{R}^{2}}u^{2}\psi'\left(\frac{\tilde x}{\lambda_5(t)}\right)
\phi\left(\frac{\tilde x}{\lambda_5^q(t)}\right)\phi\left(\frac{\tilde y}{\lambda_6(t)}\right)\,\mathrm{d}x\mathrm{d}y\\
&\quad +c \frac{\left(\lambda_{5}'(t)\right)^{2}}{\lambda_5^2(t)}\frac{\lambda_6(t)}{\eta_3(t)(\lambda_5(t))^{-2}}\\
&=:\frac{1}{2}\mathcal I _{1,1,3}(t)+\mathcal I ^*_{1,2}(t).
\end{split}
\end{equation*}
Note that the first term in the r.h.s. is the quantity to be estimated (see \eqref{Xi113}), unlike the remaining term $\mathcal I ^*_{1,2}$ which
 satisfies
\begin{equation*}\label{12}
\begin{aligned}
0\leq &~{} \mathcal I ^*_{1,2}(t)\le\frac{\left(\lambda_{5}'(t)\right)^{2}}{\lambda_5^2(t)}\frac{\lambda_6(t)}{\eta_3(t)(\lambda_5(t))^{-2}}\lesssim \frac{1}{t^2}\frac{1}{\eta_3(t)(\lambda_5(t))^{-2-r}}\\
= &~{} \frac{1}{t^{3 -(3+r)b}\log^{3+r}t}.
\end{aligned}
\end{equation*}
The  term $\mathcal I ^*_{1,2}$ belongs in $L^1(t \gg 1)$, since \eqref{bqconditions_new} implies that $b\le\frac{2}{2+q+r}<\frac{2}{r+3}$. This ends the estimate of $\mathcal I _{1,2}(t)$, concluding that 
\begin{equation}\label{I12_final}
\left|\mathcal I _{1,2}(t)\right| \leq \frac{1}{2}\mathcal I _{1,1,3}(t) +\mathcal I_{1,2,int}(t),
\end{equation}
with $\mathcal I_{1,2,int}(t) \in L^1(\{t\gg1\})$.

\medskip

Now, we consider the term $\mathcal I _{1,3}(t)$. Combining the properties attribute to $\phi$ and Young's inequality we get for $\theta(t)=t^p/\log t$,
\[
\begin{split}
\left| \mathcal I _{1,3}(t) \right|   &= q\left|\frac{\lambda_{5}'(t)}{\lambda_5(t)\eta_3(t)}\int_{\mathbb{R}^{2}}u \psi\left(\frac{\tilde x}{\lambda_5(t)}\right)\left(\frac{\tilde x}{\lambda_5^q(t)}\right)\phi'\left(\frac{\tilde x}{\lambda^q_{5}(t)}\right)\phi\left(\frac{\tilde y}{\lambda_6(t)}\right)\,\mathrm{d}x\mathrm{d}y \right| \\
&\leq q\left|\frac{\lambda_{5}'(t)}{\lambda_5(t)\eta_3(t)}\right|\int_{\mathbb{R}^{2}}|u| \left| \psi\left(\frac{\tilde x}{\lambda_5(t)}\right)\right|\left|\frac{\tilde x}{\lambda_5^q(t)}\right|\left|\phi^\prime\left(\frac{\tilde x}{\lambda_5^q(t)}\right)\right|\phi\left(\frac{\tilde y}{\lambda_6(t)}\right)\,\mathrm{d}x\mathrm{d}y\\
&\leq \frac{q}{2} \left|\frac{\lambda_{5}'(t) \theta(t)}{\lambda_5(t)\eta_3(t)}\right|\int_{\mathbb{R}^{2}}u^{2} \psi^{2}\left(\frac{\tilde x}{\lambda_5(t)}\right) \,\mathrm{d}x\mathrm{d}y\\
&\quad +\frac{q}{2}\left|\frac{\lambda_{5}'(t)}{\lambda_5(t)\eta_3(t)\theta(t)}\right|\int_{\mathbb{R}^{2}}\left(\frac{\tilde x}{\lambda_5^q(t)}\right)^{2}\left|\phi^\prime\left(\frac{\tilde x}{\lambda_5^q(t)}\right)\right|^2 \phi^2\left(\frac{\tilde y}{\lambda_6(t)}\right)\mathrm{d}x\mathrm{d}y\\
&\leq \frac{q}{2}
\left|\frac{\lambda_{5}'(t)\theta(t)}{\lambda_5(t)\eta_3(t)}\right|\|u_{0}\|_{L^{2}_{xy}}^{2} \left\|\psi\right\|_{L_x^\infty}^2 +\frac{q}{2}\left|\frac{\lambda_{5}'(t)\lambda_5^{q}\lambda_6(t)}{\lambda_5(t)\eta_3(t)\theta(t)}\right|\left\|(\cdot) \phi^{\prime}(\cdot)\right\|_{L^{2}_{x}}^2\|\phi\|_{L^{2}_{y}}^2.
\end{split}
\]
Hence,
\begin{equation}\label{I13_final}
\begin{split}
\left| \mathcal I _{1,3}(t) \right|   &\lesssim \frac{1}{t\log^2(t)}+\frac{\log t}{t^{1+p}}\frac{1}{\eta_3(t)\lambda_5^{-q-r}(t)}\\
&=\frac{1}{t\log^2(t)}+\frac{\log t}{t^{2-b}}\frac{1}{\eta_3(t)\lambda_5^{-q-r}(t)}=\frac{1}{t\log^2 t}+\frac{1}{t^{3-b(2+q+r)}\log^{q+r} t}.
\end{split}
\end{equation}
 By \eqref{bqconditions_new} we have $b\leq\frac{2}{2+q+r}$, consequently  $3-b(2+q+r)\ge 1$. Thus, $\mathcal I _{1,3}\in L^1(t\gg1)$. 
 
 \medskip
 
 As before, we get for $\theta(t)=t^p/\log t$,
\begin{equation}\label{I14_final}
\begin{aligned}
  \left|\mathcal I _{1,4}(t)\right|&\le\left|\frac{\lambda_{6}^{\prime}(t)}{\lambda_6(t) \eta_3(t)}\right| \int_{\mathbb{R}^{2}} |u| \left| \psi\left(\frac{\tilde x}{\lambda_5(t)}\right) \right| \phi\left(\frac{\tilde x}{\lambda_5^q(t)}\right)\left|\frac{\tilde y}{\lambda_6(t)}\right| \left|\phi^{\prime}\left(\frac{\tilde y}{\lambda_6(t)}\right)\right| \mathrm{d} x \mathrm{d} y\\
  &\le\left|\frac{\lambda_{6}^{\prime}(t)\theta(t)}{2\lambda_6(t) \eta_3(t)}\right| \int_{\mathbb{R}^{2}} u^2 \psi^2\left(\frac{\tilde x}{\lambda_5(t)}\right)\mathrm{d} x \mathrm{d} y\\
  &\quad+\left|\frac{\lambda_{6}^{\prime}(t)}{2\lambda_6(t) \eta_3(t)\theta(t)}\right| \int_{\mathbb{R}^{2}}  \phi^2\left(\frac{\tilde x}{\lambda_5^q(t)}\right) \left|\frac{\tilde y}{\lambda_6(t)}\right|^2\left|\phi^{\prime}\left(\frac{\tilde y}{\lambda_6(t)}\right)\right|^2 \mathrm{d} x \mathrm{d} y\\
  &\le \left|\frac{\lambda_{6}^{\prime}(t)\theta(t)}{2\lambda_6(t) \eta_3(t)}\right|\|u_0\|^2_{L^2_{xy}} \|\psi\|_\infty^2 +\left|\frac{\lambda_{6}^{\prime}(t)\lambda_5^q(t) }{2 \eta_3(t)\theta(t)}\right| \|\phi\|_{L^2_x}^2 \|(\cdot)\phi^\prime(\cdot)\|_{L^2_y}^2\\
&\lesssim\frac{1}{t\log^{2} t}+
\frac{1}{t^{3-b(2+q+r)}\log^{q+r} t}.
\end{aligned}
\end{equation}
Just like in \eqref{I13_final} we obtain that $\mathcal I _{1,4}\in L^1(t\gg 1)$. We conclude from \eqref{separation_Xi}, \eqref{I11_final}, \eqref{I12_final}, \eqref{I13_final} and \eqref{I14_final},
\begin{equation}\label{I1_final}
\mathcal I_1(t) \geq \frac{1}{4\eta_3(t)\lambda_5(t)}\int_{\mathbb{R}^{2}}  u^{2}\psi'\left(\frac{\tilde x}{\lambda_5(t)}\right)\phi\left(\frac{\tilde x}{\lambda_5^q(t)}\right)\phi\left(\frac{\tilde y}{\lambda_6(t)}\right)\,\mathrm{d}x\mathrm{d}y -\mathcal I_{1,int}(t),
\end{equation}
with $\mathcal I_{1,int}(t)\in L^1(\{t\gg1\})$. Finally, combining \eqref{separation}, \eqref{2} and \eqref{I1_final}, we have
\begin{equation*}
\begin{split}
 \frac{1}{4\eta_3(t)\lambda_5(t)}\int_{\mathbb{R}^{2}}  u^{2}\psi'\left(\frac{\tilde x}{\lambda_5(t)}\right)\phi\left(\frac{\tilde x}{\lambda_5^q(t)}\right)\phi\left(\frac{\tilde y}{\lambda_6(t)}\right)\,\mathrm{d}x\mathrm{d}y \leq \frac{d\mathcal I }{dt}(t) + \mathcal I _{int}(t),
\end{split}
\end{equation*}
with $I_{int}(t)\in L^1(\{t\gg1\})$.
In consequence, we conclude \eqref{dT_ppal} by noticing that for some $\sigma_2>0$ fixed, independent of time, we have $\eta_3(t)\lambda_5(t)=t$, $q>1$ and
\[
\begin{aligned}
&~{} \frac{\sigma_2}{t}\int_{\mathbb{R}^{2}} u^{2} \phi\left(\frac{\tilde x}{\lambda_5(t)}\right)\phi\left(\frac{\tilde y}{\lambda_6(t)}\right)\,\mathrm{d}x\mathrm{d}y \\
&~{} \leq \frac{1}{4\eta_3(t)\lambda_5(t)}\int_{\mathbb{R}^{2}}  u^{2}\psi'\left(\frac{\tilde x}{\lambda_5(t)}\right)\phi\left(\frac{\tilde x}{\lambda_5^q(t)}\right)\phi\left(\frac{\tilde y}{\lambda_6(t)}\right)\,\mathrm{d}x\mathrm{d}y.
\end{aligned}
\]
This ends the proof of Lemma \ref{le:dT}.

\section{Proof of Theorem \ref{thm_KPI}}\label{Final}

This last section is devoted to connect Propositions \ref{paso1} and \ref{paso2} and Theorem \ref{KP_decay} to obtain the proof of Theorem \ref{thm_KPI}.

\subsection{End of proof of Theorem \ref{thm_KPI}, part \eqref{decay_v}} Recall Proposition \ref{paso2} and Proposition \ref{le:dT}. We have
\[
\begin{aligned}
&~{} \frac{\sigma_1}{t\log t}\int_{\mathbb{R}^{2}} v^{2}\left(\partial_{y}\Psi\right)\left(\frac{\tilde x}{\lambda_{3}(t)},\frac{\tilde y}{\lambda_{4}(t)}
\right)\mathrm{d}x\mathrm{d}y \\
&~{} \qquad \leq   -\frac{d}{dt}\mathcal J(t)  +  \frac{C_1}{t\log t}\int_{\mathbb{R}^{2}} |u|^{3} \phi\left(\frac{\tilde x}{\lambda_{3}(t)} \right) \phi \left(\frac{\tilde y}{\lambda_{4}(t)}\right) \mathrm{d}x\mathrm{d}y\\
&~{} \qquad \quad + \frac{C_1}{t} \int_{\mathbb{R}^{2}}u^2 \phi\left(\frac{\tilde x}{\lambda_{3}(t)} \right) \phi \left(\frac{\tilde y}{\lambda_{4}(t)}\right)\mathrm{d}x\mathrm{d}y  + \mathcal J_{int}(t),
\end{aligned}
\]
where $\mathcal J_{int}(t)\in L^1(\{ t\gg1\})$. 

\medskip

Note that the last term above is integrable in time, thanks to Lemma \ref{le:dT}. Indeed, using the fact that from \eqref{def_eta2} one has
\[
\lambda_3(t) \ll \lambda_5(t), \quad \lambda_4(t) \ll \lambda_6(t),
\]
we obtain using  \eqref{dT_ppal},
\[
\begin{aligned}
&~{}\frac1{t} \int_{\mathbb{R}^{2}}u^2 \phi\left(\frac{\tilde x}{\lambda_{3}(t)} \right) \phi \left(\frac{\tilde y}{\lambda_{4}(t)}\right)\mathrm{d}x\mathrm{d}y  \\
&~{} \quad \lesssim \frac{1}{t}\int_{\mathbb{R}^{2}} u^{2} \phi\left(\frac{\tilde x}{\lambda_5(t)}\right)\phi\left(\frac{\tilde y}{\lambda_6(t)}\right)\,\mathrm{d}x\mathrm{d}y \in L^1(\{t\gg1\}).
\end{aligned}
\]
The final conclusion in \eqref{estimate_v} follows from the following result. Recall \eqref{derivada}.

\begin{lem}[Control of cubic terms]\label{cubico}
	Let $u\in E^1(\mathbb R^2)$ be   a  solution of the KP-I equation verifying \eqref{bounded1}, then 
\begin{equation}\label{estimate_u3}
\frac{1}{\eta_2(t)\lambda_{4}(t)}\int_{\mathbb{R}^{2}}|u|^{3}(x,y,t)\left(\partial_{y}\Psi\right)\left(\frac{\tilde x}{\lambda_{3}(t)},\frac{\tilde y}{\lambda_{4}(t)}\right)\mathrm{d}x\mathrm{d}y \in L^1(\{t\gg 1\}).
\end{equation}
\end{lem}

\begin{proof}
First, we apply Young's inequality properly,  from  where  we obtain  
\begin{equation*}
	\begin{split}
		&\frac{1}{\eta_{2}(t)\lambda_{4}(t)}\int_{\mathbb{R}^{2}}|u|^{3}\left(\partial_{y}\Psi\right)\left(\frac{\tilde x}{\lambda_{3}(t)},\frac{\tilde y}{\lambda_{4}(t)}\right)\mathrm{d}x\mathrm{d}y\\
		&\lesssim \frac{1}{t^{4}\log^{3}t}\int_{\mathbb{R}^{2}}u^{2}\left(\partial_{y}\Psi\right)\left(\frac{\tilde x}{\lambda_{3}(t)},\frac{\tilde y}{\lambda_{4}(t)}\right)\mathrm{d}x\mathrm{d}y\\
		&\quad+\frac{1}{t^{4/3}\log^{5/3}t}\int_{\mathbb{R}^{2}}u^{6}\left(\partial_{y}\Psi\right)\left(\frac{\tilde x}{\lambda_{3}(t)},\frac{\tilde y}{\lambda_{4}(t)}\right)\mathrm{d}x\mathrm{d}y.
	\end{split}
\end{equation*}

Next,   combining  H\"{o}lder's inequality and \eqref{bounded1}, we  find that 
\begin{equation}\label{e1}
	\begin{split}
		&\frac{1}{\eta_{2}(t)\lambda_{4}(t)}\int_{\mathbb{R}^{2}}|u|^{3}\left(\partial_{y}\Psi\right)\left(\frac{\tilde x}{\lambda_{3}(t)},\frac{\tilde y}{\lambda_{4}(t)}\right)\mathrm{d}x\mathrm{d}y\\
		&\lesssim \frac{1}{t}\int_{\mathbb{R}^{2}}u^{2}\left(\partial_{y}\Psi\right)\left(\frac{\tilde x}{\lambda_{3}(t)},\frac{\tilde y}{\lambda_{4}(t)}\right)\mathrm{d}x\mathrm{d}y+\frac{1}{t\log^{4}t}\|u\|_{L^{6}_{xy}}^{6}\left\|(\partial_{y}\Psi)\left(\frac{\cdot}{\lambda_{3}(t)},\frac{\cdot}{\lambda_{4}(t)}\right)\right\|_{L^{\infty}_{xy}}\\
		&\lesssim \frac{1}{t}\int_{\mathbb{R}^{2}}u^{2}\left(\partial_{y}\Psi\right)\left(\frac{\tilde x}{\lambda_{3}(t)},\frac{\tilde y}{\lambda_{4}(t)}\right)\mathrm{d}x\mathrm{d}y+\frac{C}{t\log^{4}t}.
	\end{split}
\end{equation} 
We use now  \eqref{psi_phi}, \eqref{producto}, \eqref{def_phi2}, and the fact that from \eqref{def_eta2} we  have 
\begin{equation*}
\lambda_{3}(t) \ll \lambda_5(t), \quad \lambda_{4}(t) \ll \lambda_6(t).
\end{equation*}
Thus,  in virtue of    \eqref{dT_ppal} we obtain  
\begin{equation*}
\begin{split}
	\frac{1}{t} \int_{\mathbb{R}^{2}}u^2 \left(\partial_{y}\Psi\right)\left(\frac{\tilde x}{\lambda_{3}(t)},\frac{\tilde y}{\lambda_{4}(t)}\right)\mathrm{d}x\mathrm{d}y&\sim \frac{1}{t} \int_{\mathbb{R}^{2}}u^2 \phi\left(\frac{\tilde x}{\lambda_{3}(t)} \right) \phi \left(\frac{\tilde y}{\lambda_{4}(t)}\right)\mathrm{d}x\mathrm{d}y \\
	 &\lesssim \frac{1}{t}\int_{\mathbb{R}^{2}} u^{2} \phi\left(\frac{\tilde x}{\lambda_5(t)}\right)\phi\left(\frac{\tilde y}{\lambda_6(t)}\right)\mathrm{d}x\mathrm{d}y \in\! L^1(\{t\gg1\}).
\end{split}
\end{equation*}
Thus, coming  back to \eqref{e1} we  obtain \eqref{estimate_u3}.
\end{proof}
Now we conclude. By taking $b,r$ slightly smaller but fixed if necessary, and using \eqref{def_eta2}, we obtain
\begin{equation}\label{estimate_v_new}
\int_{\{t\gg1\}}\frac{1}{t\log t}\int_{\widetilde \Omega_1(t)} v^{2}\mathrm{d}x\mathrm{d}y\mathrm{d}t \lesssim \int_{\{t\gg1\}} \frac{1}{t\log t}\int_{\mathbb{R}^{2}} v^{2}\left(\partial_{y}\Psi\right)\left(\frac{\tilde x}{\lambda_{3}(t)},\frac{\tilde y}{\lambda_{4}(t)}
\right)\mathrm{d}x\mathrm{d}y\mathrm{d}t  <\infty.
\end{equation}
Here, $\widetilde \Omega_1(t)$ is the subset of the plane $\mathbb R^2$ introduced in \eqref{Omega1}. This estimate and \cite{MuPo2} readily implies \eqref{decay_v}.

\subsection{End of proof of Theorem \ref{thm_KPI}, part \eqref{decay_ux}}\label{finaliza} Recall Proposition \ref{paso1}. From \eqref{estimate_v_new} and $v=\partial_x^{-1}\partial_y u$,
\[
 \frac{1}{t\log t}\int_{\mathbb{R}^{2}} (\partial_x^{-1}\partial_y u)^2   \phi \left(\frac{\tilde x}{\lambda_1(t)} \right) \phi\left(\frac{\tilde y}{\lambda_2(t)}\right)   \, \mathrm{d}x\mathrm{d}y \in L^1(\{t\gg 1\}).
\]
Here we have used that $\lambda_1(t)\ll \lambda_3(t)$ and $\lambda_2(t)\ll \lambda_4(t)$, see \eqref{comparacion_final}. On the other hand, from \eqref{derivada}, \eqref{comparacion_final}, \eqref{producto} and  \eqref{estimate_u3}, 
\[
\begin{aligned}
&~{} \frac{1}{t\log t} \int_{\mathbb{R}^{2}} |u|^3 \phi \left(\frac{\tilde x}{\lambda_1(t)} \right) \phi\left(\frac{\tilde y}{\lambda_2(t)}\right) \, \mathrm{d}x\mathrm{d}y \\
&~{} \quad \lesssim \frac{1}{\eta_2(t)\lambda_{4}(t)}\int_{\mathbb{R}^{2}}|u|^{3}(x,y,t)\left(\partial_{y}\Psi\right)\left(\frac{\tilde x}{\lambda_{3}(t)},\frac{\tilde y}{\lambda_{4}(t)}\right)\mathrm{d}x\mathrm{d}y \in L^1(\{t\gg 1\}).
\end{aligned}
\]
We conclude
\[
\int_{\{t\gg1\}}\frac{1}{t\log t}\int_{\mathbb{R}^{2}} (\partial_x u)^2 \phi \left(\frac{\tilde x}{\lambda_1(t)} \right) \phi\left(\frac{\tilde y}{\lambda_2(t)}\right)   \, \mathrm{d}x\mathrm{d}y\mathrm{d}t <+\infty.
\]
Using \eqref{Omega2},
\[
\liminf_{t\to \infty}\int_{\widetilde \Omega_2(t)} (\partial_x u)^2  \, \mathrm{d}x\mathrm{d}y \leq  \liminf_{t\to \infty}\int_{\mathbb{R}^{2}} (\partial_x u)^2 \phi \left(\frac{\tilde x}{\lambda_1(t)} \right) \phi\left(\frac{\tilde y}{\lambda_2(t)}\right)   \, \mathrm{d}x\mathrm{d}y =0.
\]
This ends the proof.

\subsection{Proof of Corollary \ref{Non}} This corollary is a direct consequence of Proposition \ref{le:dT} and the fact that $u$ is periodic in time. Indeed, assume $u$ periodic in time nontrivial and let $T>0$ be the minimal period of $u$. Fix $\varepsilon >0$ and consider $R(\varepsilon)>0$ from Definition \ref{def_Non}. Thanks to \eqref{adentro} and the definition of $\Omega_1(t)$, one has for $b$ sufficiently small in \eqref{indices},
\[
\begin{aligned}
\frac{\sigma_2}{t}\int_{\|(x-x(t),y-y(t))\|\leq R(\varepsilon)} u^{2} (x,y,t)\,\mathrm{d}x\mathrm{d}y \leq &~{} \frac{\sigma_2}{t}\int_{\mathbb{R}^{2}} u^{2} \phi\left(\frac{\tilde x}{\lambda_5(t)}\right)\phi\left(\frac{\tilde y}{\lambda_6(t)}\right)\,\mathrm{d}x\mathrm{d}y \\
\leq &~{} \frac{d\mathcal I }{dt}(t) + \mathcal I _{int}(t),
\end{aligned}
\]
with $q=1+\varepsilon_0$ in $\mathcal I(t)$ \eqref{mI}, and where $\mathcal I _{int}(t)$ are terms that belong to $L^1\left(\{t\gg 1\}\right)$. The last inequalities imply
\[
\int_{t\gg 1} \frac{\sigma_2}{t} \int_{\|(x-x(t),y-y(t))\|\leq R(\varepsilon)} u^{2} (x,y,t)\,\mathrm{d}x\mathrm{d}y dt <+\infty.
\]
However, since $u$ is $T$-time periodic and nontrivial, for each $n\in\mathbb N$, one has
\[
\sup_{t\in [nT,(n+1)T]}\int_{\|(x-x(t),y-y(t))\|\leq R(\varepsilon)} u^{2} (x,y,t)\,\mathrm{d}x\mathrm{d}y \geq \eta_0>0,
\]
revealing that 
\[
\int_{t\gg 1} \frac{\sigma_2}{t} \int_{\|(x-x(t),y-y(t))\|\leq R(\varepsilon)} u^{2} (x,y,t)\,\mathrm{d}x\mathrm{d}y dt =+\infty,
\]
a contradiction. 

\begin{rem}\label{theend}
Although not stated in Corollary \ref{Non}, the periodicity of $u$ is just a sufficient condition for the nonexistence of compact solutions. The proof actually works in less restrictive settings, for instance provided there exists $\eta_0>0$ and an increasing sequence of times $T_n\to +\infty$ such that 
\[
\sup_{t\in [T_n,T_{n+1}]}\int_{\|(x-x(t),y-y(t))\|\leq R(\varepsilon)} u^{2} (x,y,t)\,\mathrm{d}x\mathrm{d}y \geq \frac{\eta_0}{\log T_n},
\]
as it is easily checked.
\end{rem}
%
%
%
%

\end{document}